\documentclass[11pt]{amsart}
\usepackage{amsmath,amsfonts,amsbsy,amsgen,amscd,mathrsfs,amssymb,amsthm}
\usepackage{enumerate,mathtools}
\usepackage{microtype}
\usepackage{tabularx}
\usepackage{array}
\newcolumntype{P}[1]{>{\centering\arraybackslash}p{#1}}
\usepackage{makecell}

\usepackage{stmaryrd}
\usepackage{fullpage}
\usepackage{url}
\usepackage{mathrsfs}
\usepackage{float}
\usepackage{bbm}
\usepackage{quiver}
\usepackage[colorlinks=true,allcolors=blue]{hyperref}

\usepackage[
backend=biber,
style=alphabetic,
maxalphanames=5,
maxnames=5
]{biblatex}
\makeatletter
\renewcommand*{\eqref}[1]{%
  \hyperref[{#1}]{\textup{\tagform@{\ref*{#1}}}}%
}
\makeatother

\usepackage{tikz}
\usetikzlibrary{shadings}

\numberwithin{equation}{section}

\newtheorem{theorem}{Theorem}[section]
\newtheorem{proposition}[theorem]{Proposition}

\newtheorem{conjecture}[theorem]{Conjecture}
\newtheorem{corollary}[theorem]{Corollary}
\newtheorem{lemma}[theorem]{Lemma}
\theoremstyle{definition}
\newtheorem{definition}[theorem]{Definition}

\newtheorem{example}[theorem]{Example}

\newtheorem*{claim*}{Claim}

\newcommand{\A}{\mathbb{A}}
\newcommand{\C}{\mathbb{C}}

\newcommand{\F}{\mathbb{F}}
\newcommand{\G}{\mathbb{G}}

\newcommand{\N}{\mathbb{N}}

\newcommand{\Q}{\mathbb{Q}}

\newcommand{\Z}{\mathbb{Z}}

\DeclareMathOperator{\Conf}{Conf}

\DeclareMathOperator{\disc}{disc}

\DeclareMathOperator{\Frob}{Frob}

\DeclareMathOperator{\Gr}{Gr}

\DeclareMathOperator{\Hur}{Hur}

\DeclareMathOperator{\id}{id}

\DeclareMathOperator{\Ind}{Ind}

\DeclareMathOperator{\inv}{inv}

\DeclareMathOperator{\ord}{ord}

\DeclareMathOperator{\PConf}{PConf}

\DeclareMathOperator{\res}{res}

\DeclareMathOperator{\Spec}{Spec}

\DeclareMathOperator{\Tor}{Tor}
\DeclareMathOperator{\tr}{tr}

\newcommand{\SL}{\operatorname{SL}}

\usepackage[shortlabels]{enumitem}

\usepackage[OT2,T1]{fontenc}
\DeclareSymbolFont{cyrletters}{OT2}{wncyr}{m}{n}
\DeclareMathSymbol{\Sha}{\mathalpha}{cyrletters}{"58}
\newcommand{\Mod}[1]{\ (\mathrm{mod}\ #1)}

\usepackage{array, booktabs}

\title{Optimal homological vanishing: cancellation of character sums and Patterson's conjecture over $\F_q[t]$}

\author{Zhao Yu Ma}
\address{Princeton University, Princeton, NJ, USA}
\email{zm5336@princeton.edu}
\begin{document}
\begin{abstract}
Many arithmetic sums over function fields can be expressed in terms of $H_i(B_n, V^{\otimes n})$ for some braided vector space $V$, and a vanishing line for these homology groups gives power-savings cancellation for the arithmetic sum. We prove an explicit vanishing line for $H_i(B_n,V^{\otimes n})$ depending only on the homology up to some finite $n$. Moreover, as the range of $n$ increases, the slope of the resulting vanishing line converges to the optimal slope. We also apply our methods to two different families of arithmetic sums. Firstly, we prove an upper bound for the bias of higher order Gauss sums over function fields, extending Patterson's conjecture beyond the cubic and quartic cases over number fields, and we conjecture this bound is sharp for orders that are prime powers. Secondly, we show that over Galois $G$-extensions, almost all character sums exhibit near square-root cancellation.
\end{abstract}

\maketitle

\hypersetup{linktocpage=true}
\tableofcontents

%%Specific Tikz setup
\newcommand{\drawnaxis}[1]{
% max n
  \draw[->] (0,0) -- (#1+0.5,0) node[right] {$n$};
  \foreach \x in {1,...,#1}
    \node[below] at (\x-0.5,0) {\x};
}

\newcommand{\drawiaxis}[1]{
% max i
  \draw[->] (0,0) -- (0,#1+1.5) node[above] {$i$};
  \foreach \y in {1,...,\numexpr#1+1\relax}
    \pgfmathtruncatemacro{\yy}{\y-1}
    \node[left] at (0,\y-0.5) {\yy};
}

\newcommand{\drawblock}[6]{%
  % Xstart Xend Ystart Y end filled with Color, Entry
  \foreach \n in {\numexpr#1-1\relax,...,\numexpr#2-1\relax}{
    \foreach \i in {#3,...,#4}{
      \fill[#5] (\n,\i) rectangle (\n+1,\i+1);
      \node at (\n+0.5,\i+0.5) {$#6$};
    }
  }
}

\newcommand{\drawstaircase}[5]{
  % Xstart Xend Ystart Color Entry
  \foreach \n in {\numexpr #1-1\relax,...,\numexpr#2-1\relax}{
    \foreach \i in {#3,...,\numexpr #3+\n+1-#1\relax}{
      \fill[#4] (\n,\i) rectangle (\n+1,\i+1);
      \node at (\n+0.5,\i +0.5) {$#5$};
    }
  }
}

\newcommand{\cell}[3]{
    % Xcell, Ycell, Entry
    \fill[black!10] (#1,#2) rectangle (#1-1,#2+1);
    \node at (#1-0.5,#2+0.5) {$#3$};
  }
  
\newcommand{\smallcell}[3]{
    % Xcell, Ycell, Entry
    \fill[black!10] (#1,#2) rectangle (#1-1,#2+1);
    \node at (#1-0.5,#2+0.5) {\fontsize{9}{11}\selectfont $#3$};
  }

\section{Introduction}
In \cite{EVW16}, Ellenberg-Venkatesh-Westerland pioneered the geometric approach to arithmetic statistics over function fields via the cohomology of Hurwitz spaces. By using the Grothendieck-Lefschetz trace formula and étale to complex comparison theorems, they reduce the Cohen-Lenstra heuristics to proving homological stability of the corresponding Hurwitz space. Since then, there has been much work on this geometric approach for Hurwitz spaces, culminating in \cite{landesman2025homologicalstabilityhurwitzspaces,landesman2025cohenlenstramomentsfunctionfields,landesman2025stablehomologyhurwitzmodules} which prove many arithmetic statistics conjectures in different settings. 

Instead of looking just at Hurwitz spaces, one can study the more general case of local systems on configuration space, which leads to a larger variety of arithmetic applications like the asymptotics of certain arithmetic sums, as was carried out in \cite{ETW17,bergstrom2023hyperellipticcurvesscanningmap,ES26}. Here, we will only consider the configuration space of the plane $\Conf^n(\C)$, which corresponds to arithmetic sums over the rational function field $\F_q(t)$. Since $\Conf^n(\C)$ is the Eilenberg–MacLane space $K(B_n,1)$ where $B_n$ is the Artin braid group, local systems on $\Conf^n(\C)$ are just representations of $B_n$. Moreover, in most arithmetic situations of interest, the relevant $B_n$-representations are of the form $V^{\otimes n}$ for some braided vector space $V$. It is therefore often enough to understand the homology groups $H_i(B_n, V^{\otimes n})$, and thus we restrict our attention to the homology of braided vector spaces.

Most of the papers mentioned above study the phenomenon of \textit{homological stability}. We say that $H_i(B_n,V^{\otimes n})$ has homological stability with \textit{slope} $s$ if there exists some constant $t$ such that for each homological degree $i$, the homology $H_i(B_n,V^{\otimes n})$ is independent of $n$ once $sn+t\ge i$. In terms of arithmetic applications, homological stability theorems can tell us the primary term of relevant arithmetic sums on the order of $q^n$ together with a power-saving error term on the order of $q^{(1-s/2)n}$. As a special case of this, there is the phenomenon of \textit{homological vanishing}, where $H_i(B_n,V^{\otimes n})$ vanishes with slope $s$ if there is some constant $t$ where $H_i(B_n,V^{\otimes n})$ is zero for $sn+t\ge i$. In this case, one can upper bound the relevant arithmetic sum with a power-saving term of the order of $q^{(1-s/2)n}$. 

In this paper, we study the \textit{optimal slope} of homological vanishing for braided vector spaces. Our results were inspired by \cite[Theorem 1.1.1]{ES26} which proves that if $H_0(B_n,V^{\otimes n})$ vanishes for all $n\ge N$, then $H_i(B_n,V^{\otimes n})$ vanishes with slope $\frac{1}{N+1}$. We will show that this is actually a special case of a more general homological vanishing theorem, Theorem \ref{thm: main extrapolating vanishing}. As a direct consequence of our main theorem, we use this to show that one can obtain vanishing lines with slope arbitrarily close to the true slope in Corollary \ref{cor: optimal slope cor}. We also apply our methods to two different families of arithmetic sums, which we will discuss in Section \ref{sec: intro patterson} and \ref{sec: intro character sums G-extensions}.

We remark that there have been other instances of arithmetic applications of homological vanishing, such as in \cite{Sawin_2021_sqroot,Sawin_2024_sqroot}, where they deduce square-root cancellation of certain sums from the homological vanishing of certain varieties, however, these are not in the context of local systems on configuration space or Hurwitz spaces.

\subsubsection{Main theorem} Our main result is a homological vanishing ``extrapolation'' theorem that, from knowing which homological degrees for which $H_i(B_n,V^{\otimes n})$ vanish for small $n$, allows us to extrapolate this to a certain slope of vanishing for all $n$. We explain this as follows. Let $V$ be a braided vector space over $\C$.
\begin{definition}\label{defn: vanishing staircase}
Let $N\ge 2$ and $I\ge 0$ be integers. We say that $H_i(B_n,V^{\otimes n})$ \textit{vanishes in the $(N,I)$ staircase} if it is zero for all $(n,i)$ satisfying $n\le N$ and $i\le I+n-N$.  
\end{definition}
\begin{theorem}\label{thm: main extrapolating vanishing}
Suppose $H_i(B_n,V^{\otimes n})$ vanishes in an $(N,I)$ staircase. Then, it vanishes with slope $\frac{I+1}{N+1}$. More precisely, $H_i(B_n,V^{\otimes n})=0$ for all $(n,i)$ satisfying
\begin{equation}\label{eqn: i in terms of n}
i\le (I+1)\left\lfloor \frac{n+1}{N+1}\right\rfloor+ \max\left(I+\left\{ \frac{n+1}{N+1}\right\}-N\ ,\ 0\right)-1. 
\end{equation}
Here, we denote $\{\frac a b\}$ to be the remainder when $a$ is divided by $b$.
\end{theorem}

To illustrate this, we consider the case of $(N,I)=(4,1)$ in Figure \ref{fig: (4,1)}. The vanishing of homology in the $(4,1)$ staircase is indicated by the red staircase in the figure. Theorem \ref{thm: main extrapolating vanishing} then says that if the homology in the red staircase vanishes, then the homology in the blue region vanishes as well. Note the blue region is simply the red region translated up by the vector $(N+1,I+1)=(5,2)$ each time, so it vanishes with slope $\frac 2 5$. This tells us that vanishing in an $(N,I)$ staircase implies vanishing in a $(k(N+1)-1,k(I+1)-1)$ staircase for all $k\ge 1$.

\begin{figure}[h]
    \centering
\begin{tikzpicture}[x=1cm,y=1cm,scale=0.5]
\drawnaxis{15}
\drawiaxis{6}
\drawstaircase{1}{7}{0}{black!10}{*}
\drawblock{8}{15}{0}{6}{black!10}{*}
\drawstaircase{3}{4}{0}{red!40}{0}
\drawstaircase{8}{9}{2}{blue!40}{0}
\drawstaircase{13}{14}{4}{blue!40}{0}
\drawblock{5}{15}{0}{1}{blue!40}{0}
\drawblock{10}{15}{2}{3}{blue!40}{0}
\drawblock{15}{15}{4}{5}{blue!40}{0}
\end{tikzpicture}
\caption{Theorem \ref{thm: main extrapolating vanishing} for $(N,I)=(4,1)$.}\label{fig: (4,1)}
\end{figure}
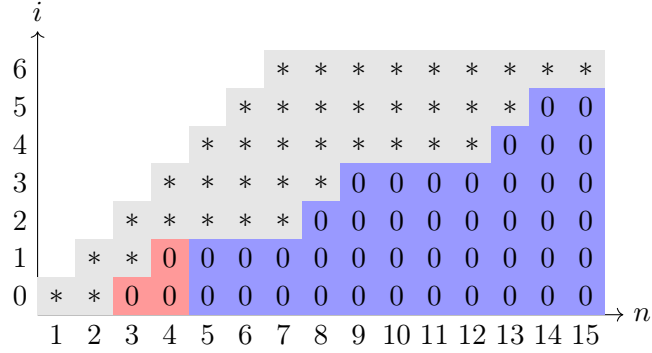
When we take the special case of $I=0$, we get a slight generalization of \cite[Theorem 1.1.1]{ES26}, as we now only require vanishing of $H_0(B_N,V^{\otimes N})$ instead of vanishing of $H_0(B_n,V^{\otimes n})$ for all $n\ge N$.
\begin{corollary}\label{cor: es main theorem}
Let $N\ge 2$ be an integer such that $H_0(B_N,V^{\otimes N})=0$. Then, we have $H_i(B_n,V^{\otimes n})=0$ for $i\le \frac{n+1}{N+1}-1$, so it vanishes with slope $\frac{1}{N+1}$.
\end{corollary}
We illustrate the example of $N=3$ in Figure \ref{fig: (3,0)}. Given that the homology group which is colored red vanishes, the homology in the blue region will also vanish.

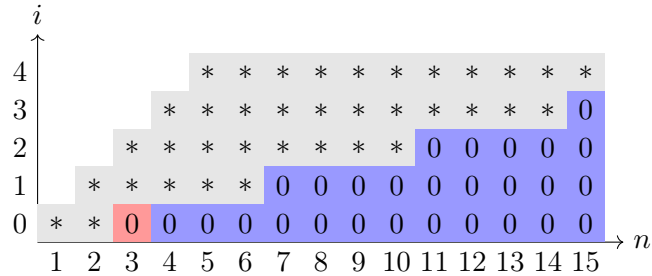
\begin{figure}[h]
    \centering
\begin{tikzpicture}[x=1cm,y=1cm,scale=0.5]
\drawnaxis{15}
\drawiaxis{4}
\drawstaircase{1}{5}{0}{black!10}{*}
\drawblock{6}{15}{0}{4}{black!10}{*}
\drawstaircase{3}{3}{0}{red!40}{0}
\drawstaircase{7}{7}{1}{blue!40}{0}
\drawstaircase{11}{11}{2}{blue!40}{0}
\drawstaircase{15}{15}{3}{blue!40}{0}
\drawblock{4}{15}{0}{0}{blue!40}{0}
\drawblock{8}{15}{1}{1}{blue!40}{0}
\drawblock{12}{15}{2}{2}{blue!40}{0}
\end{tikzpicture}
\caption{Theorem \ref{thm: main extrapolating vanishing} for $(N,I)=(3,0)$/Corollary \ref{cor: es main theorem} for $N=3$.}\label{fig: (3,0)}
\end{figure}
We will prove our main theorem in Section \ref{sec: proof of main theorem}. The key idea is to construct a filtration on the $n$-th bar-complex, where vanishing of each graded pieces can be deduced from vanishing of the $m$-th bar-complex for smaller $m<n$. This filtration will be constructed combinatorially as the terms in the bar-complex can be indexed by partitions.
\subsubsection{Optimal vanishing slope for braided vector spaces}
An important consequence of Theorem \ref{thm: main extrapolating vanishing} is that one can prove a vanishing slope for $H_i(B_n,V^{\otimes n})$ that is arbitrarily close to the optimal slope simply by computing the homology up to some sufficiently large $n$. More precisely, we have the following.
\begin{corollary}\label{cor: optimal slope cor}
Let $V$ be a braided vector space such that $H_i(B_n,V^{\otimes n})$ has homological vanishing of slope $s$. Then, for any $\epsilon>0$, there exist some $N>0$ and $I\ge 0$ such that $H_i(B_n,V^{\otimes n})$ vanishes in an $(N,I)$ staircase and $\frac{I+1}{N+1}>s-\epsilon$. In particular, computing $H_i(B_n,V^{\otimes n})$ and applying Theorem \ref{thm: main extrapolating vanishing} will prove homological vanishing with slope $>s-\epsilon$.
\end{corollary}

This follows immediately from the definition of slope for homological vanishing and Theorem \ref{thm: main extrapolating vanishing}. Furthermore, we observe empirically that it is often the case that computing the cohomology up to some $N$ would give us the exact optimal slope -- for example, this seems to be the case for the bias in Gauss sums for orders $p^k$ and $2p^k$ (see Figure \ref{fig: C_wedge zeta 3,6,4}), and some character sums (see Figure \ref{fig: S_3 trans 6th root}). Unfortunately, the computation time required is exponential in $n$ because the dimension of each term in the bar-complex grows exponentially as $O(|V|^n)$, so in some cases we are not able to get very close to the optimal slope. We will discuss the computation in more detail in Section \ref{sec: on homological computations}.

For arithmetic applications, the significance of the optimal vanishing slope of the associated braided vector space is that it usually determines the optimal power-savings bound for cancellation of the corresponding arithmetic sum. However, occasionally, one may obtain stronger cancellation for two possible reasons: the Frobenius weights may be smaller than expected from homological degree, or there may be additional cancellation among the top-weight Frobenius eigenvalues. It is probably possible to account for the first factor by using the weight filtration on quantum shuffle algebras introduced in \cite{ma2026weightfiltrationhurwitzspaces}, and we believe that analogous statements to Theorem \ref{thm: main extrapolating vanishing} and Corollary \ref{cor: optimal slope cor} should hold with the homological degree $i$ replaced by the true Frobenius weight, and ``slope'' replaced by ``weight slope''. However, more work has to be done, because \cite[Section 6]{ma2026weightfiltrationhurwitzspaces} only proved that the geometric weight filtration over $\C$ and $\F_p$ agrees for Hurwitz spaces, and this crucially uses a normal crossings compactification for Hurwitz space constructed in \cite[Appendix B]{ellenberglandesman}. For the applications in our paper, one would also have to prove a similar comparison theorem for local systems coming from representations of $S_n$ and character sheaves (Example \ref{eg: local system rep S_n} and \ref{eg: local system character sheaves}). We leave this as a possible direction for future work. On the other hand, for the second factor, it is very difficult to account for cancellation in Frobenius eigenvalues, but we believe in the general case that it is unlikely for the eigenvalues completely cancel one another unless there is some special reason for this to occur. 

In Section \ref{sec: individual BVS}, we will attempt to compute slope bounds that approach the optimal slope for some individual examples of braided vector spaces, in line with what Corollary \ref{cor: optimal slope cor} tells us. In some cases, we will also find it helpful to use the vanishing results deduced for families of twists obtained in the previous sections.
\subsubsection{Large cancellation for families of braided vector spaces} While Corollary \ref{cor: optimal slope cor} tells us that one can compute arbitrarily close to the optimal slope for any \textit{individual} braided vector space $V$, we are often interested in \textit{families} of braided vector spaces corresponding to families of arithmetic sums, and one cannot obtain a vanishing statement in families directly from computation. Nevertheless, there are ways to deduce results in families, although the vanishing slope in each case is no longer guaranteed to be optimal. For example, one can apply Corollary \ref{cor: es main theorem} or \cite[Theorem 1.1.1]{ES26} to a family of braided vector spaces by computing the $H_0$ across the family, and this was carried out in \cite{ES26}. However, this can only give a maximum slope of $\frac 1 3$ which is achieved when $N=2$, corresponding to a power-savings of $q^{5/6}$. 

In Section \ref{sec: criteria}, we give two other ways to apply Theorem \ref{thm: main extrapolating vanishing} in families that goes past this $\frac 1 3$ limit and can give homological vanishing of slope arbitrarily close to $1$, which corresponds to near square-root cancellation for arithmetic applications. Firstly, instead of looking at the ``minimal'' staircase of $(N,0)$ as in Corollary \ref{cor: es main theorem}, we look at the ``maximal'' staircase where $(N,N-2)$ (note that because $H_0(B_1,V)=V$ is always nonzero, we cannot have a vanishing staircase of $(N,I)=(N,N-1)$). This would give a vanishing slope of $\frac{N-1}{N+1}$ by Theorem \ref{thm: main extrapolating vanishing}. For this to be useful, we showed that to prove vanishing in an $(N,N-2)$ staircase, it suffices to check that certain shuffle products in the corresponding quantum shuffle algebra are isomorphisms. Secondly, in some cases, the center of the braid group acts via a nontrivial scalar which forces the group cohomology to be zero for some $n$. As a corollary to Theorem \ref{thm: main extrapolating vanishing}, we show that there is homological vanishing with slope at least the density of $n$ where homology vanishes.

Arguably, the most natural examples of families of braided vector spaces satisfying homological vanishing are families obtained by twisting a single braided vector space $V$. We will study the general case in Section \ref{sec: twists of BVS} by applying Theorem \ref{thm: main extrapolating vanishing} and the methods in Section \ref{sec: criteria}. In Section \ref{sec: special cases and extensions}, we will study special cases where we impose additional conditions on $V$ like having finite monodromy, coming from a rack of conjugacy classes, or $V^{\otimes n}$ being an $S_n$-representation. These additional assumptions lead to much stronger results, and they turn out to be relevant to our arithmetic applications later. We also tackle the ``multicolored'' case where there are multiple braided vector spaces. Finally, we will deduce our arithmetic applications from the corresponding homological results in Section \ref{sec: arithmetic applications}.
\subsection{Application to bias in Gauss sums and Patterson's conjecture}\label{sec: intro patterson} In 1846, Kummer \cite{Kummer_1846} studied the distribution of cubic exponential sums of the form
$$S_p=\sum_{x=0}^{p-1} e(x^3/p)$$
where $p\equiv 1\pmod 3$ is a prime and we write $e(x)=e^{2\pi ix}$. This exponential sum is closely related to Gauss sums, which we recall as follows. Given a nontrivial multiplicative character $\chi\colon \F_p^\times \rightarrow \C^\times$, we can attach a Gauss sum
$$G(\chi)=\sum_{x=0}^{p-1}\chi(x)e(x/p).$$
It is known that $|G(\chi)|=\sqrt p$ for all $\chi$, but while its absolute value is known, its argument is difficult to determine in general. Now, let $\chi_3$ be a cubic character, then it follows directly that
\begin{equation}\label{eqn: S_p to Gauss}
S_p=\sum_{x=0}^{p-1}\left(1+\chi_3(x)+\bar\chi_3(x)\right)e(x/p)= G(\chi_3)+G(\bar\chi_3).
\end{equation}
By writing $G(\chi_3)=\sqrt p e^{i\theta_p}$, where $\theta_p\in [-\pi,\pi]$ is determined up to sign depending on the choice of cubic character, we have $$S_p=2\sqrt p \cos(\theta_p).$$

Kummer asked what the distribution of $\theta_p$ is as $p$ varies across primes. He computed the frequency that $\cos(\theta_p)\in [-1,1]$ lies in the intervals $I_1=[-1,-\frac 1 2], I_2 = [-\frac 1 2, \frac 1 2], I_3=[\frac 1 2,1]$ for primes $p\le 500$, and found that it was in a $1\colon 2\colon 3$ ratio, which seems to imply that $\theta_p$ is not equidistributed. We refer the reader to \cite[Chapter 3]{Davenport_2009} for more details.

However, more than a hundred years later, Heath-Brown and Patterson \cite{Heath-Brown_Patterson_1979} showed instead that $\theta_p$ was equidistributed, disproving Kummer's conjecture. Patterson also conjectured the following stronger statement about the order of the bias in $S_p$.
\begin{conjecture}[{\cite{Patterson_1978}}]\label{conj: patterson} We have
$$\sum_{\substack{p\le X\\p\equiv 1\Mod{3}}} \frac{S_p}{2\sqrt p} \asymp \frac{X^{5/6}}{\log X}$$
as $X\rightarrow \infty$.
\end{conjecture}
This conjecture explains why Kummer saw such a large discrepancy between the counts in each interval. The upper bound was proven unconditionally in \cite{Heath-Brown_2000}, and the lower bound (with a smooth cutoff) was proven assuming the Generalized Riemann Hypothesis recently in \cite{Dunn_Radziwill_2024}.

Patterson's conjecture can also be stated in terms of Gauss sums over Eisenstein integers, essentially by using Equation \eqref{eqn: S_p to Gauss}, and this formulation will be useful for our function field analogue later. We sketch this and refer the reader to the exposition in \cite[Section 1.2]{Dunn_Radziwill_2024} for more details. For any rational prime $p\equiv 1\pmod 3$, we can factor $p=\pi\bar\pi$ in $\Z[\zeta_3]$ where we write $\zeta_m=e^{2\pi i/m}$ as the $m$-th root of unity. Furthermore, by modifying by units, there is exactly one factorization where $\pi\equiv \bar \pi \equiv 1 \pmod 3$. In analogy to the cubic character $\chi_3$ on $\F_p^\times$, one can define a cubic symbol $\left(\frac a b\right)_3$ for $a,b\in \Z[\zeta_3]$ that is multiplicative in both variables and takes values in $\{1,\zeta_3,\zeta_3^2\}$. Then, for a prime $\pi\equiv 1\pmod 3$ in $\Z[\zeta_3]$, we define the corresponding Gauss sum as 
\begin{equation}\label{eqn: gauss over eisenstein}
G_3(\pi)=\sum_{a\in \Z[\zeta_3]/\pi}\left(\frac a \pi\right)_3 \psi\left(\frac{a}{\pi}\right)
\end{equation}
where $\psi(z)=e^{2\pi i(z+\bar z)}$, and we can rewrite the conjecture as
\begin{equation}\label{eqn: rewritten patterson}
\sum_{\substack{p\le X\\p\equiv 1\Mod{3}}} \frac{S_p}{2\sqrt p} = \sum_{\substack{N(\pi)<X\\ \pi \text{ prime}\\ \pi\equiv 1\Mod 3}} \frac{G_3(\pi)}{|\pi|}
\asymp \frac{X^{5/6}}{\log X}. 
\end{equation}

One could ask the same question for higher Gauss sums, i.e. with $3$ replaced by another order $o$, like in \cite{Patterson_1978}. The case of $o=2$ has an asymptoptic of $X/\log X$ by the formula for quadratic Gauss sums in \cite{Gauss_1801}. For $o=4$, \cite{david2026quarticgausssumsprimes} conjectured an asymptoptic of $X^{3/4}/\log X$ while proving an upper bound (with a smooth cutoff) on the order of $X^{5/6}/\log X$. As far as we know, there are no conjectures or results for $o>4$.

Now, we turn to the case for function fields. To set this up, let $q=p^e$ be an odd prime power and $\chi\colon \F_q^\times \rightarrow \C^\times$ be a multiplicative character, and denote its order by $\ord(\chi)$. Given a rational function $f\in \F_p(t)$, we let $\res(f)$ be its residue at infinity, i.e. the coefficient of $t^{-1}$ when $f$ is expressed as a formal Laurent series, and define the additive character $\psi(x) = e^{2\pi i \tr^{\F_p}_{\F_q}x/ p}$. The Gauss sum associated to $\chi$ and a polynomial $\pi\in \F_p[t]$ is then
$$G_\chi(\pi)=\sum_{a\in \F_q[t]/\pi} \left(\frac{a}{\pi}\right)_\chi \psi \left(\res\left(\frac a \pi \right)\right)$$
as defined in \cite[Section 2]{Sawin_2024}, in direct analogy to Equation \eqref{eqn: gauss over eisenstein}.

\subsubsection{Results for order $o=p^k,2p^k$}
Our first result is an upper bound for the bias of Gauss sums over function fields in the case where the order $o$ is a prime power or twice of a prime power, whose power-savings match the known results in $o=2,3$ as well as the conjecture for $o=4$ in the number field case. 
\begin{theorem}\label{thm: gauss sums prime powers} Let $p$ be a prime and $k\ge 1$ be an integer. Suppose that $\chi\colon \F_q^\times \rightarrow \C^\times$ is a non-trivial character with order $o=\ord(\chi)$ where $o=p^k$; if $p$ is odd, we also allow $o=2p^k$. Then,
$$\left|\sum_{\substack{\deg(\pi)=n\\ \pi \text{ prime}}} \frac{G_\chi(\pi)}{q^{n/2}} \right|\le \frac{2^{2n-2}}{n} q^{\left(\frac 1 2 +\frac 1 {p^k}\right)(n-1)+1}.$$
\end{theorem}

This arithmetic sum corresponds to twists of the braided vector space $\C_\wedge$ by a root of unity. In our proof, we will use the fact that $\C_\wedge^{\otimes n}$, the tensor power of the untwisted braided vector space, descends to a $S_n$-representation. We also remark that the twisted tensor power $\C_\wedge^{\otimes n}$ factors through a lift $\hat S_n$ of $S_n$, and such representations were also considered in \cite{Miller_Tosteson_2021}.

Note that for an odd prime $p$ and integer $k\ge 1$, the bounds for the sums with order $p^k$ and $2p^k$ given in Theorem \ref{thm: gauss sums prime powers} are the same. This is because these two character sums correspond to two different twists of the braided vector space $\C_\wedge$ that differ by the sign representation, and from this one can deduce that their homology is the same because of the fact that the $B_n$-representations $\C_\wedge^{\otimes n}$ and $\C_\wedge^{\otimes n}\otimes \text{sgn}$ are isomorphic. In general, for an odd integer $o$, homological methods give the same bounds for order $o$ and $2o$.

The key homological lemma behind the proof is to show for the relevant braided vector space $V$ which is $\C_\wedge$ twisted by $o$ that $H_i(B_n,V^{\otimes n})$ vanishes when $o$ does not divide $n(n-1)$. Then, when $o=p^k$ or $o=2p^k$ (for $p$ odd), this tells us that the homology vanishes for $2\le n\le p^k-1$, and applying Theorem \ref{thm: main extrapolating vanishing} for the ``maximal'' staircase will prove a homological vanishing slope that implies Theorem \ref{thm: gauss sums prime powers}.

We conjecture that Theorem \ref{thm: gauss sums prime powers} is sharp, in the sense that the exponent $\frac 1 2 +\frac 1 {p^k}$ in the exponent of $q$ is the best possible. We also believe that the $2^{2n-1}$ term coming from the Betti-number bound can be improved, as most of the homology should be concentrated away from the slope of vanishing. We state our conjecture as follows, and we also include the analogous conjecture in the number field case.
\begin{conjecture}\label{conj: gauss order prime power}
Let $p$ be a prime and $k\ge 1$ be an integer. Suppose that $\chi\colon \F_q^\times \rightarrow \C^\times$ is a non-trivial character with order $o=\ord(\chi)$ where $o=p^k$ or, if $p$ is odd, we also allow $o=2p^k$. Then,
$$\left|\sum_{\substack{\deg(\pi)=n\\ \pi \text{ prime}}} \frac{G_\chi(\pi)}{q^{n/2}} \right| = q^{\left(\frac 1 2 +\frac 1 {p^k} +o(1)\right)n}.$$

Similarly, in the number field case, the bias in $o$-th order Gauss sums (stated analogously to Equation \eqref{eqn: gauss over eisenstein}) where $o=p^k$ or, if $p$ is odd, we also allow $o=2p^k$, is $(1+o(1))X^{\frac 1 2 +\frac 1 {p^k}}/\log X$.
\end{conjecture}

We give two lines of reasoning to support this conjecture, first from a homological vanishing standpoint (for function fields) and then from an analytic perspective.

\begin{figure}[h]
    \centering
\begin{tikzpicture}[x=1cm,y=1cm,scale=0.5]
\drawnaxis{13}
\drawiaxis{12}
\drawstaircase{1}{13}{0}{black!10}{0}
\drawstaircase{2}{2}{0}{red!40}{0}
\drawstaircase{5}{5}{1}{blue!40}{0}
\drawstaircase{8}{8}{2}{blue!40}{0}
\drawstaircase{11}{11}{3}{blue!40}{0}
\drawblock{3}{13}{0}{0}{blue!40}{0}
\drawblock{6}{13}{1}{1}{blue!40}{0}
\drawblock{9}{13}{2}{2}{blue!40}{0}
\drawblock{12}{13}{3}{3}{blue!40}{0}
\foreach \n/\i/\txt in {
    1/0/2,
    3/1/2,
    3/2/2,
    4/1/2,
    4/2/4,
    4/3/2,
    6/2/2,
    6/3/4,
    6/4/8,
    7/2/2,
    7/3/4,
    7/4/10,
    7/5/16,
    7/6/8,
    9/3/2,
    9/4/4,
    9/5/8,
    9/6/22,
    9/7/32,
    9/8/16,
    10/3/2,
    10/4/4,
    10/5/8,
    10/6/22,
    10/7/44,
    10/8/58,
    10/9/30,
    12/4/2,
    12/5/4,
    12/6/8,
    12/7/20,
    12/8/46,
    12/9/98,
    12/11/78,
    13/4/2,
    13/5/4,
    13/6/8,
    13/7/20,
    13/8/46
 }{
    \cell{\n}{\i}{\txt}
  }
\foreach \n/\i/\txt in {
    12/10/144,
    13/9/100,
    13/10/204,
    13/11/276,
    13/12/140
 }{
    \smallcell{\n}{\i}{\txt}
  }
\end{tikzpicture}
\begin{tikzpicture}[x=1cm,y=1cm,scale=0.5]
\drawnaxis{13}
\drawiaxis{12}
\drawstaircase{1}{13}{0}{black!10}{0}
\drawstaircase{2}{3}{0}{red!40}{0}
\drawstaircase{6}{7}{2}{blue!40}{0}
\drawstaircase{10}{11}{4}{blue!40}{0}
\drawblock{4}{13}{0}{1}{blue!40}{0}
\drawblock{8}{13}{2}{3}{blue!40}{0}
\drawblock{12}{13}{4}{5}{blue!40}{0}
\foreach \n/\i/\txt in {
1/0/2,
4/2/4,
4/3/4,
5/2/4,
5/3/8,
5/4/4,
8/4/4,
8/5/12,
8/6/20,
8/7/12,
9/4/4,
9/5/10,
9/6/24,
9/7/38,
9/8/20,
12/6/4,
12/7/8,
12/8/20,
12/9/72,
12/11/72,
13/6/4,
13/7/8,
13/8/16,
13/9/56
  }{
    \cell{\n}{\i}{\txt}
  }
\foreach \n/\i/\txt in {
12/10/128,
13/10/152,
13/11/240,
13/12/132
 }{
    \smallcell{\n}{\i}{\txt}
  }
\end{tikzpicture}
    \caption{Homology of $\C_\wedge$ twisted by $\zeta_3$ or $\zeta_6$ (left), and twisted by $\zeta_4$ (right).}\label{fig: C_wedge zeta 3,6,4}
\end{figure}
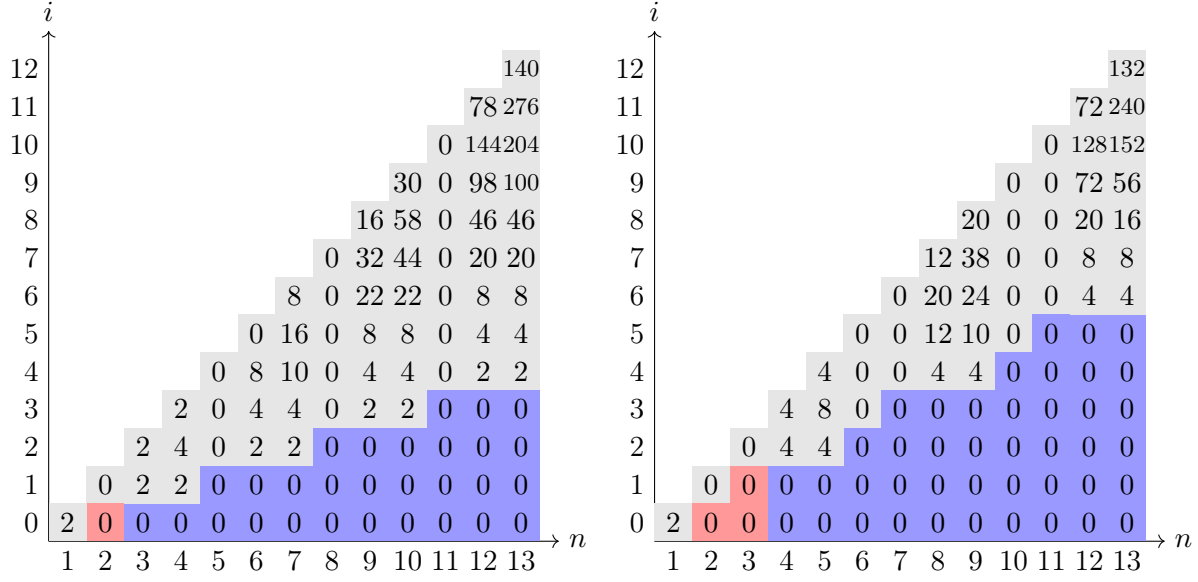
Homologically, our computations suggest that the bound we obtain for the slope of homological vanishing slope for the associated braided vector space is sharp in every order we examined. Consider the two examples of (a) order $o=3,6$ (recall that they have the same homology) and (b) order $o=4$, with their corresponding associated braided space $\C_\wedge$ twisted by $\zeta_3$ or $\zeta_6$ in (a) and $\zeta_4$ in (b). We compute the dimension of their respective homologies in Figure \ref{fig: C_wedge zeta 3,6,4}. We will prove that the homology vanishes in the red ``maximal'' staircases, leading to vanishing in the blue region of slope $1/3$ and $1/2$ respectively, which further corresponds to a power-savings of $q^{5/6}$ and $q^{3/4}$ which agrees with Theorem \ref{thm: gauss sums prime powers}. From the data in Figure \ref{fig: C_wedge zeta 3,6,4}, it seems reasonable to conclude that the slope of homological vanishing is sharp. In fact, we observe something stronger: the dimensions of the nonzero homology for the $\zeta_3,\zeta_6$ twist seems to stabilize to $2,4,8,20,\cdots$, which matches the dimensions of homological stability for the untwisted $\C_\wedge$ as given in Figure \ref{fig: C_wedge wo twist}. Here, the blue region in Figure \ref{fig: C_wedge wo twist} indicates the stable homology of $\C_\wedge$, and we remark that homological stability here follows from representation stability of \cite{Church_Farb_2013}. Meanwhile, the homological dimensions for the $\zeta_4$ twist seems to stabilize to twice the dimension of $\C_\wedge$. We believe that this observation is true in general: the nonzero stable homology for $\C_\wedge$ twisted by $\zeta_{p^k}$ or $\zeta_{2p^k}$ should be isomorphic to the direct sum of several copies of the stable homology of the untwisted $\C_\wedge$. One could ask whether this isomorphism, if it exists, is Frobenius equivariant, and if this were true, would lead to a leading term for the bias in Gauss sums, and is an interesting direction for future research. 

\begin{figure}[h]
    \centering
\begin{tikzpicture}[x=1cm,y=1cm,scale=0.5]
\drawnaxis{11}
\drawiaxis{10}
\drawstaircase{1}{11}{0}{black!10}{0}
\drawblock{1}{11}{0}{0}{blue!40}{2}
\drawblock{3}{11}{1}{1}{blue!40}{4}
\drawblock{5}{11}{2}{2}{blue!40}{8}
\drawblock{7}{11}{3}{3}{blue!40}{20}
\drawblock{9}{11}{4}{4}{blue!40}{44}
\drawblock{11}{11}{5}{5}{blue!40}{96}
\foreach \n/\i/\txt in {
2/1/2,
3/2/2,
4/2/6,
4/3/2,
5/3/12,
6/3/18,
5/4/6,
6/4/22,
7/4/36,
8/4/42,
6/5/10,
7/5/40,
8/5/72,
9/5/88,
10/5/94,
7/6/18,
8/6/76,
8/7/32,
9/8/56
  }{
    \cell{\n}{\i}{\txt}
  }
\foreach \n/\i/\txt in {
9/6/146,
10/6/188,
11/6/204,
9/7/144,
10/7/292,
11/7/392,
10/8/270,
11/8/580,
10/9/102,
11/9/512,
11/10/186
 }{
    \smallcell{\n}{\i}{\txt}
  }
\end{tikzpicture}
    \caption{Homology of $\C_\wedge$.}\label{fig: C_wedge wo twist}
\end{figure}
We refer to \cite[Section 2]{Chinta_Friedberg_Hoffstein_2011} and \cite[Section 1.2]{david2026quarticgausssumsprimes} for our analytic reasoning. The Dirichlet series corresponding to the bias in Gauss sums is closely related to Kubota's Eisenstein series on the $o$-fold cover of $\SL_2$, and this has a rightmost simple pole at $\frac 1 2 + \frac 1 o$. For $o=p^k$, this leads to the correct asymptotic, but it does not match for the $o=2p^k$ ($p$ odd) case. For the latter case, we speculate why the asymptotic $\frac 1 2 +\frac 1 {p^k}$ could be larger than expected from the pole at $\frac 1 2 +\frac 1 {2p^k}$. When sieving for primes, one would need to consider the sum over arithmetic progressions, for example, this appears in the Type I sum as written in \cite[Equation (1.4), (1.8)]{david2026quarticgausssumsprimes}. The sum over the arithmetic progression with common difference $a$ is proportional to the residue of the relevant Kubota's Eisenstein series at $\frac 1 2 +\frac 1 {2p^k}$, but these appear as the $a^2$-th Fourier-Whittaker coefficients of the $2p^k$-th order theta function, essentially due to \cite[Equation (1.6)]{david2026quarticgausssumsprimes}. For odd integers $n$, \cite{Chinta_Friedberg_Hoffstein_2011} conjectures that the Dirichlet series corresponding to these square coefficients of the $2n$-th order theta function has as a factor the Dirichlet series of the $n$-th order bias in Gauss sum, see Conjecture 2.1 of loc. cit. for the $n=3$ case and the end of Section 5 of loc. cit. for the general case. Substituting $n=p^k$, this contributes a pole of $\frac 1 2 +\frac 1 {p^k}$ to the Dirichlet series of square coefficients, so it seems plausible that the growth of the square coefficients of the theta function leads to this asymptotic. As in this argument, we believe that the behavior of the sum over prime arguments and the sum over squarefree arguments differs.
\subsubsection{Results for general order $o$}
It is less clear for general order $o$ what the asymptoptic of the bias in Gauss sums should be. We prove the following upper bound.
\begin{theorem}\label{thm: gauss sums general order} Suppose that $\chi\colon \F_q^\times \rightarrow \C^\times$ is a non-trivial character with order $o=\ord(\chi)$, and let $\omega(o)$ be the number of distinct prime factors of $o$. Then, we have 
$$\left|\sum_{\substack{\deg(\pi)=n\\ \pi \text{ prime}}} \frac{G_\chi(\pi)}{q^{n/2}} \right|\le \frac{2^{2n-2}}{n} q^{\left(\frac 1 2 +\frac {2^{\omega(o)-1}}{o}\right)n+C}$$
for some explicit constant $C$ depending only on $o$.
\end{theorem}

Note that this means that for any $\epsilon>0$, all but finitely many characters $\chi$ have cancellation with exponent smaller than $\frac 1 2+\epsilon$. Furthermore, the exponent agrees with Theorem \ref{thm: gauss sums prime powers} in the special cases of $o=p^k,2p^k$. However, outside of these special cases, the slope of $\frac 1 2 +\frac{2^{\omega(o)-1}}{o}$ is not sharp at all. For most orders, it is rather easy to prove a sharper slope. This can be done by computation in the spirit of Corollary \ref{cor: optimal slope cor}. Alternatively, without any explicit computation of homology, we can use the lemma that the homology vanishes for $n$ where $o\nmid n(n-1)$, and apply Theorem \ref{thm: main extrapolating vanishing} multiple times in ad-hoc ways on various vanishing staircases, depending on the small set of $n$ for which homology does not vanish. Often, a combination of the two methods will lead to a better bound.

We give examples of the vanishing slopes we can prove for various orders $o$ as well as the corresponding exponent in the power-savings bound in Figure \ref{fig: table order}, and compare it to the exponent we get from Theorem \ref{thm: gauss sums general order}. Note that most of these are likely not sharp. The corresponding homological results are obtained in Section \ref{sec: individual C_wedge}.

\begin{figure}[h]
\centering
\begin{tabular}{>{\centering\arraybackslash}m{3cm}
                >{\centering\arraybackslash}m{3cm}
                >{\centering\arraybackslash}m{3cm}
                >{\centering\arraybackslash}m{3cm}}
\toprule
\textbf{Order} $\mathbf{o}$ & \textbf{Vanishing slope} & \textbf{Exponent} & $\mathbf{\frac{1}{2}+\frac{2^{\omega(o)-1}}{o}}$ \\
\midrule
$12$ & $\frac 5 6 -\epsilon$ & $\frac 1 2 +\frac 1 {12} + \epsilon$ & $\frac 1 2 +\frac 1 6$ \\
 \midrule
$15,30$ & $\frac{21}{25}$ & $\frac 1 2 + \frac 2 {25}$ & $\frac 1 2 + \frac 2 {15}$\\
\midrule
$20$ & $\frac{17}{20}$ & $\frac 1 2 +\frac 3 {40}$ & $\frac 1 2 + \frac{1}{10}$ \\
\midrule
$60$ & $\frac 9 {10}$ & $\frac 1 2+ \frac 1 {20}$ & $\frac 1 2 + \frac 1 {15} $\\
\midrule
$105,210$ & $\frac{33}{35}-\epsilon$ & $\frac 1 2 + \frac 1 {35}+\epsilon$ & $\frac 1 2+\frac 4 {105}$\\
\bottomrule
\end{tabular}
\caption{Bounds for some orders for any $\epsilon>0$.}\label{fig: table order}
\end{figure}
We believe the true exponent should be somewhere between $\frac 1 2+\frac 1 o$ and $\frac 1 2 +\frac {2^{\omega(o)-1}}{o}$, because of Theorem \ref{thm: gauss sums general order} and the pole of Kubota's Eisenstein series at $\frac 1 2 +\frac 1 o$ as mentioned above. As we managed to prove an upper bound of $\frac 1 2 +\frac 1 o+\epsilon$ when $o=12$, this leads us to believe that the true exponent is more likely to be $\frac 1 2+\frac 1 o$, although we do not have enough evidence to make this a conjecture.
\subsection{Application to Möbius and character sums over Galois $G$-extensions}\label{sec: intro character sums G-extensions} We give another arithmetic application of our technique, this time to Möbius and character sums over $G$-extensions of the rational function field $\F_q(t)$. We set this up following \cite{ES26}. Let $G$ be a group and $R=R_1\sqcup \cdots \sqcup R_n\subseteq G$ be a union of conjugacy classes, and for simplicity we assume that each $R_i$ is closed under $q$-th powering, i.e. $g\in R_i\Rightarrow g^q\in R_i$. For example, one could simply take $q\equiv 1\pmod {|G|}$ by Lagrange's theorem. Define $\mathcal E^{R}_q(G;n_1,\ldots, n_k)$ to be the set of $G$-extensions of $\F_q(t)$ split over $\infty$, where the sum of degrees of all places with monodromy in $R_j$ is $n_j$, and let $n=n_1+\cdots +n_k$. In addition, we fix an an extension of the infinite place of the function field to its separable closure, which effectively marks a point over $\infty$ of the corresponding branched cover. Furthermore, denote $f_L\in \F_q[t]$ to be the conductor of $L\in \mathcal E^{R}_q(G;n_1,\ldots, n_k)$, this is defined to be the radical of the discriminant of $L$, or alternatively the monic squarefree polynomial with the branch points of $L/\F_q[t]$ viewed as a branched cover over $\mathbb P^1_{\F_q}$. We refer the reader to \cite[Section 1.2]{ES26} for more details. 

\subsubsection{Möbius function} Recall that the Möbius function $\mu(f)$ of a polynomial $f\in \F_p[t]$ is defined by $\mu(f)=(-1)^{\omega(f)}$, where $\omega(f)$ is the number of distinct irreducible factors of $f$; this is completely analogous to the case over $\Z$. \cite[Theorem 1.2.4]{ES26} showed that there is cancellation with a power-savings for the sum of the Möbius function $\mu(f_L)$ over $\mathcal E^R_q(G;n_1,\ldots, n_k)$ if $q$ is large enough relative to $|G|$. This is not surprising because we expect the conductor $f_L$ to behave randomly. Explicitly, they prove 
\begin{equation}\label{eqn: mobius es}
\left| \sum_{L\in \mathcal E^{R}_q(G;n_1,\ldots, n_k)} \mu(f_L)\right | \le 2^{n-1}|R|^nq^{n-\frac{n-d}{2d+4}}  
\end{equation}
where $d=\deg(\bigoplus_{n=0}^\infty H_0(B_n,(\C R)_{-1}))\ge 1$ is finite. Here, $(\C R)_{-1}$ is the relevant braided vector space for this arithmetic problem, and we will define it in Example \ref{eg: BVS}. The bound follows from applying \cite[Theorem 1.1.1]{ES26} (or equivalently Corollary \ref{cor: es main theorem}), and as discussed earlier, because of this, the best power-saving exponent they can get is $\frac 5 6$ when $d=1$. 

Instead, we apply Corollary \ref{cor: optimal slope cor} to obtain better exponents for some choices of $G$ and $R$ by computing the homology up to $n\le N$, and we present our results in Figure \ref{fig: table mobius G,R}. Note that we can theoretically carry this out for any choice of $G,R$, but since the computation time is exponential with base $|R|$, this is more practical for small $|R|$. The corresponding homological results are obtained in Section \ref{sec: individual neg twists}.

\begin{figure}[h]
\centering
\begin{tabular}{>{\centering\arraybackslash}m{1.5cm}
>{\centering\arraybackslash}m{5cm}
>{\centering\arraybackslash}m{3cm}
>{\centering\arraybackslash}m{3cm}
>{\centering\arraybackslash}m{1cm}}
\toprule
$\mathbf{G}$ & $\mathbf{R}$ & \textbf{Vanishing slope} & \textbf{Exponent} & $\mathbf{N}$ \\
\midrule
$S_3$ & transpositions & $\frac 2 3$ & $\frac 2 3$ & 11\\
\midrule
$S_4$ & transpositions or 4-cycles & $\frac 3 7$ & $\frac {11}{14}$ & $6$\\
\midrule
$S_5$ & transpositions & $\frac 1 3$ & $\frac 5 6$ & $5$\\
\midrule
$D_5,D_7$ & reflections & $\frac 1 2$ & $\frac 3 4$ & 7\\  
\midrule 
$A_4$ & one conj class of $3$-cycles & $\frac 1 2$ & $\frac 3 4$ & $5$\\
\midrule
$C_5\rtimes C_4$ & $\{(x,1)\colon x\in C_5\}$ & $\frac 1 2$ & $\frac 3 4$ & $7$\\
\bottomrule
\end{tabular}
\caption{Bounds for some $G,R$.}\label{fig: table mobius G,R}
\end{figure}
Note that these bounds are not sharp, and computing the homology up to larger and larger values of $N$ appears to yield progressively stronger bounds. We predict that in most cases the true exponent will be close to $\frac 1 2$, which reflects true randomness of $f_L$. However, there seems to be some exceptions to this. For example, when $G=S_4$ and $R$ is the conjugacy class of transpositions or $4$-cycles (these have the same homology as the local systems are isomorphic by \cite[Example 5.1.7(3)]{Andruskiewitsch_Grana_2003}), the homology of $(\C R)_{-1}$ seems to have slope $\frac 1 2$ which corresponds to an exponent of $\frac 3 4$, as suggested empirically by Figure \ref{fig: transpositions in S_4}. Hence, we do not think that a uniform theorem for Möbius sums over all $G,R$ exists, and the best one can do is to compute for each individual choice of $G,R$.

\begin{figure}[h]
    \centering
\begin{tikzpicture}[x=1cm,y=1cm,scale=0.7]
\drawnaxis{7}
\drawiaxis{6}
\drawstaircase{1}{7}{0}{black!10}{0}
\drawstaircase{4}{6}{0}{red!40}{0}
\drawblock{7}{7}{0}{2}{blue!40}{0}
\foreach \n/\i/\txt in {
1/0/\mathbf{6},
2/0/3,
2/1/3,
3/1/\mathbf{6},
3/2/6,
4/1/3,
4/2/40,
4/3/37,
5/2/\mathbf{6},
5/3/168,
5/4/162,
6/3/39,
7/3/\mathbf{6},
7/4/42,
6/4/672,
6/5/633
  }{
    \cell{\n}{\i}{\txt}
  }
\foreach \n/\i/\txt in {
7/5/2610,
7/6/2574
 }{
    \smallcell{\n}{\i}{\txt}
  }
\end{tikzpicture}
\caption{Homology of $(\C R)_{-1}$ for $R$ the conjugacy class of transpositions or $4$-cycles in $S_4$.}\label{fig: transpositions in S_4}
\end{figure}
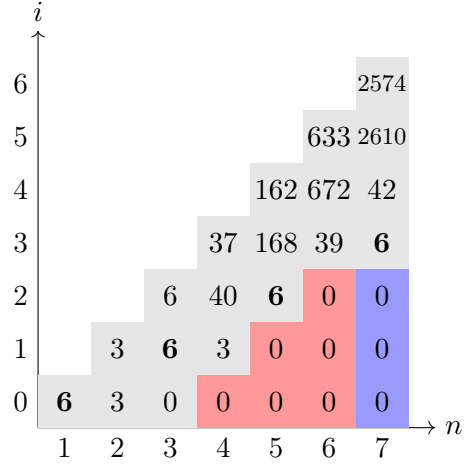 
\subsubsection{Character sums} Over function fields, if $q$ is not a power of $2$, the Möbius function can be written as $\mu(f)=(-1)^{\deg f}\xi(\disc (f_L))$, where $\xi\colon \F_q^\times \rightarrow \C^\times$ is the unique quadratic character. This motivates one to consider the sum of  $\chi(\disc(f_L))$ for a general character $\chi\colon \F_q^\times \rightarrow \C^\times$ instead. Generalizing Equation \eqref{eqn: mobius es}, \cite[Theorem 1.2.5]{ES26} prove the bound
$$\left| \sum_{L\in \mathcal E^{R}_q(G;n_1,\ldots, n_k)} \chi(\disc(f_L))\right | \le 2^{n-1}|R|^nq^{n-\frac{n-d}{2d+4}} $$
where $d=\deg(\bigoplus_{n=0}^\infty H_0(B_n,(\C R)_{\zeta}))\ge 1$ is finite, and $\zeta\in \C^\times$ is a root of unity of the same order as $\chi$. Instead, we prove the following bound.
\begin{theorem}\label{thm: character G extensions}
Suppose that $\chi\colon \F_q^\times \rightarrow \C^\times$ is a non-trivial character with order $o=\ord(\chi)$, let $o'=o/\gcd(o,|G|)$ and $\omega(o')$ be the number of distinct prime factors of $o'$. Then, we have 
$$\left| \sum_{L\in \mathcal E^{R}_q(G;n_1,\ldots, n_k)} \chi(\disc(f_L))\right | \le 2^{n-1}|R|^n q^{\left(\frac 1 2 +\frac {2^{\omega(o')-1}}{o'}\right)n+C}.$$
for some explicit constnat $C$ depending on $o'$.
\end{theorem}

Note that if we fix $G,R$ and consider the family of sums obtained by varying $\chi$, the result implies that for any $\epsilon>0$, all but finitely many characters $\chi$ have cancellation with exponent smaller than $\frac 1 2+\epsilon$. Indeed, such an almost-all square-root cancellation statement is the best one can hope for because just like in the Möbius case, we do not expect all characters to have close to square-root cancellation. In addition to Figure \ref{fig: transpositions in S_4}, another example which seems to have smaller power-savings is $G=S_3$, $R$ the conjugacy class of transpositions and $\chi$ a sextic character, which we show in Figure \ref{fig: S_3 trans 6th root}. The vanishing slope in this example seems to be $\frac 1 3$ which corresponds to an exponent of $\frac 5 6$. As we can also see from this example, the $q$-exponent in the bound is often not optimal, and one can in theory apply Corollary \ref{cor: optimal slope cor} to get better bounds for specific characters.

\begin{figure}[h]
    \centering
\begin{tikzpicture}[x=1cm,y=1cm,scale=0.55]
\drawnaxis{9}
\drawiaxis{8}
\drawstaircase{1}{9}{0}{black!10}{0}
\drawstaircase{2}{2}{0}{red!40}{0}
\drawstaircase{5}{5}{1}{blue!40}{0}
\drawstaircase{8}{8}{2}{blue!40}{0}
\drawblock{3}{9}{0}{0}{blue!40}{0}
\drawblock{6}{9}{1}{1}{blue!40}{0}
\drawblock{9}{9}{2}{2}{blue!40}{0}
\foreach \n/\i/\txt in {
1/0/3,
3/1/\mathbf{6},
3/2/6,
4/1/6,
4/2/11,
4/3/5,
6/2/\mathbf{6},
6/3/11,
6/4/16,
6/5/11,
7/2/6,
7/3/9,
7/4/9,
7/5/36,
7/6/30,
8/4/2,
8/5/8,
8/6/34,
8/7/28,
9/3/\mathbf{6},
9/4/6,
9/5/3,
9/6/18
 }{
    \cell{\n}{\i}{\txt}
  }
\foreach \n/\i/\txt in {
9/7/153,
9/8/138
 }{
    \smallcell{\n}{\i}{\txt}
  }
\end{tikzpicture}
\caption{Homology of $(\C R)_{\zeta_6}$ for $R$ the conjugacy class of transpositions in $S_3$.}\label{fig: S_3 trans 6th root}
\end{figure}
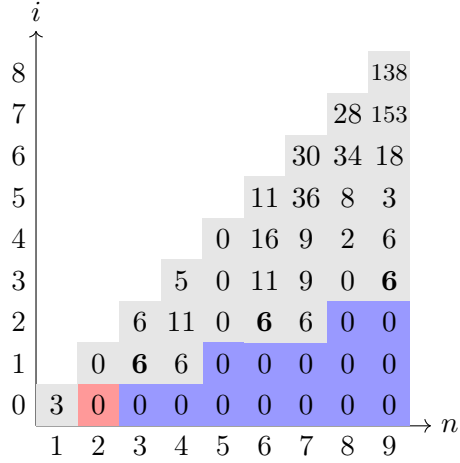
\subsubsection{Multiple characters}\label{sec: intro multiple characters} We can also consider the more general situation where we have a product of characters. Let us first fix a choice of $G,R$. For $L\in \mathcal E^R_q(G;n_1,\ldots, n_k)$, let $f_{L,i}$ be the monic squarefree polynomial with roots that are branch points with monodromy in $R_i$. Let $\chi_i\colon \F_q^\times \rightarrow \C^\times$ for $1\le i\le k$ and $\chi_{ij}\colon \F_q^\times \rightarrow \C^\times$ for $1\le i<j\le k$ be characters, and we consider the sum of $\prod_{i} \chi_{i} (\disc(f_{L,i}))\prod_{i<j}\chi_{ij}(\res(f_{L,i},f_{L,j}))$ where $\res(f,g)$ denotes the resultant of the two polynomials $f,g$. To see why this generalizes our previous scenario, note that when $\chi_i=\chi$ and $\chi_{ij}=\chi^2$ the product of characters is simply $\chi(\disc (f_L))$. 

Just like in Theorem \ref{thm: character G extensions}, we expect that for almost all choices of $\chi_i,\chi_{ij}$ that there is good cancellation in this sum. It will again be convenient to view these characters as roots of unity, we do this by picking a generator $a$ of $\F_q^\times$ and defining $\zeta_i=\chi_i(a)$, $\zeta_{ij}^2=\chi_{ij}(a)$ (up to a sign). We choose the convention of $\zeta_{ij}^2$ because of the identity $\disc(fg)=\disc(f)\disc(g)\res(f,g)^2$. The space of characters is then parametrized by the torsion points $((\zeta_i)_i,(\zeta_{ij})_{ij})\in (\C^\times)^{k(k+1)/2}$. Now, consider the example where $n_1$ is nonzero and $n_2=\cdots=n_k=0$, here the sum simplifies to the sum in Theorem \ref{thm: character G extensions} with $\chi=\chi_1$ and $R=R_1$, and as discussed above this may not have good cancellation for some values of $\chi$. Such a set $\{\chi_1=b\}$ is a codimension one torsion coset, and one may expect other codimension one torsion cosets to be exceptions to good cancellation. Our theorem shows that we have near square-root cancellation after accounting for these exceptions.
\begin{theorem}\label{thm: multiple characters} Fix $G,R$ and let $N\ge 2$ be an integer. Then there exist finitely many equations of the form $\prod_i\zeta_i^{a_i}\prod_{i<j}\zeta_{ij}^{a_{ij}}=b_i$ where $a_i,a_{ij}$ are integers which are not all zero and $b_i\in \C^\times$ is a root of unity, such that if $((\zeta_i)_i,(\zeta_{ij})_{ij})$ is not a solution of any of these equations, then we have the bound
$$\left|\sum_{L\in \mathcal E^{R}_q(G;n_1,\ldots, n_k)} \prod_{1\le i\le k} \chi_{i} (\disc(f_{L,i}))\prod_{1\le i< j \le k}\chi_{ij}(\res(f_{L,i},f_{L,j}))\right|\le 2^{n-1}|R|^n q^{\left(\frac 1 2 +\frac 1 {N+1}\right)(n-1)+1}.$$
\end{theorem}

One might notice the resemblance of the condition to the Manin-Mumford theorem for tori, and this is not a coincidence as we use the theorem in our proof. We also obtain one explicit system of equations in Proposition \ref{prop: explicit torsion coset}.

\subsection{Acknowledgments}
I would like to thank my advisor Will Sawin for providing me much valuable guidance along the way. I would also like to thank Jordan Ellenberg, Aaron Landesman, Ishan Levy, David Lin and Mark Shusterman for helpful discussions and feedback. 

I acknowledge the use of OpenAI's GPT-5.5 in generating code used for the computations of homology in this paper, this is discussed in Section \ref{sec: on homological computations}. I take full responsibility for the correctness of the outputs of any AI-generated code. The mathematical content of this paper is my own.
\section{Preliminaries}\label{sec: preliminaries}
\subsection{Braided vector spaces and shuffle algebras}
Here, we introduce the notions of a braided vector space, quantum shuffle algebra and bar-complex following \cite[Section 2]{ETW17} and \cite[Section 2,3]{KS20}.
\begin{definition}
A \textit{braided vector space} $(V,R)$ over a field $k$ consists of a finite-dimensional $k$-vector space $V$ with a braiding $R\colon V\otimes V\xrightarrow{\sim} V\otimes V$ which satisfies the Yang-Baxter equation
$$(R\otimes \id)\circ (\id \otimes R)\circ (R\otimes \id)=(\id \otimes R)\circ (R\otimes \id)\circ (\id \otimes R).$$
This relation is exactly what is required for $V^{\otimes n}$ to be a representation of $B_n$ for all $n$, where $\sigma_i$ acts on the $i$ and $i+1$-th copies of $V$ by $R$. 
\end{definition}

Throughout this paper, we will assume unless otherwise stated that our coefficient field $k=\C$. We say that a braided vector space $(V,R)$ over $\C$ is defined over a number field $K$ if there exists some braided vector space $(V_K,R_K)$ over $K$ where $V=V_K\otimes_K C$ and $R=R_K\otimes_K \C$. Often, we will refer to $(V,R)$ simply as $V$.

\begin{example}\label{eg: BVS}
We give the two examples of braided vector spaces $(V,R)$ that will be relevant for us.
\begin{enumerate}[(\alph*)]
    \item Braided vector spaces of diagonal type. Let $\{v_i\}_i$ be a basis of $V$, and $(q_{ij})$ be a matrix with entries in $\C^\times$. We define the braiding on the basis $v_i\otimes v_j$ by $R(v_i\otimes v_j)=q_{ij}v_j\otimes v_i$.
    \item Braided vector spaces of rack type. A \textit{rack} is a set $S$ with binary operation $(a,b)\mapsto b^a$ satisfying (i) $(c^a)^{b^a}=(c^b)^a$ and (ii) for each $a,b\in S$, there is a unique $c\in S$ with $c^a=b$ which we denote as $c={}^ab$. We define the braided vector space $\C S$ to have basis $S$ and braiding $R(a\otimes b)=b\otimes a^b$. For example, if $S$ is a union of conjugacy classes in a group $G$, then it is a rack with binary operation $g^h=h^{-1}gh$.
 \end{enumerate}
 
One can also twist a braided vector space $(V,R)$ by some $\zeta\in \C^\times$ to get $(V_\zeta,R_\zeta)$ where $R_\zeta(a\otimes b)=\zeta\cdot R(a\otimes b)$. One usually takes $\zeta$ to be a root of unity.
\end{example}

To each braided vector space $V$, we can associate a quantum shuffle algebra $\mathfrak A(V)$. Before that, we define the braiding $R_\sigma\colon V^{\otimes n}\xrightarrow \sim V^{\otimes n}$ attached to a permutation $\sigma \in S_n$ as follows. Given such a $\sigma$, there exists a Matsumoto lift $\tilde \sigma\in B_n$ given by expressing $\sigma$ as the minimum length (reduced) word in transpositions $s_i=(i,i+1)$ and replacing these with the braid $\sigma_i$. This lift is well-defined because any two reduced words are related by $s_is_{i+1}s_i=s_{i+1}s_is_{i+1}$ or $s_is_j=s_js_i$ for $|i-j|>1$ which are exactly the braid relations. Then, the braiding $R_\sigma$ is given by the action of $\tilde \sigma$ on the $B_n$-representation $V^{\otimes n}$.

\begin{definition}
For a braided vector space $V$, the quantum shuffle algebra $$\mathfrak A(V)=\bigoplus_{n=0}^\infty V^{\otimes n}$$ 
is the graded bialgebra with free comultiplication, where components $\Delta_{p,q}\colon V^{\otimes p+q}\xrightarrow\sim V^{\otimes p}\otimes V^{\otimes q}$ are given by the identity map, and multiplication $\tau$ with components $\tau_{p,q}$ given by the shuffle product
$$\tau_{p,q}(v,w) = \sum_{\sigma\in \text{Sh}(p,q)} R_\sigma (v\otimes w).$$
Here, we define the set of $(p,q)$-shuffles to be
$$\text{Sh}(p,q)\coloneqq\{\sigma\in S_{p+q} \colon \sigma(1)<\cdots <\sigma(p) \text{ and }\sigma(p+1)<\cdots<\sigma(p+q)\}$$
which are the permutations that preserve the order of the first $p$ elements and last $q$ elements.
\end{definition}

Now, we define the $n$-th bar-cube of $\mathfrak A=\mathfrak A(V)$ to be the hypercube diagram where vertices are labeled by the $2^{n-1}$ ordered partitions $[n]$, with the vertex $\lambda=(\lambda_1,\cdots ,\lambda_k)$ having the term $\mathfrak A_{\lambda_1}\otimes \cdots \otimes \mathfrak A_{\lambda_k}$, where the edges are given by the shuffle products. For example, for $n=3$ we have
% https://q.uiver.app/#q=WzAsNCxbMCwwLCJcXG1hdGhmcmFrIEFfMVxcb3RpbWVzIFxcbWF0aGZyYWsgQV8xXFxvdGltZXMgXFxtYXRoZnJhayBBXzEiXSxbMSwwLCJcXG1hdGhmcmFrIEFfMVxcb3RpbWVzIFxcbWF0aGZyYWsgQV8yIl0sWzAsMSwiXFxtYXRoZnJhayBBXzJcXG90aW1lcyBcXG1hdGhmcmFrIEFfMSJdLFsxLDEsIlxcbWF0aGZyYWsgQV8zIl0sWzAsMSwiXFxpZCBcXG90aW1lcyBcXHRhdV97MSwxfSJdLFswLDIsIlxcdGF1X3sxLDF9XFxvdGltZXMgXFxpZCIsMl0sWzIsMywiXFx0YXVfezIsMX0iLDJdLFsxLDMsIlxcdGF1X3sxLDJ9Il1d
\[\begin{tikzcd}[column sep=large]
	{\mathfrak A_1\otimes \mathfrak A_1\otimes \mathfrak A_1} & {\mathfrak A_1\otimes \mathfrak A_2} \\
	{\mathfrak A_2\otimes \mathfrak A_1} & {\mathfrak A_3.}
	\arrow["{\id \otimes \tau_{1,1}}", from=1-1, to=1-2]
	\arrow["{\tau_{1,1}\otimes \id}"', from=1-1, to=2-1]
	\arrow["{\tau_{1,2}}", from=1-2, to=2-2]
	\arrow["{\tau_{2,1}}"', from=2-1, to=2-2]
\end{tikzcd}\]
From the $n$-th bar-cube, we can form the $n$-th bar-complex 
$$B_n(\mathfrak A)=\left\{\mathfrak A_1^{\otimes n}\rightarrow \cdots \rightarrow \bigoplus_{\substack{p+q+r=n\\
0<p,q,r<n}} \mathfrak A_p\otimes \mathfrak A_q \otimes \mathfrak A_r\rightarrow \bigoplus_{\substack{p+q=n\\
0<p,q<n}} \mathfrak A_p\otimes \mathfrak A_q \rightarrow \mathfrak A_n\right\}$$
by collapsing the maps in the bar-cube, and adding a Koszul sign twist so that the differentials of the complex $d$ satisfy $d\circ d=0$. The grading of this complex is normalized so that $\mathfrak A_1^{\otimes n}$ is in degree $-n$ and $\mathfrak A_n$ is in degree $-1$.

The point of introducing quantum shuffle algebras and bar-complexes is to give a way to compute the homology of braided vector spaces:
\begin{equation}\label{eqn: homology to bar-complex}
H_i(B_n,V^{\otimes n})^* \cong H^i(B_n,V^{*\otimes n})\cong \Tor^{\mathfrak A(V^*_{-1})}_{n-i,n}(\C,\C)\cong H^{i-n}(B_n(\mathfrak A(V^*_{-1}))). 
\end{equation}
The asterisk here refers to the dual of the corresponding vector space, and $V_{-1}$ is $V$ with braiding twisted by a sign as previously defined. The first isomorphism follows from a universal coefficient theorem, the second isomorphism follows from \cite[Corollary 3.3.4]{KS20} or the dual of \cite[Theorem 1.3]{ETW17}, and the last isomorphism follows is a general fact that $\Tor$ can be calculated from the bar-complex. Hence, to show vanishing of $H_i(B_n,V^{\otimes n})$, it suffices to show that the cohomology of the corresponding bar-complex $\mathfrak A(V_{-1}^*)$ vanishes in degree $i-n$. 

\subsection{Local systems on $\Conf^n$} We define $\Conf^n$ to be the configuration space of $n$ distinct unordered points on the affine line $\A^1$ over $\Z$. By sending the $n$ points to the monic polynomial $f$ with these $n$ roots, we can identify this with the open subset of $\A^n$ where the discriminant is invertible. Then, the discriminant gives a morphism $\delta\colon \Conf^n\rightarrow \G_m$. 

The analytification $\Conf^n(\C)$ of the fiber over $\C$ is the Eilenberg-MacLane space $K(B_n,1)$. By choosing the basepoint $c_n=\{1,2,\ldots, n\}\in \Conf^n(\C)$, we have the isomorphism $\pi_1(\Conf^n(\C),c_n)\cong B_n$, where the generator $\sigma_i$ of the braid group corresponds to the half Dehn twist swapping the points $i$ and $i+1$ by moving them counterclockwise around each other. 

For non-negative integers $n_1+\ldots +n_k=n$, we define $\Conf^{n_1,\ldots, n_k}$ to be the multicolored configuration space of $n$ distinct points on the affine line with $n_i$ points of each color. We can view this as a $k$-tuple of monic polyomials $(f_1,\ldots, f_k)$ where $f_i$ has degree $n_i$. Now, we have $\pi_1(\Conf^{n_1,\ldots, n_k}(\C),c_n)\cong B_{n_1,\ldots, n_k}$, which is defined to be the subgroup of $B_n$ which preserves the coloring, i.e. the preimage of $S_{n_1}\times \cdots \times S_{n_k}$ under the map $B_n\twoheadrightarrow S_n$. Forgetting the coloring gives a finite étale map $\tau\colon \Conf^{n_1,\ldots n_k}\rightarrow \Conf^n$. As a special case, define the ordered configuration space $\PConf^n=\Conf^{1,\ldots, 1}$ where there are $n$ $1$s, then $\tau\colon \PConf^n\rightarrow \Conf^n$ is a finite étale $S_n$-cover.

In the following examples, we recall the construction of three types of local systems over $\Conf^n$ that were given in \cite[Section 3]{ES26}, together with their trace functions over $\F_q$ and their analytifications over $\C$ as $B_n$-representations. Let $l$ be any prime coprime to $q$ and fix an isomorphism $\bar \Q_l\cong\C$ which we will implicitly use throughout the paper.
\begin{example}[Local systems from representations of $S_n$]\label{eg: local system rep S_n}
Let $\rho$ be a finite-dimensional representation of $S_n$ over $\bar \Q_l$. We can view this as a local system $\mathcal L_\rho$ over $\Conf^n$ via the homomorphism $h\colon \pi_1^{et}(\Conf^n)\rightarrow S_n$ that corresponds to the cover $\tau\colon \PConf^n\rightarrow \Conf^n$. Over $\C$, the local system $\mathcal L_\rho$ corresponds to the $B_n$-representation given by the composition of $B_n\twoheadrightarrow S_n$ and $\rho$.

Let $f\in \Conf^n(\F_q)$ be a point viewed as a monic polynomial, and let $\sigma_f\in S_n$ be the permutation given by the action of $\Frob_q$ on the roots of $f$. Then, $h$ sends $\Frob_{q,f}$ to $\sigma_f$, so the trace function is $$\tr(\Frob_q,(\mathcal L_{\rho})_{\bar f})=\tr(\rho(\sigma_f))$$ where $\bar f$ is a geometric point over $f$.
\end{example}
\begin{example}[Character Sheaves] \label{eg: local system character sheaves} Let $\chi\colon \F_q^\times \rightarrow \C^\times$ be a multiplicative character. Recall that the Kummer sheaf $\mathcal L_\chi$ on $\G_m$ over $\F_q$ can be constructed as follows. Consider the covering $[q-1]\colon \G_m\rightarrow \G_m$ given by raising to the power of $q-1$, this is a Galois cover with Galois group $\F_q^{\times}$ so it gives a map $\pi_1^{et}(\G_m/\F_q)\rightarrow \F_q^\times$. Composing with the character $\chi$ gives the Kummer sheaf $\mathcal L_\chi$ with trace function $\tr(\Frob_q,(\mathcal L_{\chi})_{\bar x})=\chi(x)$ for $x\in \G_m(\F_q)$.

We can pull this back via the discriminant map to get $\delta^*\mathcal L_\chi$ which has trace function
$$\tr(\Frob_q,(\delta^*\mathcal L_\chi)_{\bar f})=\chi(\disc(f))$$
for $f\in \Conf^n(\F_q)$.

By choosing a map $\Spec \F_q \rightarrow \Spec \Z_p[\zeta_{q-1}]$, where we write $q=p^e$, one can use the same construction above to ``spread-out'' the Kummer sheaf to $\G_m$ over $\Z_p[\zeta_{q-1}]$ so that the pullback to $\Spec \F_q$ is $\mathcal L_\chi$; the point is that the covering $[q-1]\colon \G_m\rightarrow \G_m$ over $\Z_p[\zeta_{q-1}]$ is still a finite étale Galois cover so the construction works as stated. Then, the analytification of $\delta^*\mathcal L_\chi$ corresponds to the $B_n$-representation where each generator $\sigma_i$ is sent to $\zeta$, where $\zeta\in \C^\times$ is a root of unity of the same order as $\chi$.

We now discuss the generalization to multicolored configuration space. As in the introduction, let $\chi_i\colon \F_q^\times \rightarrow \C^\times$ for $1\le i\le k$ and $\chi_{ij}\colon \F_q^\times \rightarrow \C^\times$ for $1\le i < j\le k$ be characters, and define the roots of unity $\zeta_i=\chi_i(a)$, $\zeta_{ij}^2=\chi_{ij}(a)$ for a generator $a$ of $\F_q^\times$. Recall that we write a point of $\Conf^{n_1,\ldots, n_k}(\F_q)$ as $x=(f_1,\ldots, f_k)$, and consider the tensor product of pullbacks of Kummer sheaves
\begin{equation}\label{eqn: multiple kummer}
\mathcal L_{\vec\chi}=\left(\bigotimes_{1\le i\le k} \disc(f_i)^*\mathcal L_{\chi_i} \right)\otimes \left(\bigotimes_{1\le i < j \le k} \res(f_i,f_j)^*\mathcal L_{\chi_{ij}} \right), 
\end{equation}
which has trace function
$$\tr(\Frob_q,\mathcal L_{\bar x})=\prod_{1\le i \le k} \chi_i(\disc(f_i))  \prod_{1\le i<j\le k} \chi_{ij}(\res(f_i,f_j)).$$

Just like before, we can spread out the sheaf $\mathcal L$, and its analytification would correspond to the $B_{n_1,\ldots ,n_k}$-representation where a positive half-twist of two adjacent strands of the same color $i$ acts by $\zeta_i$ while every positive full twist of a strand of color $i$ and a strand of color $j$ acts by $\zeta_{ij}^2$.
\end{example}
\begin{example}[Pushforward from Hurwitz spaces]\label{eg: local system hurwitz space} We first discuss the construction of Hurwitz spaces. Suppose that $q$ is coprime to $|G|$ and the union of conjugacy class $R\subseteq G$ is closed under $q$-th powering ($x\in R\Rightarrow x^q\in R$). Then, from \cite[Section 2.1]{landesman2025cohenlenstramomentsfunctionfields}, we can define the pointed Hurwitz scheme $\Hur^{n}_{G,R}$ over $\Z_p[\zeta_{q-1}]$ which roughly parametrizes $G$ covers with local inertia in $R$ together with a point marked over infinity (if $\infty$ is not split, a root-stack construction is used). The analytification over $\C$ parametrizes $n$-branched $G$-covers of a disk $D$ that have a marked point on the boundary of the cover with local monodromy around each branch point in $R$, see \cite[Section 2]{EVW16} for more details.

Now, write $R=R_1\sqcup \cdots\sqcup R_k$ where each $R_i$ is a union of conjugacy classes, each closed under $q$-th powering. The geometric components of $\Hur^n_{G,R}$ are indexed by Hurwitz orbits, which are the $B_n$-orbits on $R^n$ with braiding on consecutive elements defined by sending $g,h$ to $h,g^h$ like in Example \ref{eg: BVS}(b). Over the special fiber $\F_q$, Frobenius acts on these components (viewed as orbits in $R^n$) by sending $(g_1,\ldots, g_n)$ to $(g_1^q,\ldots, g_n^q)$. Since each $R_i$ is closed under $q$-th powering, every Galois orbit of geometric component has the same number of branch points with local monodromy in each $R_i$. Let $\Hur^{n_1,\ldots, n_k}_{G,R}$ be the subscheme of $\Hur^n_{G,R}$ which is the union of geometric components with $n_i$ branch points of local monodromy in $R_i$, and this is defined over $\Z_p[\zeta_{q-1}]$ by the argument above. For our purposes, we can also consider the open and closed subscheme $\Hur^{n_1,\ldots, n_k}_{G,R,\infty}$ where the cover is split over $\infty$. We then have
$$\# \Hur^{n_1,\ldots, n_k}_{G,R,\infty}(\F_q)= \# \mathcal E^R_q(G;n_1,\ldots, n_k).$$
We also remark that one can also use the construction of Hurwitz spaces given in \cite{Liu_Wood_Zureick-Brown_2024}.

There is a finite étale map $\pi \colon \Hur^{n_1,\ldots, n_k}_{G,R,\infty}\rightarrow \Conf^{n_1,\ldots ,n_k}$ that sends a $G$-cover to its branch points where a branch point with local mondoromy $R_i$ is colored with color $i$. On $\F_q$-points, it sends a $G$-cover $L\in \mathcal E^R_q(G;n_1,\ldots, n_k)$ to $(f_{L,1},\ldots, f_{L,k})$ as defined in the introduction. Now, we consider the local system $\mathcal L_{G,R}=\pi_*\bar\Q_l$ which is the pushforward of the constant sheaf on Hurwitz space. By the previous discussion and properties of pushforwards, the trace of Frobenius on $x=(f_1,\ldots, f_k)\in \Conf^{n_1,\ldots, n_k}(\F_q)$ is given by $$\tr(\Frob_q,(\mathcal L_{G,R})_{\overline x})= \# \{L\in \mathcal E^R_q(G;n_1,\ldots, n_k)\colon f_{L,i}=f_i \ \forall 1\le i\le k\}.$$

Furthermore, properties of topological Hurwitz spaces imply that the analytification of the local system $\mathcal L_{G,R}$, when viewed as a representation of $B_{n_1,\ldots, n_k}$, is a direct summand of the representation $(\C R_1)^{\otimes n_1}\otimes \cdots \otimes (\C R_k)^{\otimes n_k}\subseteq (\C R)^n$. 

Alternatively, we can consider the map $\pi'\colon \Hur^{n_1,\ldots, n_k}_{G,R,\infty}\rightarrow \Conf^{n}$ with the corresponding pushforward $\mathcal L'_{G,R}=\pi_*'\bar \Q_l$ which has trace
$$\tr(\Frob_q,(\mathcal L_{G,R}')_{\overline f})=\# \{L\in \mathcal E^R_q(G;n_1,\ldots, n_k)\colon f_L=f\}$$
and its analytification as a $B_n$-representation is a direct summand of $$\Ind_{B_{n_1,\ldots, n_k}}^{B_n}\left((\C R_1)^{\otimes n_1}\otimes \cdots \otimes (\C R_k)^{\otimes n_k}\right)\subseteq (\C R)^{\otimes n}.$$ This induced representation is the direct summand of all tensor products which contain the factor $(\C R_i)$ exactly $n_i$ times.
\end{example}

\subsection{Arithmetic sums from trace functions}\label{sec: arith sums from trace fns} Now, we discuss how to obtain the desired arithmetic sums in Sections \ref{sec: intro patterson} and \ref{sec: intro character sums G-extensions} from the trace functions of local systems that we discussed previously.

The arithmetic sums in Section \ref{sec: intro character sums G-extensions} can be obtained by taking the tensor product of the character sheaves constructed in Example \ref{eg: local system character sheaves} together with the pushforward of the constant sheaf from Hurwitz spaces constructed in Example \ref{eg: local system hurwitz space}. This is because the trace function of the tensor product of two sheaves is the product of the individual trace functions. 

On the other hand, obtaining the bias in Gauss sums as given in Section \ref{sec: intro patterson} is more complicated. The corresponding trace function is the product of the indicator function for irreducible (prime) polynomials $\mathbf 1_{\text{irr}} (f)$, and the Gauss sum $G_\chi(f)$. We deal with each one as follows.

Firstly, since $f\in \Conf^n(\F_q)$ is squarefree, \cite[Lemma 3.6]{Sawin_2021_sqroot} tells us that
\begin{equation}\label{eqn: formula for 1_irr}
\mathbf 1_{\text{irr}} (f) = \frac 1 n \sum_{k=0}^{n-1} (-1)^k \tr(\Frob_q,(\mathcal L_{\wedge^k \text{std}_n})_{\bar f})  
\end{equation}
where $\mathcal L_{\wedge^k \text{std}_n}$ are the local systems constructed in Example \ref{eg: local system rep S_n} corresponding to the $k$-th wedge product of the standard representation of $S_n$. 

These representations are direct summands of $V^{\otimes n}$ for a certain braided vector space $V=\C_\wedge$. Following \cite[Example 2.0.5]{ES26}, we define $\C_\wedge$ to be the $2$-dimensional braided vector space of diagonal type with basis $\{v_0,v_1\}$ corresponding to the matrix $(q_{ij})=(-1)^{ij}$ as in Example \ref{eg: BVS}(a). Explicitly, the braiding is
$$R(v_0\otimes v_0)=v_0\otimes v_0, \ R(v_0\otimes v_1)=v_1\otimes v_0,\ R(v_1\otimes v_0)=v_0\otimes v_1,\ R(v_1\otimes v_1)=-v_1\otimes v_1.$$
By \cite[Corollary 5.4.6]{ES26}, we can decompose the $S_n$-representation
\begin{equation}\label{eqn: C_wedge decomp}
\C_\wedge^{\otimes n} = \bigoplus_{k=0}^{n-1} (\wedge^k \text{std}_n)^{\oplus 2},
\end{equation}
so it suffices to look at the braided vector space $\C_\wedge$ to understand $\mathcal L_{\wedge^k \text{std}_n}$ and $\mathbf{1}_{\text{irr}}$.

Secondly, we deal with the Gauss sums $G_\chi(f)$ for a nontrivial multiplicative character $\chi\colon \F_q^\times \rightarrow \C^\times$, following \cite[Section 2]{Sawin_2024}. Recall that we defined the additive character $\psi(x)=e^{2\pi i\tr^{\F_p}_{\F_q}x/p}$. Let $G(\chi,\psi)=\sum_{x\in \F_q^\times} \chi(x)\psi(x)$, which has absolute value $\sqrt q$ by a property of Gauss sums. Then, taking $f_1=1$ in \cite[Lemma 2.4]{Sawin_2024} and combining with \cite[Lemma 2.1, 2.3]{Sawin_2024} gives for squarefree $f$ with degree $n$ that
\begin{equation}\label{eqn: formula for G_chi(f)}
G_\chi(f)=(-1)^{n(n-1)(q-1)/4} G(\chi,\psi)^n (\chi\cdot \xi)\left((-1)^{n(n-1)/2}\disc(f)\right)
\end{equation}
where $\xi\colon \F_q^\times \rightarrow \{\pm 1\}$ is the unique quadratic character and $(\chi\cdot\xi)(x)=\chi(x)\xi(x)$. We see that for fixed $\chi$, the Gauss sum is basically the character $(\chi\cdot \xi)(\disc (f))$ up to a factor with size $q^{\deg f/2}$, so we can understand this using the local system $\delta^* \mathcal L_{\chi\cdot \xi}$ in Example \ref{eg: local system character sheaves}.

At this point, we could basically deduce the arithmetic bounds in Sections \ref{sec: intro patterson} and \ref{sec: intro character sums G-extensions} assuming the corresponding homological vanishing statements. However, we choose to leave this to Section \ref{sec: arithmetic applications} after we have proven the homological vanishing theorems in the next few sections.
\section{Proof of main theorem}\label{sec: proof of main theorem}
In this section we prove Theorem \ref{thm: main extrapolating vanishing}, our main homological vanishing result. The rough idea of the proof is as follows. Suppose we have vanishing for $H_i(B_n,V^{\otimes n})$ in an $(N,I)$ staircase. Using Equation \eqref{eqn: homology to bar-complex}, this translates to $H^i(B_n(\mathfrak A))$ vanishing for all $n\le N$ and $i\le I-N$, where we let $\mathfrak A=\mathfrak A(V^*_{-1})$ be the relevant quantum shuffle algebra.

Then, for any $n>N$, we construct a filtration on the bar-complex $B_n(\mathfrak A)$. We do this in a combinatorial way -- terms in the bar-complex are indexed by partitions, and we determine which filtered piece a term is in based on its corresponding partition. In order to show that the cohomology of the bar-complex vanishes in low degree, it is enough, by a spectral sequence argument, to show this for each graded piece. By our construction, each graded piece of this filtration will be a tensor product of terms like $\mathfrak A_i$ and $B_m(\mathfrak A)$ where $m\le N$. However, by assumption, we know that $B_m(\mathfrak A)$ vanishes in low degree, and this implies the vanishing of low degree cohomology of each graded piece which completes the proof.

The novelty of our proof lies within the combinatorial construction of the filtration, and to the best of our knowledge nothing similar has been done previously at least in the setting of the bar-complexes.

\subsection{A simple example}\label{sec: an example} To illustrate the idea behind our proof, we look at the simplest example of our theorem, which is when $(N,I)=(2,0)$ and $n=5$. As mentioned in the introduction (see Corollary \ref{cor: es main theorem}), the $I=0$ case is basically the same as \cite[Theorem 1.1.1]{ES26}, so this can also serve as a point of comparison between both proofs. Their proof, as written in \cite[Theorem 2.0.11]{ES26}, uses a spectral sequence argument inspired by \cite{EVW16} which comes from homological algebra.

We now outline our proof in this case. The assumption that $H_i(B_n,V^{\otimes n})$ vanishes in the $(2,0)$ staircase corresponds to the condition that $H_0(B_2,V^{\otimes 2})=0$ or equivalently that $H^{-2}(B_2(\mathfrak A))=0$. In this simple case, since $B_2(\mathfrak A)$ is the two term complex $\mathfrak A_1\otimes \mathfrak A_1\rightarrow \mathfrak A_2$ and both terms have the same dimension, we must also have that $H^{-1}(B_2(\mathfrak A))=0$. However, we will \textit{avoid} using this fact as it will not be available to us in the general case.

Theorem \ref{thm: main extrapolating vanishing} then claims for $n=5$ that $H_0(B_5,V^{\otimes 5})=H_1(B_5,V^{\otimes 5})=0$, so we need to show that $H^{-5}(B_5(\mathfrak A))=H^{-4}(B_5(\mathfrak A))=0$. We draw the $5$-th bar-cube in Figure \ref{fig:bar-cube}, where for convenience, we write the partition $(\lambda_1,\ldots ,\lambda_k)$ for the term $\mathfrak A_{\lambda_1}\otimes \cdots \otimes \mathfrak A_{\lambda_k}$. 

\begin{figure}[h]
    \centering
% https://q.uiver.app/#q=WzAsMTYsWzAsMSwiMTExMTEiLFsyNDAsNjAsNjAsMV1dLFsyLDEsIjIxMTEiLFsyNDAsNjAsNjAsMV1dLFswLDMsIjEyMTEiLFsxMjAsNjAsNjAsMV1dLFsyLDMsIjMxMSJdLFszLDIsIjQxIl0sWzEsMiwiMTMxIl0sWzEsMCwiMTEyMSIsWzAsNjAsNjAsMV1dLFszLDAsIjIyMSIsWzAsNjAsNjAsMV1dLFs0LDIsIjExMTIiLFsyNDAsNjAsNjAsMV1dLFs0LDQsIjEyMiIsWzEyMCw2MCw2MCwxXV0sWzYsNCwiMzIiXSxbNSwzLCIxNCJdLFs3LDMsIjUiXSxbNSwxLCIxMTMiXSxbNywxLCIyMyJdLFs2LDIsIjIxMiIsWzI0MCw2MCw2MCwxXV0sWzAsMSwiIiwwLHsiY29sb3VyIjpbMjQwLDYwLDYwXX1dLFswLDJdLFsyLDNdLFswLDZdLFsxLDddLFs2LDcsIiIsMix7ImNvbG91ciI6WzAsNjAsNjBdfV0sWzUsNF0sWzMsNF0sWzIsNV0sWzYsNV0sWzcsNF0sWzEsM10sWzgsOV0sWzE1LDEwXSxbMTMsMTFdLFsxNCwxMl0sWzksMTBdLFsxMSwxMl0sWzEzLDE0XSxbOCwxNSwiIiwxLHsiY29sb3VyIjpbMjQwLDYwLDYwXX1dLFs4LDEzXSxbMTUsMTRdLFsxMCwxMl0sWzksMTFdLFsyLDksIiIsMSx7ImNvbG91ciI6WzEyMCw2MCw2MF19XSxbMywxMF0sWzUsMTFdLFs0LDEyXSxbMSwxNSwiIiwwLHsiY29sb3VyIjpbMjQwLDYwLDYwXX1dLFs2LDEzXSxbMCw4LCIiLDAseyJjb2xvdXIiOlsyNDAsNjAsNjBdfV0sWzcsMTRdXQ==
\[\begin{tikzcd}
	& \textcolor{rgb,255:red,214;green,92;blue,92}{1121} && \textcolor{rgb,255:red,214;green,92;blue,92}{221} &&&& \\
	\textcolor{rgb,255:red,92;green,92;blue,214}{11111} && \textcolor{rgb,255:red,92;green,92;blue,214}{2111} &&& 113 && 23 \\
	& 131 && 41 & \textcolor{rgb,255:red,92;green,92;blue,214}{1112} && \textcolor{rgb,255:red,92;green,92;blue,214}{212} \\
	\textcolor{rgb,255:red,92;green,214;blue,92}{1211} && 311 &&& 14 && 5 \\
	&&&& \textcolor{rgb,255:red,92;green,214;blue,92}{122} && 32
	\arrow[color={rgb,255:red,214;green,92;blue,92}, from=1-2, to=1-4]
	\arrow[from=1-2, to=2-6]
	\arrow[from=1-2, to=3-2]
	\arrow[from=1-4, to=2-8]
	\arrow[from=1-4, to=3-4]
	\arrow[from=2-1, to=1-2]
	\arrow[draw={rgb,255:red,92;green,92;blue,214}, from=2-1, to=2-3]
	\arrow[color={rgb,255:red,92;green,92;blue,214}, from=2-1, to=3-5]
	\arrow[from=2-1, to=4-1]
	\arrow[from=2-3, to=1-4]
	\arrow[color={rgb,255:red,92;green,92;blue,214}, from=2-3, to=3-7]
	\arrow[from=2-3, to=4-3]
	\arrow[from=2-6, to=2-8]
	\arrow[from=2-6, to=4-6]
	\arrow[from=2-8, to=4-8]
	\arrow[from=3-2, to=3-4]
	\arrow[from=3-2, to=4-6]
	\arrow[from=3-4, to=4-8]
	\arrow[from=3-5, to=2-6]
	\arrow[draw={rgb,255:red,92;green,92;blue,214}, from=3-5, to=3-7]
	\arrow[from=3-5, to=5-5]
	\arrow[from=3-7, to=2-8]
	\arrow[from=3-7, to=5-7]
	\arrow[from=4-1, to=3-2]
	\arrow[from=4-1, to=4-3]
	\arrow[color={rgb,255:red,92;green,214;blue,92}, from=4-1, to=5-5]
	\arrow[from=4-3, to=3-4]
	\arrow[from=4-3, to=5-7]
	\arrow[from=4-6, to=4-8]
	\arrow[from=5-5, to=4-6]
	\arrow[from=5-5, to=5-7]
	\arrow[from=5-7, to=4-8]
\end{tikzcd}\]
    \caption{$5$-th bar-cube with coloring for different filtration gradings.}
    \label{fig:bar-cube}
\end{figure}
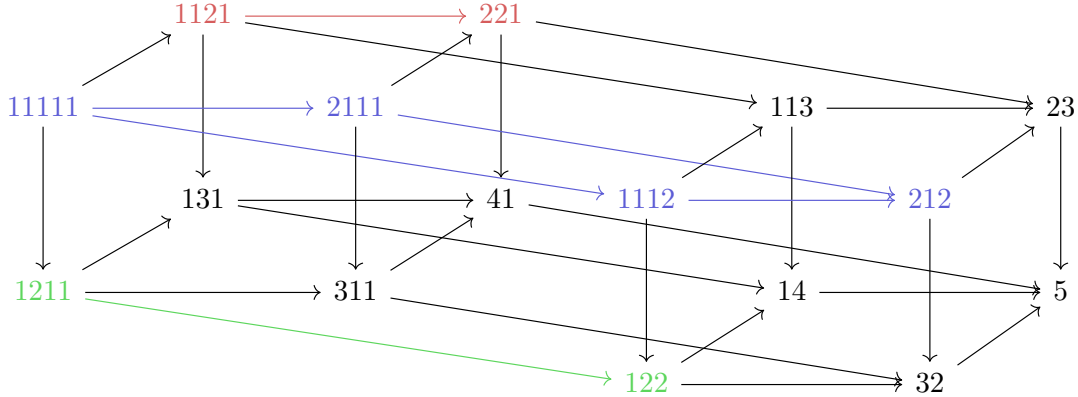
We consider a four-step filtration on $B=B_5(\mathfrak A)$ given by
$$0=F_0B\subseteq F_1B\subseteq F_2B \subseteq F_3B\subseteq F_4B=B$$
where $F_1B$ consists of all the black terms, $F_2B$ the red and black terms, and finally $F_3B$ the red, green and blue terms. By construction, the four graded pieces consists of the blue, green, red and black terms respectively. 

Recall that it suffices to show that each graded piece has cohomology only in degree $\ge -3$. Indeed, from the spectral sequence associated to this filtration
$$E_1^{p,q}=H^{p+q}(\Gr_pB)\Rightarrow H^{p+q}(B),$$
this would tell us that the $E^1$-page only has terms in degree $\ge -3$, so the abutment only has cohomology in degree $\ge -3$. 

This is immediately true for the black graded piece as it only has terms concentrated in degree $\ge -3$. The red graded piece is of the form
$B_2(\mathfrak A)\otimes \mathfrak A_2\otimes \mathfrak A_1$ which has terms in degree $-4$ to $-3$. The cohomology of this is simply $H^*(B_2(\mathfrak A))\otimes \mathfrak A_2 \otimes \mathfrak A_1$, where we normalize $\mathfrak A_i$ to be in degree $-1$ for each $i$. Our assumption tells us that $H^{-2}(B_2(\mathfrak A))=0$, so $H^{-4}(B_2(\mathfrak A)\otimes \mathfrak A_2\otimes \mathfrak A_1)=0$. Thus, the cohomology of the red graded piece is supported in degree $\ge -3$. Likewise, the green graded piece is $\mathfrak A_1\otimes \mathfrak A_2 \otimes B_2(\mathfrak A)$ and by the same argument has cohomology in degree $\ge -3$. Lastly, the blue graded piece is of the form $B_2(\mathfrak A)\otimes \mathfrak A_1\otimes B_2(\mathfrak A)$ whose cohomology is $H^*(B_2(\mathfrak A))\otimes \mathfrak A_1\otimes H^*(B_2(\mathfrak A))$, the first and last terms have cohomology in degree $\ge -1$ by assumption and the middle term has degree $-1$, which verifies the claim for the blue graded piece as well. This completes the argument for $n=5$.

We remark that the filtration described in this section will differ slightly from the one we construct in the next section, since here we have simplified it as much as possible.
\subsection{Multi-block filtration} For the general case, we will need a more complicated filtration which generalizes the one in the previous section. We call this the multi-block filtration because the construction involves breaking $n$ up into blocks of $N+1$, and this is where the quasi-periodicity in Theorem \ref{thm: main extrapolating vanishing} comes from.
\begin{proposition}\label{prop: multi-block filtration}
Let $N\ge 2$ and $n>N$ be integers, and we denote $a=\lfloor \frac{n}{N+1}\rfloor$ and $b=\{ \frac{n}{N+1}\}$. Then, there is a filtration on $B_n(\mathfrak A)$ such that each graded piece is a direct sum of complexes of the form
\begin{equation}\label{eqn: tensor product B_i, A_j}
B_{i_1}(\mathfrak A)\otimes \mathfrak A_{j_1}\otimes B_{i_2}(\mathfrak A)\otimes \mathfrak A_{j_2}\otimes \cdots \otimes \mathfrak A_{j_k}\otimes B_{i_{k+1}}(\mathfrak A),
\end{equation}
where $k\le a$ and $0\le i_l\le N$ for each $1\le l\le k+1$. Furthermore, if $k=a$, then we also have the condition that $i_{a+1}\le b$. Here, $B_0(\mathfrak A)$ denotes the empty tensor product.
\end{proposition}
\begin{proof}
Recall that we can label the terms $\mathfrak A_\lambda\coloneqq \mathfrak A_{\lambda_1}\otimes \cdots \otimes \mathfrak A_{\lambda_k}$ in the bar-cube (or bar-complex) by ordered partitions $\lambda=(\lambda_1,\ldots, \lambda_k)$ of $n$. We can view these $\lambda$ as a subset $S_\lambda \subseteq \{0,\ldots ,n\}$ by looking at the positions of the separators $S_\lambda=\{0,\lambda_1,\lambda_1+\lambda_2,\ldots, \lambda_1+\cdots +\lambda_k\}$. Note that in this definition $S_\lambda$ always includes the ends $0$ and $n$, this turns out to be more convenient for our purposes.

Now, we break $\{0,\ldots, n\}$ into blocks
$$B_1=\{0,\ldots, N\},\  B_2=\{N+1,\ldots, 2N+1\},\ \ldots,\ B_{a+1}=\{a(N+1),\ldots, a(N+1)+b\}$$
where the first $a$ blocks each have length $N+1$ and the last block has length $b+1$. 

For a partition $\lambda$, consider each block $1\le i\le a+1$. If $S_\lambda\cap B_i$ is non-empty, define the minimum and maximum separators in the block to be $m_i=\min(S_\lambda\cap B_i)$ and $M_i=\max(S_\lambda \cap B_i)$. Note that we always have $m_1=0$ and $M_{a+1}=n$. Then, define $L_i=M_i-m_i$ to be the length of the longest interval within the block whose endpoints are separators. If $S_\lambda \cap B_i$ is empty, then write $L_i=-1$ by default. Let the \textit{length} $L(\lambda)$ of the partition $\lambda$ to be the sum of lengths $L_1+\ldots +L_{a+1}$. We illustrate an example of this in Figure \ref{fig: length of partition}.

\begin{figure}[h]
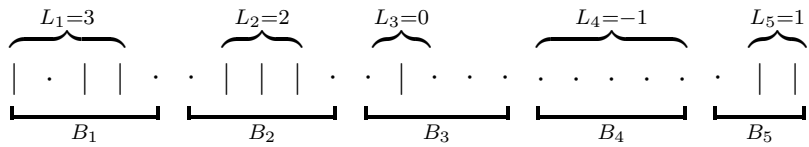

\[
\large
\begin{matrix}\underbracket{\begin{matrix}\overbrace{
   \begin{matrix} \mid &\cdot & \mid & \mid \end{matrix}}^{L_1=3} & \cdot
\end{matrix}}_{B_1}& 
\underbracket{\begin{matrix}\cdot & \overbrace{
   \begin{matrix} \mid & \mid & \mid \end{matrix}}^{L_2=2} & \cdot 
\end{matrix}}_{B_2}&
\underbracket{\begin{matrix}\cdot & \!\!\!\!\!\overbrace{
   \begin{matrix} \mid \end{matrix}}^{L_3 = 0}\!\!\!\!\! & \cdot & \cdot & \cdot
\end{matrix}}_{B_3}&
\underbracket{\begin{matrix}\overbrace{
   \begin{matrix} \cdot & \cdot & \cdot & \cdot & \cdot \end{matrix}}^{L_4=-1}
\end{matrix}}_{B_4}&
\underbracket{\begin{matrix}\cdot & 
\overbrace{
   \begin{matrix} \mid  & \mid \end{matrix}}^{L_5=1} 
\end{matrix}}_{B_5}
\end{matrix}\]
 \caption{The partition $\lambda=(2,1,3,1,1,3,10,1)$ for $n=22, N=4$ with $a=5,b=2$ and length $L(\lambda)=3+2+0-1+1=5$, where $\mid$ denotes a separator and $\cdot$ denotes the absence of a separator.}
    \label{fig: length of partition}
\end{figure}
Now, consider an increasing filtration on $B=B_n(\mathfrak A)$
$$0=F_{-a}B\subseteq F_{-a+1}B\subseteq \cdots \subseteq F_{aN+b}B=B$$
where $F_lB$ is the sub-complex of $B$ consisting of exactly the terms $\mathfrak A_\lambda$ where $L(\lambda)\le l$. Since maps in the bar-complex send $\mathfrak A_\lambda$ to $\mathfrak A_{\lambda'}$ if and only if $S_{\lambda'}$ is obtained by removing exactly one element from $S_{\lambda}$, we see that $L(\lambda')\le L(\lambda)$ so $F_iB$ are indeed sub-complexes of $B$. 

The $l$-th graded piece of this filtration $\Gr_lB$ consists of the terms of length exactly $l$. There are many possible choices of $L_i$, together with $m_i,M_i$ (for those $i$ where $L_i\neq -1$) that have length $\sum L_i=l$. Now let us fix a choice of values $(m_i,M_i,L_i)$, and we want to know what $\lambda$ can have these values. Inside each block $B_i$, if $L_i=-1$, we are forced to have no separators in the block, i.e. $S_\lambda\cap B_i=\emptyset$. If $L_i\ge 0$, then $m_i$ and $M_i$ must be separators, anything smaller than $m_i$ and bigger than $M_i$ cannot be separators, and anything in between can either be separators or not. More precisely, we have the solutions $S_\lambda\cap B_i = \{m_i\} \cup T \cup \{M_i\}$ for any subset $T\subseteq \{m_i+1,\cdots ,M_i-1\}$. It is then clear that all such terms $\mathfrak A_\lambda$ that give rise to this $(m_i,M_i,L_i)$ form a complex of the form of Equation \eqref{eqn: tensor product B_i, A_j}. Indeed, each block with $L_i\ge 0$ has an interval $[m_i,M_i]$ which corresponds to a $B_{L_i}(\mathfrak A)$ term, and in between these intervals we will have $\mathfrak A_j$ of the appropriate size. The condition that $0\le i_l\le N$ follows from $L_i\le N$, and if $k=a$, $i_{a+1}\le b$ follows from $L_{a+1}\le b$. 

We illustrate an example of the possible set of partitions $\lambda$ in Figure \ref{fig: possible lambdas}, and here the corresponding terms in the bar-complex $\mathfrak A_\lambda$ form the complex $B_3(\mathfrak A)\otimes \mathfrak A_3\otimes B_2(\mathfrak A)\otimes \mathfrak A_3 \otimes B_0(\mathfrak A)\otimes \mathfrak A_{10}\otimes B_1(\mathfrak A)$.

\begin{figure}[h]
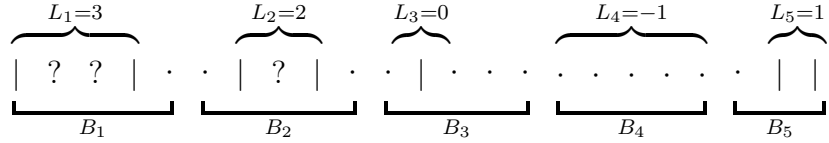

\[
\large
\begin{matrix}\underbracket{\begin{matrix}\overbrace{
   \begin{matrix} \mid &? & ? & \mid \end{matrix}}^{L_1=3} & \cdot
\end{matrix}}_{B_1}& 
\underbracket{\begin{matrix}\cdot & \overbrace{
   \begin{matrix} \mid & ? & \mid \end{matrix}}^{L_2=2} & \cdot 
\end{matrix}}_{B_2}&
\underbracket{\begin{matrix}\cdot & \!\!\!\!\!\overbrace{
   \begin{matrix} \mid \end{matrix}}^{L_3 = 0}\!\!\!\!\! & \cdot & \cdot & \cdot
\end{matrix}}_{B_3}&
\underbracket{\begin{matrix}\overbrace{
   \begin{matrix} \cdot & \cdot & \cdot & \cdot & \cdot \end{matrix}}^{L_4=-1}
\end{matrix}}_{B_4}&
\underbracket{\begin{matrix}\cdot & 
\overbrace{
   \begin{matrix} \mid  & \mid \end{matrix}}^{L_5=1} 
\end{matrix}}_{B_5}
\end{matrix}\]
 \caption{Possible partition $\lambda$s for a fixed choice of $(m_i,M_i,L_i)$ coming from the example in Figure \ref{fig: length of partition}, where $?$ denotes either the presence or absence of a separator.}
    \label{fig: possible lambdas}
\end{figure}
Hence, we have grouped the terms in $\Gr_l B$ into complexes of the form of Equation \eqref{eqn: tensor product B_i, A_j}, where each complex corresponds to some choice of $(m_i,M_i,L_i)$. We claim that $\Gr_l B$ is actually the direct sum of these complexes. To show this, it suffices to check that there are no maps between these complexes, i.e. there are no maps $\mathfrak A_\lambda\rightarrow \mathfrak A_{\lambda'}$ where $\lambda$ and $\lambda'$ have distinct values $(m_i,M_i,L_i)$ and $(m_i',M_i',L_i')$. This is not so hard to check as we just need to show that $S_{\lambda}$ does not contain $S_{\lambda'}$. Indeed, if there is some $L_i'>L_i$, then this must be true because the interval $[m_i',M_i']$ is longer than $[m_i,M_i]$ so either one of $m_i',M_i'\in S_{\lambda'}$ is not in $[m_i,M_i]$ and thus not in $S_\lambda$. Otherwise, if $L_i'=L_i$ for all $i$, there must be some $i$ where $[m_i',M_i']$ and $[m_i,M_i]$ are different, but since they are of the same length we again have either $m_i'\not\in S_\lambda$ or $M_i'\not\in S_{\lambda}$. This finishes the proof.
\end{proof}
\subsection{Finishing the proof} We then apply the multi-block filtration to finish the proof of our main theorem. 
\begin{proof}[Proof of Theorem \ref{thm: main extrapolating vanishing}]
Assume that $H_i(B_n,V^{\otimes n})$ vanishes in an $(N,I)$ staircase. We want to show that for all $n$, the homology vanishes when $i\le f(n)\coloneqq (I+1)\left\lfloor \frac{n+1}{N+1}\right\rfloor+ \max\left(I+\left\{ \frac{n+1}{N+1}\right\}-N\ ,\ 0\right)-1$. Note that this follows immediately when $n\le N$ so we are left to consider the case of $n>N$. As in Section \ref{sec: an example}, we translate this to the bar-complex via Equation \eqref{eqn: homology to bar-complex}. Our assumption is now equivalent to $H^i(B_n(\mathfrak A))$ vanishing in degrees $i\le I-N$ for $n\le N$, and we want to prove that $H^i(B_n(\mathfrak A))$ vanishes for $i\le f(n)-n$ for $n>N$.

Now fix some integer $n>N$. Consider the multi-block filtration $F_lB$ of $B=B_n(\mathfrak A)$ constructed in Proposition \ref{prop: multi-block filtration}. Recall from Section \ref{sec: an example} that by a spectral sequence argument it suffices to show that the cohomology of each graded piece vanishes when $i\le f(n)$, and since each graded piece is the direct sum of the complexes of the form of Equation \eqref{eqn: tensor product B_i, A_j} it suffices to show this for each direct summand of this form. 

The cohomology of the tensor product Equation \eqref{eqn: tensor product B_i, A_j} is the tensor product of the cohomology of each term. Our assumption tells us that each $B_{i_l}(\mathfrak A)$ has nonzero cohomology only in degree $i\ge \max(I-N+1,-i_l)$. Also recall that each $\mathfrak A_{j_l}$ lives in degree $-1$. Adding up these degrees, we see that the cohomology of the tensor product has nonzero cohomology only for $i\ge \lfloor \frac n {N+1}\rfloor(I-N)+\max(I-N+1,-\{\frac n {N+1}\})$ where we used the fact that $i_{a+1}\le b$. We rewrite the RHS as
$(I+1)\lfloor\frac{n}{N+1}\rfloor+\max(I+\{\frac{n}{N+1}\}-N+1,0)-n=f(n)+1-n$. This means that the cohomology vanishes for $i\le f(n)-n$, finishing the proof.
\end{proof}
\section{Two methods for homological vanishing}\label{sec: criteria} Let $V$ be a braided vector space. As mentioned in the introduction, in order to get homological vanishing results in families, we need a uniform way to prove some vanishing staircase of homology. One way to do this, as demonstrated in \cite{ES26}, is to prove vanishing in a minimal $(N,0)$ staircase, which is not difficult as it only requires understanding when $H^0$ vanishes. In this section, we give two more methods to detect homological vanishing, which can give much stronger vanishing results in families.
\subsection{From isomorphisms of shuffle products} Firstly, we give a necessary and sufficient condition for vanishing in a $(N,N-2)$ staircase in terms of shuffle products of the quantum shuffle algebra. This would then give homological vanishing of slope $\frac{N-1}{N+1}$ by Theorem \ref{thm: main extrapolating vanishing}. Let $\mathfrak A=\mathfrak A(V^*_{-1})$ be the relevant quantum shuffle algebra for $V$ by Equation \eqref{eqn: homology to bar-complex}.
\begin{proposition}\label{prop: vanishing from shuffle product}
Let $N\ge 2$ be an integer. The following are equivalent.
\begin{enumerate}[(\alph*)]
    \item $H_i(B_n,V^{\otimes n})$ vanishes in the $(N,N-2)$ staircase.
    \item For all $2\le n\le N$, we have $H_i(B_n,V^{\otimes n})$ vanishes for all $i$.
    \item For all $2\le n\le N$ and $p+q=n$, the shuffle product $\tau_{p,q}\colon \mathfrak A_p\otimes \mathfrak A_q\rightarrow \mathfrak A_{m}$ is an isomorphism.
    \item For all $2\le n\le N$, there exists some $p+q=n$ where the shuffle product $\tau_{p,q}\colon \mathfrak A_p\otimes \mathfrak A_q\rightarrow \mathfrak A_{m}$ is an isomorphism.
\end{enumerate}
\end{proposition}

The proof will follow the same style as the proof of the main theorem, but it will be much simpler. The crux of the proof is to deduce homological vanishing (b) from some isomorphism of shuffle products (d). The idea is to use a filtration to break up the $N$-th bar-cube into various graded pieces, each consisting of two terms connected by an arrow in a fixed direction.

For example, consider the $4$-th bar-cube broken down into graded pieces each corresponding to a red vertical arrows in Figure \ref{fig: 4th bar-cube}. Then, if we knew that $\tau_{2,2}\colon \mathfrak A_2\otimes \mathfrak A_2\rightarrow \mathfrak A_4$ is an isomorphism, we can conclude by induction on $N$ that the other graded pieces are also isomorphisms. Indeed, they can be obtained by tensoring $\tau_{p,q}\colon \mathfrak A_p\otimes \mathfrak A_q\rightarrow \mathfrak A_m$ for $p+q=m<4$ (which we know are isomorphisms by (c) due to induction) with terms of the form $\mathfrak A_j$. This tells us that the homology vanishes.

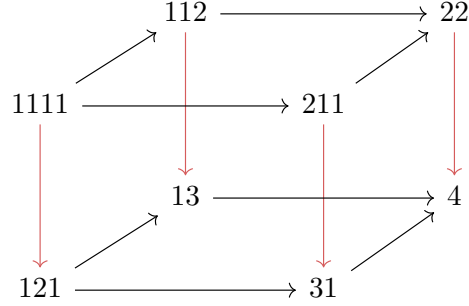
\begin{figure}[h]
    \centering
% https://q.uiver.app/#q=WzAsOCxbMCwxLCIxMTExIl0sWzIsMSwiMjExIl0sWzAsMywiMTIxIl0sWzIsMywiMzEiXSxbMywyLCI0Il0sWzEsMiwiMTMiXSxbMSwwLCIxMTIiXSxbMywwLCIyMiJdLFswLDFdLFswLDIsIiIsMCx7ImNvbG91ciI6WzAsNjAsNjBdfV0sWzIsM10sWzAsNl0sWzEsN10sWzYsN10sWzUsNF0sWzMsNF0sWzIsNV0sWzYsNSwiIiwwLHsiY29sb3VyIjpbMCw2MCw2MF19XSxbNyw0LCIiLDAseyJjb2xvdXIiOlswLDYwLDYwXX1dLFsxLDMsIiIsMCx7ImNvbG91ciI6WzAsNjAsNjBdfV1d
\[\begin{tikzcd}
	& 112 && 22 \\
	1111 && 211 \\
	& 13 && 4 \\
	121 && 31
	\arrow[from=1-2, to=1-4]
	\arrow[color={rgb,255:red,214;green,92;blue,92}, from=1-2, to=3-2]
	\arrow[color={rgb,255:red,214;green,92;blue,92}, from=1-4, to=3-4]
	\arrow[from=2-1, to=1-2]
	\arrow[from=2-1, to=2-3]
	\arrow[color={rgb,255:red,214;green,92;blue,92}, from=2-1, to=4-1]
	\arrow[from=2-3, to=1-4]
	\arrow[color={rgb,255:red,214;green,92;blue,92}, from=2-3, to=4-3]
	\arrow[from=3-2, to=3-4]
	\arrow[from=4-1, to=3-2]
	\arrow[from=4-1, to=4-3]
	\arrow[from=4-3, to=3-4]
\end{tikzcd}\]
\caption{$4$-th bar-cube with vertical graded pieces. }\label{fig: 4th bar-cube}
\end{figure}
We make this precise in the proof below.
\begin{proof}[Proof of Proposition \ref{prop: vanishing from shuffle product}]
We proceed by induction on $N$. Assuming the result for all smaller values of $N$, we now prove it for $N$. Note that $(b\Rightarrow a)$ and $ (c\Rightarrow d)$ are clear from the definitions. $(a\Rightarrow b)$ follows from the fact that the Euler characteristic of the bar-complex $B_N(\mathfrak A)$ is zero, so vanishing in $H_i(B_N,V^{\otimes N})$ for $0\le i\le N-2$ implies vanishing for $H_{N-1}(B_N,V^{\otimes N})$ as well.

We now prove $(d\Rightarrow b)$. Let $p,q$ be the isomorphism $\mathfrak A_p\otimes \mathfrak A_q \rightarrow \mathfrak A_N$ as assumed in (d). Recall that each partition $\lambda$ corresponds to a set of separators $S_\lambda\subseteq \{0,\ldots, N\}$ as in the proof of Proposition \ref{prop: multi-block filtration}. Consider the increasing filtration on $B=B_N(\mathfrak A)$
$$0=F_0B\subseteq F_1B \subseteq \cdots \subseteq  F_{N-1}B=B$$
where $F_lB$ is the sub-complex consisting of the terms $\mathfrak A_\lambda$ where $|S_\lambda \cap (\{1,\ldots, \hat p,\ldots, N-1\}|\le l$, where $\hat p$ denotes the omission of the element $p$. Then, the graded piece $\Gr_lB$ then consists of exactly those $\mathfrak A_\lambda$ where we have equality above. 

We claim that $\Gr_lB$ is a direct summand of the two-term complexes $\mathfrak A_{\lambda_1(T)} \rightarrow \mathfrak A_{\lambda_2(T)}$ where $T$ ranges over size $l$ subsets of $\{1,\ldots, \hat p, \ldots ,N-1\}$, and $\lambda_1(T),\lambda_2(T)$ are defined by $S_{\lambda_1(T)}=T\sqcup\{0,p,N\}$, $S_{\lambda_2(T)}=T\sqcup\{0,N\}$. Indeed, it is clear that these terms are exactly those with $|S_{\lambda}\cap \{1,\ldots, \hat p, \ldots, N-1\}|=l$, so we just need to check that there are no maps between different two-term complexes. This is clear as $S_{\lambda_i(T)}$ cannot contain $S_{\lambda_j(T')}$ when $T\neq T'$, because that would mean $T$ contains $T'$ which is impossible as $T,T'$ are different sets with the same size.

Hence, to show (b), it suffices to show that all these two-term complexes are isomorphisms. If $T=\emptyset$, we have $\tau_{p,q}\colon\mathfrak A_p\otimes \mathfrak A_{q}\rightarrow \mathfrak A_N$ which is an isomorphism by condition (d). Otherwise, write $T\sqcup\{0,p,N\}=\{0,i_1,i_1+i_2,\ldots, i_1+\cdots +i_k\}$, with $p=i_1+\ldots+i_j$, then the two-term complex is exactly $$\id \otimes \cdots \otimes \tau_{i_j,i_{j+1}}\otimes \cdots\otimes \id\colon \ \mathfrak A_{i_1}\otimes \cdots \otimes \mathfrak A_{i_j}\otimes \mathfrak A_{i_{j+1}}\otimes \cdots \mathfrak A_{i_k}\rightarrow \mathfrak A_{i_1}\otimes \cdots \otimes \mathfrak A_{i_j+i_{j+1}}\otimes \cdots \otimes \mathfrak A_{i_k}.$$
This is an isomorphism becaus $\tau_{i_j,i_{j+1}}$ is an isomorphism as (c) is true for $N-1$ by the induction hypothesis, and $i_j+i_{j+1}<N$.

We are left to prove $(b\Rightarrow c)$. Fix any $p+q=N$, and we want to show $\tau_{p,q}\colon\mathfrak A_p\otimes \mathfrak A_{q}\rightarrow \mathfrak A_N$ is an isomorphism. Consider the same filtration on $B=B_N(\mathfrak A)$, and by the same argument above, it follows from the induction hypothesis on (c) that all two-term complexes with $T\neq \emptyset$ are isomorphisms. Thus, the cohomology of $B$ is simply the cohomology of the last remaining graded piece $\tau_{p,q}\colon\mathfrak A_p\otimes \mathfrak A_{q}\rightarrow \mathfrak A_N$, and vanishing of the former implies vanishing of the latter, i.e. that $\tau_{p,q}$ is an isomorphism.
\end{proof}

\subsection{From the action of the fundamental braid} Another way to prove homological vanishing of a certain slope is to show that $H_i(B_n,V^{\otimes n})$ vanishes for all $i$ for a positive density $\rho$ of $n\in \N$. It is not always possible to prove such strong vanishing results, so this technique will only apply for certain cases. This would lead to homological vanishing of slope at least $\rho$, due to the following corollary of our main theorem.
\begin{corollary}\label{cor: density}
Let $S\subseteq \N$ be the set of all $n$ where $H_i(B_n,V^{\otimes n})\neq 0$ for some $i$. Let the elements of $S$ be $1=n_1<n_2<\cdots$, then we have $H_i(B_{n_j},V^{\otimes n_j})=0$ for $i\le n_j-j-1$.
\end{corollary}
In words, this is telling us that for every $n$ where homology doesn't vanish, there can be at most one more homological degree where homology is non-vanishing.
\begin{proof}
We prove this by induction on $j$. The $j=1$ case is vacuous as $n_j-j-1<0$. Now suppose we know the statement for smaller $j'<j$ and we want to prove $H_i(B_{n_j},V^{\otimes n_j})=0$ for $i\le n_j-j-1$. The induction hypothesis tells us that $H_i(B_n,V^{\otimes n})$ vanishes in a $(n_j-1,n_j-j-1)$ staircase, indeed, by definition, it suffices to check that $H_i(B_{n_{j'}},V^{\otimes n_{j'}})=0$ for $i\le (n_j-j-1)+n_{j'}-(n_j-1)=n_{j'}-j\le  n_{j'}-j'-1$, which is exactly what we assumed. Applying Theorem \ref{thm: main extrapolating vanishing} for $n=n_j$ gives the desired result.
\end{proof}

We remark that we can often get better vanishing results if we use more information on the set $S$. To see this, consider the toy example where $S=\{1,10,12,14,\ldots\}$. Then Corollary \ref{cor: density} tells us that this vanishes with slope $\frac 1 2$, but simply applying vanishing in a $(9,7)$ staircase gives $\frac 4 5$ vanishing slope. We will see more examples in Section \ref{sec: individual C_wedge} where we use ad-hoc applications of Theorem \ref{thm: main extrapolating vanishing} to get a better bound for individual braided vector spaces.

We will prove such vanishing by looking at the action of the fundamental braid on group cohomology. The fundamental braid is defined to be $\Delta_n=\prod_{k=1}^{n-1}(\sigma_1\cdots\sigma_{n-k})$ and its square is $\Delta_n^2=(\sigma_1\cdots\sigma_{n-1})^n$. The center of the braid group $Z(B_n)$ is cyclic and generated by the square of the fundamental braid $\Delta_n^2$ when $n>2$ and the fundamental braid $\Delta_2$ when $n=2$ \cite{Garside_1969}. We have the following proposition.
\begin{proposition}\label{prop: vanishing from fundamental braid}
Let $n\ge 2$ be an integer. Consider a power of the fundamental braid $\Delta_n^k$, where we require $k$ to be even if $n\ge 3$, and suppose that it acts on $V^{\otimes n}$ by a nontrivial scalar $\lambda\neq 1$. Then, the homology $H_i(B_n,V^{\otimes n})$ vanishes for all $i$. 
\end{proposition}

This will follow from a standard argument in group cohomology. Combined with Corollary \ref{cor: density}, this gives us a way to prove vanishing slope in families. We also remark that the statement is also true when $V^{\otimes n}$ is replaced by another $B_n$-representation.
\begin{proof}
We first recall some facts from group cohomology. Let $M$ be a $G$-module. The group homology $H_i(G,M)$ is the homology of a chain complex $C_i(G,M)=M\otimes \Z[G^i]$, where we write an element of $C_i(G,M)$ as $m[g_1|\cdots|g_i]$. The group $G$ acts on the chain complex by conjugation, more precisely, $h\in G$ acts on $C_i(G,M)$ by sending $m[g_1|\cdots|g_i]$ to $hm[hg_1h^{-1}|\cdots|hg_ih^{-1}]$. This induces an action of $G$ on $H_*(G,M)$ by conjugation, and it is a fact \cite[Proposition III.8.1]{Brown_1982} that this action is trivial, and this is proven by constructing an explicit chain homotopy between the action of $h$ and the identity. 

Now we specialize to our case where $G=B_n$ and $M=V^{\otimes n}$ and consider the action of $\Delta^k_n$ on $C_i(B_n,V^{\otimes n})$. Recall our assumption that $\Delta^k_n$ acts on $V^{\otimes n}$ by a scalar $\lambda\neq 1$, so this action sends $m[g_1|\cdots |g_i]$ to $\lambda m[\Delta^k_n g_1\Delta^{-k}_n|\cdots|\Delta^k_ng_i\Delta^{-k}_n]=\lambda m[g_1|\cdots |g_i]$ since $\Delta^k_n$ is in the center of the braid group. Hence, $\Delta^k_n$ acts on $H_i(B_n,V^{\otimes n})$ by $\lambda$ for all $i$. However, we also know that this action is trivial, so $\lambda z=z$ for any element $z\in H_i(B_n,V^{\otimes n})$, which forces $H_i(B_n,V^{\otimes n})=0$ for all $i$.
\end{proof}
\section{Twists of braided vector spaces}\label{sec: twists of BVS} 
We now study the family of twists of a braided vector space $V$. In this section, we will study the most general case where there are no assumptions on $V$, this will by done using the first method of analyzing isomorphisms of shuffle products as in Proposition \ref{prop: vanishing from shuffle product}. In the next section, we will study cases where there are additional restrictions imposed on $V$, which will lead to stronger vanishing results.

Recall that the twist $V_\zeta$ of $V$ by $\zeta\in \C^\times$ is defined by the braiding $R_{V_\zeta}=\zeta \cdot R_V$. We do not assume that $\zeta$ is a root of unity. Let $\mathcal F_V$ be the family of all twists of $V$, and we define
$$\mathcal F_{V,N}=\{V_\zeta\in \mathcal F_V \colon H_i(B_n,V_\zeta^{\otimes n})\text{ vanishes in a }(N,N-2)\text{ staircase}\}.$$
We think of this as a ``nice set'', as for $V_\zeta \in \mathcal F_{V,N}$, applying Theorem \ref{thm: main extrapolating vanishing} tells us that $H_i(B_n,V_\zeta^{\otimes n})$ vanishes with slope at least $\frac{N-1}{N+1}$.

We will prove vanishing in a maximal staircase of $(N,N-2)$ for all but finitely many twists in $\mathcal F_V$, i.e. $\mathcal F_V\setminus \mathcal F_{V,N}$ is finite for all $N$. As a consequence, for any $\epsilon >0$, all but finitely many twists in $\mathcal F_V$ vanish with slope $>1-\epsilon$. It turns out that the exceptional twists in $\mathcal F_V\setminus \mathcal F_{V,N}$ can be listed from a computation, and we give examples of this in Section \ref{sec: computation examples}. 

\subsection{General case}\label{sec: twist general case} Let the braided vector space $V$ have dimension $d$. We prove the following.
\begin{proposition}\label{prop: general case twist}
Let $N\ge 2$ be an integer. Then, the set of exceptions $\mathcal F_V\setminus \mathcal F_{V,N}$ has cardinality less than $N^2 d^N$, and its elements can be computed explicitly. Furthermore, if $V$ is defined over a number field $K$ and $[K(\zeta)\colon K]>(N-1)d^N$, then $V_\zeta\in \mathcal F_{V,n}$. 
\end{proposition}

We remark that the size of $\mathcal F_V\setminus \mathcal F_{V,N}$ is often much smaller than the bound given above, we will see this in the examples given in Section \ref{sec: computation examples}. We also note that if $\zeta$ is a root of unity and $V$ is defined over $\Q$, the second statement will tell us that $\phi(\ord(\zeta))>(N-1)d^N$ implies $V_\zeta \in \mathcal F_{V,N}$.
\begin{proof}
We will apply Proposition \ref{prop: vanishing from shuffle product} in this setting and deduce the vanishing from isomorphism of shuffle products. By Equation \eqref{eqn: homology to bar-complex}, the relevant quantum shuffle algebra for the homology of $V_{\zeta}$ is $\mathfrak A^\zeta\coloneqq \mathfrak A(V^*_{-\bar\zeta})$ where we write $V^*_{-\bar\zeta}=(V^*)_{-\bar\zeta}$. We choose a basis $\{v_1,\ldots , v_d\}$ of $V^*$. If $V$ is defined over a number field $K\subseteq \C$, then $V^*$ is defined over the complex conjugate $c(K)$. We choose a basis of $V^*$ so that the braiding is defined over $c(K)$. Then, as a vector space, $V_{-\bar\zeta}^{*\otimes m}$ has a basis of words of the form $v_{i_1}\ldots v_{i_m}$. 

Note that for $p+q=n$, we have $\mathfrak A^\zeta_p\otimes \mathfrak A^\zeta_q\cong V_{-\bar\zeta}^{*\otimes p}\otimes V_{-\bar\zeta}^{*\otimes q}\cong V_{-\bar\zeta}^{*\otimes n}$ and $\mathfrak A^\zeta_n\cong V_{-\bar\zeta}^{*\otimes n}$, so we can use the basis of words earlier in lexicographical order to view the shuffle product $\tau_{p,q}^\zeta \colon \mathfrak A^\zeta_p\otimes \mathfrak A^\zeta_{q}\rightarrow \mathfrak A^\zeta_n$ as a $d^n\times d^n$ matrix. Let us call this matrix $M^\zeta_{p,q}$, this has coefficients which are polynomials in $\bar \zeta$. The shuffle product $\tau^\zeta_{p,q}$ is an isomorphism if and only if $M^\zeta_{p,q}$ is invertible, which happens if and only if $f_{p,q}(\bar \zeta)\coloneqq \det(M^\zeta_{p,q})$ is nonzero. 

Now, the equivalence between (a) and (d) of Proposition \ref{prop: vanishing from shuffle product} immediately tells us the following.

\begin{lemma}\label{lem: P_N}
We have $V_\zeta\in \mathcal F_V\setminus \mathcal F_{V,N}$ if and only if $\bar\zeta$ is a root of $$P_N(\bar\zeta)=\prod_{n=2}^N \gcd(f_{1,n-1},\ldots, f_{n-1,1}).$$
\end{lemma}
\begin{proof}
By Proposition \ref{prop: vanishing from shuffle product}, $V_\zeta\in \mathcal F_{V,N}$ if and only if for each $2\le n\le N$, there is some $p+q=n$ where $f_{p,q}(\bar\zeta)\neq 0$. The contrapositive says that $V_\zeta\in \mathcal F_V\setminus \mathcal F_{V,N}$ if and only if there is some $2\le n\le N$ where $f_{p,q}(\bar \zeta)=0$ for all $p+q=n$, and this is equivalent to the condition in the lemma.
\end{proof}

Hence, the proof will boil down to the following properties of $f_{p,q}(\bar \zeta)$.
\begin{lemma}\label{lem: f}
$f_{p,q}(\bar \zeta)$ is a nonzero polynomial in $\bar \zeta$ of degree $\le pqd^n$. Furthermore, if $V$ is defined over $K$, the coefficients of $f_{p,q}(\bar \zeta)$ are in the complex conjugate $c(K)$.
\end{lemma}
\begin{proof}
Recall that the shuffle product can be written as
\begin{equation*}
\begin{split}
\tau_{p,q}^\zeta(v_{i_1}\cdots v_{i_p}\otimes v_{i_{p+1}}\cdots v_{i_n}) &= \sum_{\sigma\in \text{Sh}(p,q)} R_{V^*_{-\bar\zeta},\sigma}(v_{i_1}\cdots v_{i_p}\otimes v_{i_{p+1}}\cdots v_{i_n}) \\ 
&=\sum_{\sigma\in \text{Sh}(p,q)} \bar\zeta^{\inv(\sigma)}R_{V^*_{-1},\sigma}(v_{i_1}\cdots v_{i_p}\otimes v_{i_{p+1}}\cdots v_{i_n}),
\end{split}
\end{equation*}
where $\inv(\sigma)$ denotes the number of inversions in the shuffle $\sigma$. From this, we see that constant terms only appear when $\inv(\sigma)=0$, i.e. $\sigma=\id$. In this case, $R_{V^*_{-1},\id}(v_{i_1}\cdots v_{i_p}\otimes v_{i_{p+1}}\cdots v_{i_n})=v_{i_1}\cdots v_{i_n}$, so the only constant terms in the matrix $M^\zeta_{p,q}$ are the $1$s in the diagonal, so $f_{p,q}(\bar \zeta)$ has constant coefficient $1$ and thus is nonzero. 

The degree of each entry of $M^\zeta_{p,q}$ as a polynomial in $\bar \zeta$ is at most $pq$ because $\inv(\sigma)\le pq$ with equality achieved by the shuffle which sends the first $p$ entries to the last $p$ entries, and the last $q$ entries to the first $q$. This implies that the determinant has degree at most $pqd^m$. Lastly, if the basis is defined over $c(K)$, then each $R_{V^*_{-1},\sigma}(v_{i_1}\cdots v_{i_p}\otimes v_{i_{p+1}}\cdots v_{i_n})$ is a $c(K)$-linear combination of words, so each entry of $M^\zeta_{p,q}$ is a polynomial in $\bar \zeta$ with coefficients in $c(K)$, hence the same is true for $f_{p,q}(\bar \zeta)$.
\end{proof}

We can now complete the proof. The first statement follows from Lemma \ref{lem: P_N} because $\deg(P_N)\le \sum_{n=2}^N\deg(f_{1,{n-1}})\le \sum_{n=2}^N (n-1)d^n< N^2d^N$ by Lemma \ref{lem: f}, and it also follows immediately that the set $\mathcal F_V\setminus \mathcal F_{V,N}$ is explicitly computable. Lemma \ref{lem: P_N} reduces the second statement to showing that if $[K(\zeta)\colon K]=[c(K)(\bar\zeta)\colon c(K)]>(N-1)d^N$, then $\bar\zeta$ is not a root of $P(\bar\zeta)$. Indeed, if $\bar\zeta$ were a root of $P(\bar\zeta)$, then it must be a root of some $f_{1,n-1}(\bar\zeta)$. Since this has coefficients in $c(K)$, we have $[c(K)(\bar\zeta)\colon c(K)]\le \deg(f_{1,n-1})\le (n-1)d^n\le (N-1)d^N$ by Lemma \ref{lem: f}, yielding a contradiction.
\end{proof}
\subsection{Examples}\label{sec: computation examples}
We now demonstrate how to compute the set of exceptions $\mathcal F_V\setminus \mathcal F_{V,N}$ for two examples, which are relevant to the arithmetic situations of character sums over Galois $G$-extensions and Gauss sums respectively.
\begin{example}\label{eg: computation S3}
Let $G=S_3$ and $R=\{a,b,c\}$ be the rack of conjugacy class of transpositions where $a=(12),b=(13),c=(23)$, and let $V=\C R$ be the corresponding braided vector space of rack-type. One can check that the dual braided vector space $V^*$ has the same braiding as $V$, so $V^*_{-\bar\zeta}$ has braiding given by $R_{V^*_{-\bar\zeta}}(a\otimes b)=-\bar\zeta \ b\otimes a^b$. We write the matrix $M^\zeta_{1,1}$ with basis in lexicographical order $aa,ab,\ldots, cc$ as follows:
\begin{equation*}\tiny
M^\zeta_{1,1}=
\begin{pmatrix}
1-\bar\zeta & 0 & 0 & 0 & 0 & 0 & 0 & 0 & 0 \\
0 & 1 & 0 & 0 & 0 & 0 & -\bar\zeta & 0 & 0 \\
0 & 0 & 1 & -\bar\zeta & 0 & 0 & 0 & 0 & 0 \\
0 & 0 & 0 & 1 & 0 & 0 & 0 & -\bar\zeta & 0 \\
0 & 0 & 0 & 0 & 1-\bar\zeta & 0 & 0 & 0 & 0 \\
0 & -\bar\zeta & 0 & 0 & 0 & 1 & 0 & 0 & 0 \\
0 & 0 & 0 & 0 & 0 & -\bar\zeta & 1 & 0 & 0 \\
0 & 0 & -\bar\zeta & 0 & 0 & 0 & 0 & 1 & 0 \\
0 & 0 & 0 & 0 & 0 & 0 & 0 & 0 & 1-\bar\zeta
\end{pmatrix}.
\end{equation*}
We can decompose this into two $3\times 3$ cyclic blocks and three $1\times 1$ blocks of $(1-\bar\zeta)$, so we can compute the determinant to be
$$P_2(\bar\zeta)=f_{1,1}(\bar\zeta)=(1-\bar\zeta)^5(1+\bar\zeta+\bar\zeta^2)^2=-\Phi_1(\bar\zeta)^5\Phi_3(\bar\zeta)^2,$$
where $\Phi_n(x)$ denotes the $n$-th cyclotomic polynomial. 

Hence, by Lemma \ref{lem: P_N}, we have that $\mathcal F_V\setminus \mathcal F_{V,2}=\{V,V_{\zeta_3},V_{\zeta_3^2}\}$ consists of exactly the twists with $\ord(\zeta)=1,3$. Theorem \ref{thm: main extrapolating vanishing} then tells us that all other twists apart from these are guaranteed to have vanishing slope least $\frac 1 3$. For these three exceptions, $V$ exhibits homological stability from \cite{EVW16} and thus has no homological vanishing of any slope. For the other two twists with $\ord(\zeta)=3$, we can compute the cohomology for small $n$ which proves a vanishing slope of at least $\frac 1 2$, as seen in Figure \ref{fig: S_3 trans 3rd root}. Thus, we see that being in $\mathcal F_{V,N}$ is not a necessary condition for a twist to vanish with slope $\frac{N-1}{N+1}$.
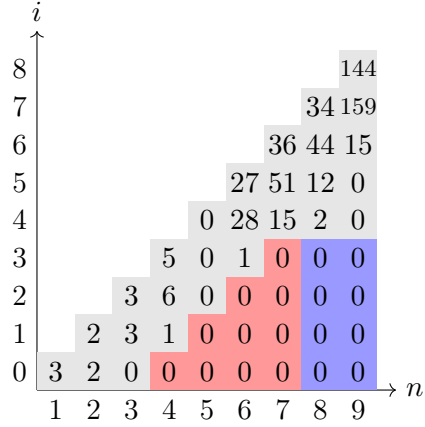
\begin{figure}[h]
    \centering
\begin{tikzpicture}[x=1cm,y=1cm,scale=0.5]
\drawnaxis{9}
\drawiaxis{8}
\drawstaircase{1}{9}{0}{black!10}{0}
\drawstaircase{4}{7}{0}{red!40}{0}
\drawblock{8}{9}{0}{3}{blue!40}{0}
\foreach \n/\i/\txt in {
1/0/3,
2/0/2,
2/1/2,
3/1/3,
3/2/3,
4/1/1,
4/2/6,
4/3/5,
6/3/1,
6/4/28,
6/5/27,
7/4/15,
7/5/51,
7/6/36,
8/4/2,
8/5/12,
8/6/44,
8/7/34,
9/6/15
 }{
    \cell{\n}{\i}{\txt}
  }
\foreach \n/\i/\txt in {
9/7/159,
9/8/144
 }{
    \smallcell{\n}{\i}{\txt}
  }
\end{tikzpicture}
\caption{Homology of $(\C R)_{\zeta_3}$ for $R$ the conjugacy class of transpositions in $S_3$.}\label{fig: S_3 trans 3rd root}
\end{figure}

We can compute the exceptions $\mathcal F_V\setminus \mathcal F_{V,N}$ for higher $N$ by computing the polynomials $f_{p,q}(\bar\zeta)$. For example, we have the following polynomials 

\begin{equation*}
\begin{split}
f_{1,2}(\bar\zeta)&=-\Phi_1^{12}\Phi_3^9\Phi_6^6\Phi_{12}^3,\\
f_{1,3}(\bar\zeta)&=-\Phi_1^{43}\Phi_2^2\Phi_3^{31}\Phi_4^7\Phi_6^{17}\Phi_{9}^2\Phi_{12}^{13}\Phi_{36}^2,\\
f_{1,4}(\bar\zeta)&=-\Phi_1^{132}\Phi_2^{12}\Phi_3^{93}\Phi_4^{24}\Phi_5^6\Phi_6^{51}\Phi_8^6\Phi_9^6\Phi_{10}^9\Phi_{12}^{39}\Phi_{20}^6\Phi_{36}^6\Phi_{40}^6,\\
f_{1,5}(\bar\zeta)&=-\Phi_1^{419}\Phi_2^{42}\Phi_3^{305}\Phi_4^{72}\Phi_5^{36}\Phi_6^{165}\Phi_8^{18}\Phi_9^{28}\Phi_{10}^{28}\Phi_{12}^{117}\Phi_{15}^{21}\Phi_{18}^4\Phi_{20}^{18}\Phi_{30}^7\Phi_{36}^{18}\Phi_{40}^{18}\Phi_{45}^8\Phi_{90}^2,
\end{split}
\end{equation*}
where we write $\Phi_n=\Phi_n(\bar\zeta)$. We omit the other polynomials like $f_{2,1},f_{2,2},\ldots$, as we observe from computation that $f_{p,q}$ have exactly the same roots as $f_{1,p+q-1}$ for $p+q\le 6$. Indeed, this is not hard to show using Proposition \ref{prop: vanishing from shuffle product} (c) and (d) given the observation that the roots of $f_{1,N}$ include the roots of $f_{1,n}$ for $n<N$. 

We also make an interesting observation that $f_{1,m-1}(\bar\zeta)$ are all Kronecker polynomials, i.e. product of cyclotomic polynomials, which means that the exceptions $\mathcal F_V\setminus \mathcal F_{V,N}$ are all twists by roots of unity $\zeta$. Indeed, we will prove that this is true more generally for the finite monodromy case in Section \ref{sec: finite monodromy etc.}.

This gives us the exceptional twists which we present in Figure \ref{fig: S_3 twist exceptions computation}. Since all the exceptional $\zeta$ are Galois orbits of roots of unity, we simply record the order $\ord(\zeta)$. We also write record the corresponding power-savings exponent of the corresponding arithmetic sum as in Theorem \ref{thm: character G extensions}. One could then compute the vanishing for each of these exceptional twists individually.
\begin{figure}[h]
\centering
\begin{tabular}{>{\centering\arraybackslash}m{1cm}
>{\centering\arraybackslash}m{2cm}
>{\centering\arraybackslash}m{2cm}
>{\centering\arraybackslash}m{9cm}}
\toprule
$\mathbf{N}$ & \textbf{Vanishing slope} & \textbf{Exponent} & \textbf{With possible exceptions } $\mathbf{\ord(\zeta)}$ \\
\midrule
$2$ & $\frac 1 3$ & $\frac 5 6$& $1,3$\\
\midrule
$3$ & $\frac 1 2$ & $\frac 3 4$ & $1,3,6,12$\\
\midrule
$4$ & $\frac 3 5$ & $\frac 7 {10}$ & $1,2,3,4,6,9,12,36$\\
\midrule
$5$ & $\frac 2 3$ & $\frac 2 3$ & $1,2,3,4,5,6,8,9,10,12,20,36,40$\\
\midrule
$6$ & $\frac 5 7$ & $\frac 9 {14}$ & $1,2,3,4,5,6,8,9,10,12,15,18,20,30,36,40,45,90$\\
\bottomrule
\end{tabular}
\caption{Table of orders of exceptional twists $\mathcal F_V\setminus \mathcal F_{V,N}$ for $V=\C R$ where $R$ is the rack of transpositions in $S_3$.}\label{fig: S_3 twist exceptions computation}
\end{figure}

We notice that the exceptional twists $V_\zeta\in \mathcal F_V\setminus \mathcal F_{V,N}$ have orders that are quite small, in fact, they always divide $3n(n-1)$ for some $n\le N$ (however, the converse is clearly not true). We will see that this is true more generally for racks coming from groups in Proposition \ref{prop: additional conditions twist}.
\end{example}
\begin{example}\label{eg: computation C_wedge}
Let $V=\C_\wedge$, and again we check that the dual vector space $V^*$ has the same braiding so $V^*_{-\bar\zeta}$ has braiding $R(v_i\otimes v_j)=-(-1)^{ij} \bar\zeta v_j\otimes v_i$. Using the lexicographical ordering, we have, for example, the two matrices
\begin{equation*}\tiny
M_{1,1}^\zeta=
\begin{pmatrix}
1-\bar\zeta & 0 & 0 & 0\\
0 & 1 & -\bar\zeta & 0\\
0 & -\bar\zeta & 1 & 0\\
0 & 0 & 0 & 1+\bar\zeta
\end{pmatrix},\ 
M_{1,2}^\zeta=
\begin{pmatrix}
1-\bar\zeta+\bar\zeta^2 & 0 & 0 & 0 & 0 & 0 & 0 & 0  \\
0 & 1- \bar\zeta & 0 & 0 & \bar\zeta^2 & 0 & 0 & 0\\
0 & \bar\zeta^2 & 1 & 0 & -\bar\zeta & 0 & 0 & 0\\
0 & 0 & 0 & 1 & 0 & -\bar\zeta -\bar\zeta^2 & 0 & 0\\
0 & 0 & -\bar\zeta +\bar\zeta^2 & 0 & 1 & 0 & 0 & 0\\
0 & 0 & 0 & -\bar\zeta & 0 & 1 & -\bar\zeta^2 & 0\\
0 & 0 & 0 & \bar\zeta^2 & 0 & 0 & 1+\bar\zeta & 0\\
0 & 0 & 0 & 0 & 0 & 0 & 0 & 1+\bar\zeta+\bar\zeta^2
\end{pmatrix}.
\end{equation*}
\end{example}
One can compute the following polynomials

\begin{equation*}
\begin{split}
f_{1,1}(\bar\zeta)&= \Phi_1^2\Phi_2^2,\\
f_{1,2}(\bar\zeta)&=\Phi_1^4\Phi_2^4\Phi_3^2\Phi_6^2,\\
f_{1,3}(\bar\zeta)&=\Phi_1^{12}\Phi_2^{12}\Phi_3^4\Phi_4^4\Phi_6^4,\\
f_{1,4}(\bar\zeta)&= \Phi_1^{24}\Phi_2^{24}\Phi_3^8\Phi_4^8\Phi_5^2\Phi_6^8\Phi_{10}^2\Phi_{20}^2,\\
f_{1,5}(\bar\zeta)&=\Phi_1^{52}\Phi_2^{52}\Phi_3^{22}\Phi_4^{16}\Phi_5^{4}\Phi_6^{22}\Phi_{10}^4 \Phi_{15}^2\Phi_{20}^4\Phi_{30}^2,\\
f_{1,6}(\bar\zeta)&= \Phi_1^{108}\Phi_2^{108}\Phi_3^{44}\Phi_4^{32}\Phi_5^8\Phi_6^{44}\Phi_7^6\Phi_{10}^8\Phi_{14}^6\Phi_{15}^4\Phi_{20}^8\Phi_{21}^2\Phi_{30}^4\Phi_{42}^2,\\
f_{1,7}(\bar\zeta)&= \Phi_1^{224}\Phi_2^{224}\Phi_3^{88}\Phi_4^{72}\Phi_5^{16}\Phi_6^{88}\Phi_7^{16}\Phi_8^8\Phi_{10}^{16}\Phi_{14}^{16}\Phi_{15}^8\Phi_{20}^{16}\Phi_{21}^4\Phi_{28}^4\Phi_{30}^8\Phi_{42}^4\Phi_{56}^4,\\
f_{1,8}(\bar\zeta)&= \Phi_1^{452}\Phi_2^{452}\Phi_3^{182}\Phi_4^{148}\Phi_5^{32}\Phi_6^{182}\Phi_7^{32}\Phi_8^{20}\Phi_9^8\Phi_{10}^{32}\Phi_{12}^6\Phi_{14}^{32}\Phi_{15}^{16}\Phi_{18}^8\Phi_{20}^{32}\Phi_{21}^8\Phi_{24}^6\Phi_{28}^8\Phi_{30}^{16}\Phi_{36}^8\Phi_{42}^8\Phi_{56}^8\Phi_{72}^8,
\end{split}
\end{equation*}
which gives us the table of of exceptional orders in $\mathcal F_V\setminus \mathcal F_{V,N}$ in Figure \ref{fig: C_wedge exceptions computation}. 
\begin{figure}[h]
\centering
\begin{tabular}{>{\centering\arraybackslash}m{1cm}
>{\centering\arraybackslash}m{2cm}
>{\centering\arraybackslash}m{2cm}
>{\centering\arraybackslash}m{9cm}}
\toprule
$\mathbf{N}$ & \textbf{Vanishing slope} & \textbf{Exponent} & \textbf{With possible exceptions } $\mathbf{\ord(\zeta)}$ \\
\midrule
$2$ & $\frac 1 3$ & $\frac 5 6$& $1,2$\\
\midrule
$3$ & $\frac 1 2$ & $\frac 3 4$ & $1,2,3,6$\\
\midrule
$4$ & $\frac 3 5$ & $\frac 7 {10}$ & $1,2,3,4,6$\\
\midrule
$5$ & $\frac 2 3$ & $\frac 2 3$ & $1,2,3,4,5,6,10,20$\\
\midrule
$6$ & $\frac 5 7$ & $\frac 9 {14}$ & $1,2,3,4,5,6,10,15,20,30$\\
\midrule
$7$ & $\frac 3 4$ & $\frac 5 8$ & $1,2,3,4,5,6,7,10,14,15,20,21,30,42$\\
\midrule
$8$ & $\frac 7 9$ & $\frac {11} {18}$ & $1,2,3,4,5,6,7,8,10,14,15,20,21,28,30,42,56$\\
\midrule
$9$ & $\frac 4 5$ & $\frac 3 5$ & $1,2,3,4,5,6,7,8,9,10,12,14,15,$ $18,20,21,24,28,30,36,42,56,72$\\
\bottomrule
\end{tabular}
\caption{Table of orders of exceptional twists $\mathcal F_V\setminus \mathcal F_{V,N}$ for $V=\C_\wedge$.}\label{fig: C_wedge exceptions computation}
\end{figure}

We observe from our computation that the exceptional twists $V_\zeta\in \mathcal F_V\setminus \mathcal F_{V,N}$ are essentially those whose orders divide some $n(n-1)$ for some $n\le N$. This is a phenomenon of $S_n$-representations, as we will see in Proposition \ref{prop: additional conditions twist}. There is, however, one exception to this when $N=4$, the twist $V_\zeta$ with $\ord(\chi)=12$ is actually in $\mathcal F_{V,4}$. In fact, $V_\zeta$ lies in $\mathcal F_{V,N}$ for all $N\le 8$, but not in $\mathcal F_{V,9}$. We will give a more in-depth analysis of this exception in Section \ref{sec: explain zeta_12}.
\section{Special cases and extensions for twists}\label{sec: special cases and extensions}

In this section, we discuss several variants and extensions of the general result, arising under stronger hypotheses on $V$. In Section \ref{sec: finite monodromy etc.}, we study three special cases: finite monodromy, racks coming from groups, and $S_n$-representations. We discuss an alternate approach to $S_n$-representations via determinantal identities in Section \ref{sec: aside}. Finally, in Section \ref{sec: multicolored case}, we discuss the multicolored generalization where $V$ is the sum of multiple braided vector spaces and we can introduce twists between each pair of them. 
\subsection{Finite monodromy, racks and $S_n$-representations}\label{sec: finite monodromy etc.} We will consider the following three additional conditions on $V$. 
\begin{enumerate}[(\alph*)]
\item \textbf{Finite monodromy.} For each $n$, the action of $B_n$ on $V^{\otimes n}$ factors through a finite quotient $H_n$ of $B_n$. This is typical in most arithmetic settings.
\item \textbf{Racks from groups.} We have $V=\C R$, where $R$ is a union of conjugacy classes of a group $G$. This is relevant to the case of Hurwitz spaces.
\item \textbf{$S_n$-representations.} For each $n$, the action of $B_n$ on $V^{\otimes n}$ factors through the symmetric group $S_n$. This is relevant to $\C_\wedge$ and irreducibility of polynomials.
\end{enumerate}

Here, (a) is the most general, while (b) and (c) are special cases of (a). That (c) is a sub-case of (a) is obvious by taking $H_n=S_n$. For (b), $\C R$ has finite monodromy because the action of $B_n$ factors through the symmetric group $S_{|R|^n}$, this is because $B_n$ acts on the set $R^n$ of words $v_1\cdots v_n$ with $v_i\in R$, and if it acts by the identity permutation on $R^n$, then it is clear that it also acts by identity on $(\C R)^{\otimes n}$, since $R^n$ is a basis of this vector space. As in both cases above, it is typical that $|H_n|$ grows exponentially or super-exponentially.

We obtain results in increasing strength from (a) to (c).
\begin{proposition}\label{prop: additional conditions twist}
Let $n\ge 2$ be an integer. In each of the three cases above, we have that $H_i(B_n,V_\zeta^{\otimes n})=0$ for all $i$, unless $\zeta$ is a root of unity and
\begin{enumerate}[(\alph*)]
    \item $\ord(\zeta)$ divides $n(n-1)|H_n|$ for $n\ge 3$, $\ord(\zeta)$ divides $|H_n|$ for $n=2$,
    \item $\ord(\zeta)$ divides $n(n-1)|G|$,
    \item $\ord(\zeta)$ divides $n(n-1)$,
\end{enumerate}
for each corresponding case.
\end{proposition}

Let us first compare this with the results in the general case in the previous section. When we have one of these additional assumptions (a), (b) or (c), then $V_\zeta
\in \mathcal F_V\setminus \mathcal F_{V,N}$ only if $\zeta$ is a root of unity and $\ord(\zeta)$ satisfies the corresponding divisibility condition for some $n\le N$. For (b) and (c), there are a lot less of these exceptional twists, and this also explains our observations in Section \ref{sec: computation examples}. Hence, we can obtain stronger vanishing results by again applying Theorem \ref{thm: main extrapolating vanishing} to the vanishing $(N,N-2)$ staircase.

On the other hand, we can also obtain vanishing results by applying Corollary \ref{cor: density} which tells us that the slope of vanishing is greater than or equals to the density $\rho$ of $n$ where homology vanishes. This can be deduced explicitly for the cases of (b) and (c).
\begin{corollary}\label{cor: density for case (b), (c)}
Let $V$ be in either case (b) or (c) above. Let $\zeta$ be a root of unity of order $o$. Let $o'=o/\gcd(o,|G|)$ for (b) and $o'=o$ for (c). Then, $H_i(B_n,V_\zeta^{\otimes n})$ vanishes with slope $1-\frac{2^{\omega(o')}}{o'}$. 
\end{corollary}

As discussed after Corollary \ref{cor: density}, this bound on the slope is often not optimal for most $o$. However, it is a strong uniform statement, which is why we use this for the results stated in the introduction on Gauss sums and character sums. We also remark that for our application to Möbius functions, the relevant braided vector space is twisted by $-1$ which has order $o=2$, and this corollary does not give any vanishing in that case.

\subsubsection{Proofs} 
We now prove Proposition \ref{prop: additional conditions twist} and Corollary \ref{cor: density for case (b), (c)}. Instead of looking at isomorphism of shuffle products like in the previous section, we will instead deduce results from the action of the fundamental braid $\Delta_n$ as in Proposition \ref{prop: vanishing from fundamental braid}. The following lemma on $\Delta_n$ is key.
\begin{lemma}\label{lem: fundamental braid acts trivially}
Let $n\ge 2$ be an integer. In each of the three cases above, we have that $$\begin{cases}
\Delta_n^{|H_n|} & \text{for (a)}\\
\Delta_n^{2|G|} & \text{for (b)}\\
\Delta_n^{2} & \text{for (c)}\\
\end{cases}$$
acts trivially on $V^{\otimes n}$.
\end{lemma}
\begin{proof}
The case of $(a)$ follows from Lagrange's theorem which implies that the image of $\Delta_n^{|H_n|}$ is trivial in $H_n$ so it acts by the identity by definition of finite monodromy. Similarly, $(c)$ follows as the image of $\Delta_n^2$ is the identity in $S_n$.

The case of $(b)$ is more involved. We need to analyze the action of the fundamental braid on $R^n$ which is a basis for $(\C R)^n$. Recall from Example \ref{eg: BVS}(b) that the action of $B_n$ on $R^n$ is defined by the generators $\sigma_i$ sending $(g_1,\ldots, g_i,g_{i+1},\ldots, g_n)$ to $(g_1,\ldots, g_{i+1},g_i^{g_{i+1}},\ldots, g_n)$ where $g^h=h^{-1}gh$. Since $\Delta_n=(\sigma_1\cdots\sigma_{n-1})\cdots(\sigma_1\sigma_2)(\sigma_1)$, so we need to apply the sequence of braids $w_1=\sigma_1, w_2=\sigma_1\sigma_2,\ldots, w_{n-1}=\sigma_1\sigma_2\ldots \sigma_{n-1}$ in this order. Starting with $(g_1,\ldots, g_n)$, it is easy to see by induction that applying $w_1$ to $w_i$ in sequence will give
$$(g_{i+1},g_i^{g_{i+1}},g_{i-1}^{g_ig_{i+1}},\ldots, g_1^{g_2\cdots g_{i+1}},g_{i+2},\ldots, g_n)$$
so applying $\Delta_n$ sends
$$(g_1,\ldots, g_n)\mapsto (g_n,g_{n-1}^{g_n},\ldots, g_1^{g_2\cdots g_n}).$$
Applying the fundamental braid twice, the $i$-th term would then be
$$\left(g_i^{g_{i+1}\cdots g_n}\right)^{\left(g_{i-1}^{g_i\cdots g_n}\cdots g_1^{g_2\cdots g_n}\right)}=g_i^{g_ig_{i+1}\cdots g_ng_{i-1}^{g_i\cdots g_n}\cdots g_1^{g_2\cdots g_n}},$$
where in the second term we added a $g_i$ at the front of the exponent because $g_i^{g_i}=g_i$. By expanding the conjugated terms, the exponent of the second term can be written as
$$(g_i\cdots g_n)(g_n^{-1}\cdots g_i^{-1})(g_{i-1}\cdots g_n)(g_n^{-1}\cdots g_{i-1}^{-1})\cdots (g_n^{-1}\cdots g_2^{-1})(g_1\cdots g_n).$$
We see that the $2j$-th and $(2j+1)$-th brackets cancel for $0\le j\le i-2$, yielding just $g_1\cdots g_n$. Hence, applying $\Delta_n^2$ sends
$$(g_1,\ldots, g_n)\mapsto (g_1^g,\ldots, g_n^g)$$
where $g=g_1\cdots g_n$ is the global monodromy. Thus, $\Delta_n^{2|G|}$ acts by the identity on $R^n$ (and hence acts trivially on $(\C R)^n$) because it is given by conjugating each term by $g^{|G|}$, which is the identity by Lagrange's theorem.
\end{proof}

The proposition now follows from a direct application of Proposition \ref{prop: vanishing from fundamental braid}.
\begin{proof}[Proof of Proposition \ref{prop: additional conditions twist}] Choose $k=2|H_n|$ if $n\ge 3$ and $k=|H_2|$ if $n=2$ for case (a), $k=2|G|$ for (b) and $k=2$ for (c). Lemma \ref{lem: fundamental braid acts trivially} implies that $\Delta_n^k$ act trivially on $V^{\otimes n}$, so $\Delta_n^k$ will act by $\zeta^{k \inv(\Delta_n)}=\zeta^{kn(n-1)/2}$ on $V_\zeta^{\otimes n}$. Hence, applying Proposition \ref{prop: vanishing from fundamental braid} tells us that if $\zeta^{kn(n-1)}\neq 1$, we must have $H_i(B_n,V_\zeta^{\otimes n})=0$. The contrapositive gives the desired statement.
\end{proof}

Finally, we prove the corollary by determining the density $\rho$ of $n\in \N$ satisfying the conditions in Proposition \ref{prop: additional conditions twist}.
\begin{proof}[Proof of Corollary \ref{cor: density for case (b), (c)}] Let $C=|G|$ in case (b) and $C=1$ in case (c). We want to count the number of solutions to $Cn(n-1)\equiv 0 \pmod o$. By removing the common factor between $C$ and $o$, this is equivalent to $n(n-1)\equiv 0 \pmod {o'}$ where $o'$ is defined as in the statement of the corollary. 

We claim that there are $2^{\omega(o')}$ such solutions modulo $o'$. Consider the prime factorization of $o'=p_1^{\alpha_1}\cdots p_k^{\alpha_k}$. Each prime can divide at most one of $n$ or $n-1$, so let $S$ be the set of indices $i$ where $p_i$ divides $n$. Then, we must have $\prod_{i\in S}p_i^{\alpha_i}\mid n$ and $\prod_{i\not\in S}p_i^{\alpha_i}\mid n-1$. By the Chinese remainder theorem, there is exactly one $n$ modulo $o'$ which satisfies this. Since there are $2^{\omega(o')}$ ways to choose $S$, the claim is true, and the statement follows by applying Corollary \ref{cor: density}.
\end{proof}
\subsection{Aside on determinantal identities}\label{sec: aside} Let $V$ satisfy condition (c), that is, each $V^{\otimes n}$ is an $S_n$-representation. We deduced from Proposition \ref{prop: additional conditions twist}(c) that $V_\zeta\in \mathcal F_V\setminus \mathcal F_{V,N}$ is an exceptional twist only if $\zeta$ is a root of unity and $\ord(\zeta)\mid n(n-1)$ for some $n\le N$.

In this section, we give a separate proof of this fact, which we believe to be rather enlightening and showcases the connection to determinantal identities of Zagier \cite{Zagier_1992} and Hanlon-Stanley \cite{Hanlon_Stanley_1998}. By Proposition \ref{prop: vanishing from shuffle product}, it suffices to prove the following statement.
\begin{proposition}\label{prop: tau isom det identity}
Let $p+q=n\ge 2$ be positive integers, then $\tau^\zeta_{p,q}$ is an isomorphism unless $\zeta$ is a root of unity with $\ord(\zeta) \mid i(i-1)$ for some $i\le n$.
\end{proposition}

Since $V^{\otimes n}$ is an $S_n$-representation, it follows that the twisted dual $V_{-1}^{*\otimes n}$ is also an $S_n$-representation for all $n$. Let us view the shuffle product $\tau^\zeta_{p,q}\colon \mathfrak A^\zeta _p\otimes \mathfrak A^\zeta_q\rightarrow \mathfrak A^\zeta_n$ as a map $V^{*\otimes n}_{-1}\rightarrow V^{*\otimes n}_{-1}$ via the obvious isomorphisms. Now, by viewing $V^{*\otimes n}_{-1}$ as a representation of $S_n$, the map $\tau^\zeta_{p,q}$ is simply given by the operator
$$\Omega_{p,q}(\bar\zeta)=\sum_{\sigma \in \text{Sh}(p,q)} \bar\zeta^{\inv(\sigma)} \sigma \in \C[S_n]$$
on this representation.

We want to know for which $\zeta$ the determinant of the operator $\Omega_{p,q}(\bar\zeta)$ acting on the representation $V^{*\otimes n}_{-1}$ is zero. By the representation theory of finite groups, we can decompose the representation
$$V^{*\otimes n}_{-1}=\bigoplus_i W_i^{k_i}$$
into irreducible representations of $S_n$. On the other hand, the regular representation $\C[S_n]$ contains all irreducible representations of $S_n$. Hence, this means that if $\Omega_{p,q}(\bar\zeta)$ is invertible on $\C[S_n]$ then it must be invertible on $V^{\otimes n}_{-1}$. So Proposition \ref{prop: tau isom det identity} will follow from the following statement.
\begin{lemma}\label{lem: det Omega_p,q}
The determinant of $\Omega_{p,q}(\bar\zeta)$ on $\C[S_n]$ is a polynomial in $\bar \zeta$, and the roots of this polynomial are exactly the roots of unity with order dividing $i(i-1)$ for some $i\le n$.
\end{lemma}

It turns out that the determinant can be evaluated explicitly using the methods of \cite{Zagier_1992} and \cite{Hanlon_Stanley_1998}. This turns out to be quite elegant and rather tricky. Our argument is based on \cite[Section 2]{Hanlon_Stanley_1998} and \cite[Section 3]{Zagier_1992}. Although some of the statements are already known, we adapt the arguments to give simple and self-contained proofs.

\subsubsection{Determinant computations} We define the operator $$\Omega_n(\bar\zeta)=\sum_{\sigma\in S_n}\bar\zeta^{\inv(\sigma)}\sigma.$$
We begin with evaluating the determinant of $\Omega_{1,n-1}(\bar\zeta)$. This will give us the determinant of $\Omega_n(\bar\zeta)$ which will in turn give us the determinant of $\Omega_{p,q}(\bar\zeta)$.
\begin{lemma}[{\cite[Theorem 2']{Zagier_1992}}]\label{lem: Thm 2' Zagier}
The determinant of $\Omega_{1,n-1}(\bar\zeta)$ on $\C[S_n]$ is
$$\prod_{i=2}^n(1-\bar\zeta^{i(i-1)})^{\frac{n!}{i(i-1)}}.$$
\end{lemma}
\begin{proof}
Let $\sigma_j=(j \ j-1\ \cdots 1)$ so $\Omega_{1,n-1}(\bar\zeta)=\sum_{j=1}^{n} \bar\zeta^{j-1} \sigma_j$. Using the inclusion $S_{n-1}\hookrightarrow S_n$ into the first $n-1$ elements, we can also view $\Omega_{1,n-2}(\bar\zeta)=\sum_{j=1}^{n-1}\bar\zeta^{j-1}\sigma_j$ inside $\C[S_n]$. Now, denote $w=(n\ n-1\ \cdots 2)$, and we note the identity $\sigma_n\sigma_j=\sigma_{j-1}w$. This allows us to derive the identity
\begin{equation}\label{eqn: stanley identity}
(1-\bar\zeta^{n-1}\sigma_n)\Omega_{1,n-1}(\bar\zeta)=\Omega_{1,n-2}(\bar\zeta)(1-\bar\zeta^n w).
\end{equation}
which is essentially \cite[Equation (2.2)]{Hanlon_Stanley_1998}. 

Now we induct on $n$. We prove this for $n$ assuming the identity for $n-1$. The determinant $\Omega_{1,n-2}(\bar\zeta)$ on $\C[S_n]$ is the $n$-th power of its determinant on $\C[S_{n-1}]$ which we know by the induction hypothesis, so by Equation \eqref{eqn: stanley identity} we have
$$\det(\Omega_{1,n-1}(\bar\zeta))=\frac{\det(1-\bar\zeta^nw)}{\det(1-\bar\zeta^{n-1}\sigma_n) }\prod_{i=2}^{n-1}(1-\bar\zeta^{i(i-1)})^{\frac{n!}{i(i-1)}}$$
as polynomials in $\bar\zeta$.

It remains to evaluate $\det(1-\bar\zeta^nw)$ and $\det(1-\bar\zeta^{n-1}\sigma_n)$. Since $w$ is an $n-1$ cycle, the matrix of its action on $\C[S_n]$ decomposes into $n!/(n-1)$ blocks and each block is a cyclic matrix. Hence, the spectrum of $\bar\zeta^n w$ has $n!/(n-1)$ copies of the eigenvalue $\bar\zeta^n\zeta_{n-1}^j$ for each $0\le j<n-1$ where $\zeta_{n-1}$ is an $n-1$-th root of unity. The spectrum of $1-\bar\zeta^nw$ is thus determined, and we have
$$\det(1-\bar\zeta^nw)=\prod_{j=0}^{n-2}(1-\bar\zeta^n\zeta_{n-1}^j)^{\frac{n!}{n-1}}=(1-\bar\zeta^{n(n-1)})^{\frac{n!}{n-1}}.$$
Similarly, we obtain
$$\det(1-\bar\zeta^{n-1}\sigma_n)=(1-\bar\zeta^{n(n-1)})^{(n-1)!}$$
which finishes the induction.
\end{proof}
\begin{lemma}[{\cite[Theorem 2]{Zagier_1992}}]\label{lem: thm 2 zagier}
The determinant of $\Omega_n(\bar\zeta)$ on $\C[S_n]$ is
$$\prod_{i=2}^n(1-\bar\zeta^{i(i-1)})^{\frac{n!(n-i+1)}{i(i-1)}}.$$
\end{lemma}
\begin{proof}
Using the inclusion $S_{n-1}\hookrightarrow S_n$ into the \textit{last} $n-1$ elements, the operator $\Omega_{n-1}(\bar\zeta)$ acts on $\C[S_n]$ with the $n$-th power of its determinant on $\C[S_{n-1}]$. We also note that 
$$\Omega_n(\bar\zeta)=\Omega_{1,n-1}(\bar\zeta)\Omega_{n-1}(\bar\zeta)$$
because any $\sigma\in S_n$ is a unique composition of permutation on the last $n-1$ elements with a $(1,n-1)$-shuffle. The formula thus follows from induction and Lemma \ref{lem: Thm 2' Zagier}.
\end{proof}
Finally, we have our desired formula, which proves Lemma \ref{lem: det Omega_p,q}.
\begin{lemma} The determinant of $\Omega_{p,q}(\bar\zeta)$ on $\C[S_n]$ is
$$\prod_{i=2}^n(1-\bar\zeta^{i(i-1)})^{\frac{n!}{i(i-1)}\min(p,q,i-1,n-i+1)}.$$
\end{lemma}
\begin{proof}
We use a similar trick to above. Since every $\sigma\in S_n$ is a unique composition of a permutation on the first $p$ elements, a permutation on the last $q$ elements, and a $(p,q)$-shuffle, we have
$$\Omega_n(\bar\zeta)=\Omega_{p,q}(\bar\zeta)\Omega_p(\bar\zeta)\Omega_q(\bar\zeta)$$
where the operators $\Omega_p(\bar\zeta)$ and $\Omega_q(\bar\zeta)$ acts on the first $p$ and last $q$ elements respectively. The formula then follows from Lemma \ref{lem: thm 2 zagier}.
\end{proof} 
\subsubsection{An explanation for $\C_\wedge$ and $\zeta_{12}$}\label{sec: explain zeta_12} We end the section by briefly explaining the exceptional case of $\ord(\zeta)=12$ as observed at the end of Example \ref{eg: computation C_wedge}. From the previous lemma, we see that the determinant of $\Omega_{p,q}(\bar\zeta_{12})$ on the regular representation $\C[S_n]$ is zero for $n\ge 4$. However, this only implies that $\Omega_{p,q}(\bar\zeta_{12})$ is degenerate for some irreducible representation $W_i$. On the other hand, not all irreducible representations of $S_n$ will appear in $V^{*\otimes n}_{-1}$ when $V=\C_\wedge$. In fact, by Equation \eqref{eqn: C_wedge decomp} and \cite[Lemma 8.0.3]{ES26}, $V^{\otimes n}$ is a direct sum of irreducible representations $\wedge^k\text{std}_n$, so the irreducible representations in $V^{*\otimes n}_{-1}$ are simply the dual of these tensored by the sign representation. Hence, it could be possible that $\Omega_{p,q}(\bar\zeta_{12})$ are non-degenerate on these irreducible representations. We believe that this phenomenon to be a coincidence, and that there are likely no other exceptions for larger $n$ and $\ord(\zeta)$.
\subsection{Multicolored case}\label{sec: multicolored case}
We now generalize this to the setting of multiple braided vector spaces, which is relevant to applications to multicolored configuration space in Section \ref{sec: intro multiple characters}. Roughly speaking, we want to consider a direct sum decomposition $V=V_1\oplus\cdots \oplus V_k$, and add twists to the braiding between each $V_i$ and $V_j$. We make this precise as follows.

We need to deal with a collection of braided vector spaces with compatible braidings between them, and the following construction makes this possible. Define the notion of a \textit{braided monoidal abelian category} $(\mathcal V, \otimes, 1, R)$ over $\C$ to be a monoidal $\C$-linear abelian category $(\mathcal V, \otimes, 1)$ with the following braiding data. For any $V,W\in \mathcal V$, there is a braiding $R_{V,W}\colon V\otimes W\xrightarrow\sim W\otimes V$ which satisfies the braiding axioms, see \cite[Definition 3.2.1] {HS20}. 

Let $\mathcal V$ be a braided monoidal abelian category of $\C$-vector spaces, which means that the underlying objects of the category are $\C$-vector spaces and the tensor product is given by the tensor product as vector spaces. Then, each object $V\in \mathcal V$ is itself a braided vector space by looking at $R_{V,V}$. This category now admits direct sums by taking the usual direct sum of the underlying vector spaces, with braiding data defined blockwise. The example relevant to our applications is the category of Yetter-Drinfeld modules $\mathcal{YD}_G$, see \cite[Section 1.4]{HS20} or \cite[Example 2.5]{ma2026weightfiltrationhurwitzspaces} for a definition.

Now, let $V_1,\ldots, V_k\in \mathcal V$ be objects in this category with braiding data $R_{V_i,V_j}$. We can construct the direct sum $V=V_1\oplus\cdots\oplus V_k$, and to consider twists of $V$ we add twists to the braiding data. Let $\zeta_i\in \C^\times$ for $1\le i\le k$ and $\zeta_{ij}\in \C^\times$ for $1\le i<j\le k$. We can construct another braided monoidal category $\mathcal V'$ which is freely generated by $V_1',\ldots, V_k'$, where $V_i'\cong V_i$ are isomorphic as vector spaces but the braiding data is now twisted: $R_{V_i',V_i'}=\zeta_i\cdot R_{V_i,V_i}$ for $1\le i\le k$, $R_{V_i',V_j'}=\zeta_{ij}\cdot R_{V_i,V_j}$ and $R_{V_j',V_i'}=\zeta_{ij}\cdot R_{V_j,V_i}$ for $1\le i<j\le k$. This is well defined as the braiding axioms will still be true as the twisting simply introduces a constant multiplier and the required diagrams will commute. Finally, we can define the twist given by the data $\vec\zeta=((\zeta_i)_i,(\zeta_{ij})_{ij})$ to be the braided vector space which is the direct sum $V_{\vec\zeta}=V_1'\oplus\cdots \oplus V_k'$ in the category $\mathcal V'$. We remark that there are other ways to define this, for example, one could follow the definitions in \cite[Section 5]{ES26}.

Let $\mathcal F_V$ be the family of twists $V_{\vec \zeta}$, and $\mathcal F_{V,N}$ be the twists that vanish in a $(N,N-2)$ staircase, defined analogously to Section \ref{sec: twists of BVS}. Recall that these imply vanishing of slope $\frac{N-1}{N+1}$ by Theorem \ref{thm: main extrapolating vanishing}. Without any assumptions on $V$, we have the following proposition.
\begin{proposition}
Let $N\ge 2$ be an integer. The twist $V_{\vec\zeta}$ is in the exceptional set $\mathcal F_V\setminus \mathcal F_{V,N}$ if and only if $\vec{\bar\zeta}=((\bar\zeta_i)_i, (\bar\zeta_{ij})_{ij})$ is a root of a polynomial $P_N(\vec{\bar\zeta})$ of total degree not more than $N^2d^N$.
\end{proposition}
\begin{proof}
The proof is completely analogous to the proof of Proposition \ref{prop: general case twist}, except that now $f_{p,q}(\vec{\bar\zeta})$ is now a multivariate polynomial.
\end{proof}

Suppose we are only concerned with the case where $\zeta_i,\zeta_{ij}$ are all roots of unity, like the situation in Section \ref{sec: intro multiple characters}. Then, Manin-Mumford for tori \cite[Theorem 4.2.2]{Bombieri_Gubler_2009} immediately gives us the following corollary.
\begin{corollary}\label{cor: manin mumford}
Let $N\ge 2$ be an integer. Suppose $\zeta_i,\zeta_{ij}$ are given to be roots of unity. Then there exists a finite union of torsion cosets $U$ such that $V_{\vec\zeta}$ is in the exceptional set $\mathcal F_V\setminus \mathcal F_{V,N}$ if and only if $\vec\zeta\in U$.
\end{corollary}
\subsubsection{Additional conditions} Keeping the assumption that $\zeta_i,\zeta_{ij}$ are roots of unity, and if we also impose the additional conditions of (a), (b) and (c) in Section \ref{sec: finite monodromy etc.} on $V$, we can get much more information on these torsion cosets. 

We will derive this in an analogous way to Section \ref{sec: finite monodromy etc.}. Following the proof of Proposition \ref{prop: additional conditions twist}, we choose $k=2|H_n|$ if $n\ge 3$ and $k=|H_2|$ if $n=2$ for case (a), $k=2|G|$ for (b) and $k=2$ for (c). Then, Lemma \ref{lem: fundamental braid acts trivially} tells us that $\Delta_n^k$ acts trivially on $V^{\otimes n}$. Now, we decompose the $B_n$-representation $V^{\otimes n}$ as
$$V^{\otimes n}=\bigoplus_{n_1+\cdots+n_k=n}\Ind_{B_{n_1,\ldots, n_k}}^{B_n}\left(V_1^{\otimes n_1}\otimes \cdots \otimes V_n^{\otimes n_k}\right)$$ 
where the induced representation of the form above is the direct summand of all tensor products which contain the factor $V_i$ exactly $n_i$ times. Since $\Delta_n^k$ acts trivially on $V^{\otimes n}$, it acts trivially on each of these factors. We decompose $V_{\vec\zeta}^{\otimes n}$ analogously as
$$V_{\vec\zeta}^{\otimes n}=\bigoplus_{n_1+\cdots+n_k=n}\Ind_{B_{n_1,\ldots, n_k}}^{B_n}\left(V_1'^{\otimes n_1}\otimes \cdots \otimes V_n'^{\otimes n_k}\right),$$
so after accounting for the added twists, $\Delta_n^k$ acts on the twisted induced representation above by $\vec\zeta(n_1,\ldots, n_k)^k$ where we define
$$\vec\zeta(n_1,\ldots, n_k)\coloneqq\prod_i\zeta_i^{n_i(n_i-1)/2}\prod_{i<j}\zeta_{ij}^{n_in_j}.$$ 
If this scalar $\vec\zeta(n_1,\ldots, n_k)^k\neq 1$, then the cohomology of the twisted induced representation must be zero by the analogue of Proposition \ref{prop: vanishing from fundamental braid} where $V^{\otimes n}$ is replaced by this representation. Hence, we have the following result in analogy to Proposition \ref{prop: additional conditions twist}.
\begin{proposition}\label{prop: explicit torsion coset}
Let $n\ge 2$ be an integer. Suppose that $\zeta_i,\zeta_{ij}$ are all roots of unity. In each of the three cases, we have that $H_i(B_n,V_{\vec\zeta}^{\otimes n})=0$ for all $i$ unless there is some $n_1+\cdots +n_k=n$ such that
\begin{enumerate}[(\alph*)]
    \item $\zeta(n_1,\ldots, n_k)^{2|H_n|}=1$ for $n\ge 3$, $\zeta(n_1,\ldots, n_k)^{|H_n|}=1$ for $n=2$,
    \item $\zeta(n_1,\ldots, n_k)^{2|G|}=1$,
    \item $\zeta(n_1,\ldots, n_k)^2=1$,
\end{enumerate}
for each corresponding case.
\end{proposition}
By combining the equations above for $2\le n\le N$, we have the explicit equations for the torsion cosets in Corollary \ref{cor: manin mumford} which cut out the twists in the exceptional set $\mathcal F_V\setminus \mathcal F_{V,N}$.
\section{Individual braided vector spaces}\label{sec: individual BVS}
In this section, we consider specific examples of braided vector spaces and apply the principle of Corollary \ref{cor: optimal slope cor} to attempt to compute near optimal bounds for each individual braided vector space. We will first look at some twists of $\C_\wedge$ which are directly related to bias in Gauss sums, then look at sign twists of racks coming from groups $(\C R)_{-1}$ which correspond to Möbius sums over Galois $G$-extensions. Lastly, we will consider other examples of braided vector spaces.
\subsection{Twists of $\C_\wedge$}\label{sec: individual C_wedge} Here, we prove the homological vanishing results relevant to Figure \ref{fig: table order}. We consider $(\C_\wedge)_{\zeta}$ for the following $5$ cases $\zeta\in \{\zeta_{12},\zeta_{15},\zeta_{20},\zeta_{60},\zeta_{105}\}$. The point is to showcase how our techniques allow us to prove better bounds given specific orders, and the key is to use the vanishing in Proposition \ref{prop: additional conditions twist}(c) to extend the vanishing staircases. In the first three cases, we can get some computational results which will help us get better bounds, and we present this in Figure \ref{fig: C_wedge computational results}. In the latter two cases, the order is too large for us to observe any $n$ where the homology does not already vanish trivially by Proposition \ref{prop: additional conditions twist}(c). Nevertheless, we can obtain better vanishing results by choosing when to apply Theorem \ref{thm: main extrapolating vanishing} appropriately. We consider each of the $5$ cases below.

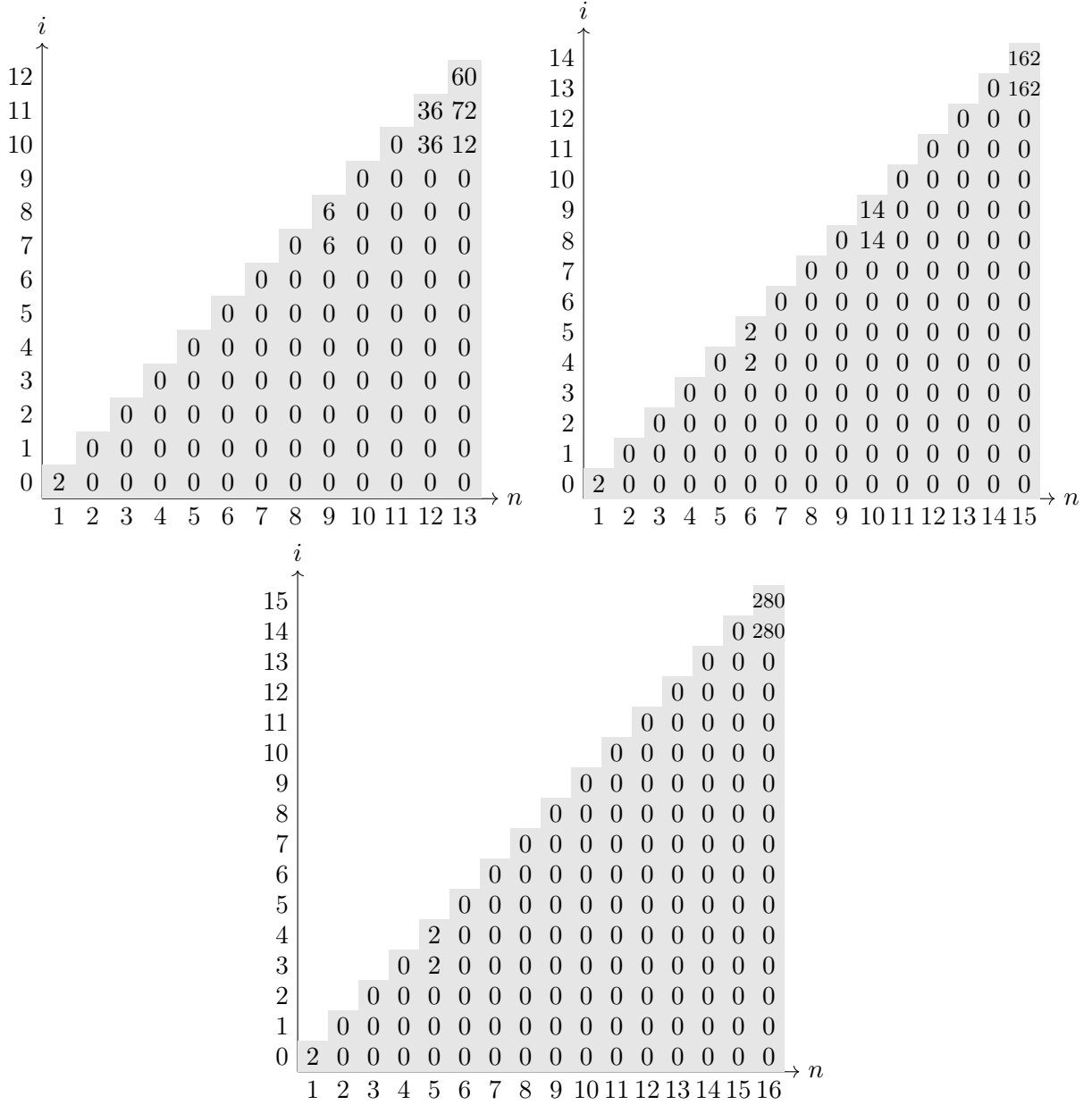
\begin{figure}[h]
    \centering
\begin{tikzpicture}[x=1cm,y=1cm,scale=0.5]
\drawnaxis{13}
\drawiaxis{12}
\drawstaircase{1}{13}{0}{black!10}{0}
\foreach \n/\i/\txt in {
1/0/2,
9/7/6,
9/8/6,
12/10/36,
12/11/36,
13/10/12,
13/11/72,
13/12/60
 }{
    \cell{\n}{\i}{\txt}
  }
\end{tikzpicture}
\begin{tikzpicture}[x=1cm,y=1cm,scale=0.45]
\drawnaxis{15}
\drawiaxis{14}
\drawstaircase{1}{15}{0}{black!10}{0}
\foreach \n/\i/\txt in {
1/0/2,
6/4/2,
6/5/2,
10/8/14,
10/9/14
 }{
    \cell{\n}{\i}{\txt}
  }
\foreach \n/\i/\txt in {
15/13/162,
15/14/162
 }{
    \smallcell{\n}{\i}{\txt}
  }
\end{tikzpicture}
\begin{tikzpicture}[x=1cm,y=1cm,scale=0.45]
\drawnaxis{16}
\drawiaxis{15}
\drawstaircase{1}{16}{0}{black!10}{0}
\foreach \n/\i/\txt in {
1/0/2,
5/3/2,
5/4/2
 }{
    \cell{\n}{\i}{\txt}
  }
\foreach \n/\i/\txt in {
16/14/280,
16/15/280
 }{
    \smallcell{\n}{\i}{\txt}
  }
\end{tikzpicture}
    \caption{Homology of $\C_\wedge$ twisted by $\zeta_{12}$, $\zeta_{15}$ and $\zeta_{20}$.}\label{fig: C_wedge computational results}.
\end{figure}
\begin{enumerate}
\item $\zeta_{12}$. Proposition \ref{prop: additional conditions twist}(c) tells us that the homology vanishes unless $n\equiv 0,1,4,9\pmod{12}$, which gives a vanishing slope of $\frac 2 3$ by Corollary \ref{cor: density for case (b), (c)}. From our computations in Figure \ref{fig: C_wedge computational results}, we see that there is vanishing in a $(8,6)$ staircase, giving a vanishing slope of $\frac 7 9$. We can combine both facts to deduce vanishing in increasingly larger staircases.
\begin{lemma}
$H_i(B_n,(\C_\wedge)_{\zeta_{12}}^{\otimes n})$ vanishes in a $(12 \cdot 2^i-4,10\cdot2^i-4)$ staircase for every $i\ge 0$.
\end{lemma}
\begin{proof}
We prove this by induction. The $i=0$ case of $(8,6)$ is true from above. We now prove the statement for $i$ assuming that it is true for $i-1$, i.e. we have vanishing in a $(N,I)\coloneqq (12\cdot2^{i-1}-4,10\cdot 2^{i-1}-4)$ staircase. Applying Theorem \ref{thm: main extrapolating vanishing} to the $(N,I)=(12 \cdot 2^{i-1}-4,10\cdot2^{i-1}-4)$ staircase tells us that it vanishes for the $(2N+1,2I+1)=(12\cdot 2^i-7,10\cdot 2^i-7)$ staircase. Furthermore, when $n=12\cdot 2^i-6,12\cdot 2^i-5,12\cdot 2^i-4$, we have $n\equiv 6,7,8\pmod {12}$ so their homology must vanish by Proposition \ref{prop: additional conditions twist}(c). This allows us to extend this to vanishing in a $(12\cdot 2^i-4,10\cdot 2^i-4)$ staircase as desired.
\end{proof}
This lemma gives a vanishing slope of $\frac{10\cdot 2^i-3}{12\cdot 2^i-3}$, and taking $i\rightarrow \infty$, the slope of vanishing approaches $\frac 5 6$ as desired.
\item $\zeta_{15}$. Proposition \ref{prop: additional conditions twist}(c) tells us that the homology vanishes unless $n\equiv 0,1,6,10\pmod{15}$, and Figure \ref{fig: C_wedge computational results} tells us that the homology vanishes in a $(15,12)$ staircase. Applying Theorem \ref{thm: main extrapolating vanishing}, we deduce vanishing in a $(15\cdot 2 +1,12\cdot 2+1)=(31,25)$ staircase. However, the homology in the range $32\le n\le 35$ vanishes, so we extend this to the $(35,29)$ staircase. Applying Theorem \ref{thm: main extrapolating vanishing} again, we have vanishing in a $(70,58)$ staircase, and we then extend this to the $(74,62)$ staircase. This gives us vanishing of slope $\frac{21}{25}$.
\item $\zeta_{20}$. From Figure \ref{fig: C_wedge computational results}, we see that there is vanishing in a $(16,13)$ staircase, and Proposition \ref{prop: additional conditions twist}(c) extends this to a $(19,16)$ staircase, giving us a vanishing slope of $\frac{17}{20}$.
\item $\zeta_{60}$. The homology vanishes unless $n\equiv 0,1,16,21,25,36,40,45\pmod{60}$. We immediately have vanishing in a $(15,13)$ staircase, giving us a $(15\cdot 2+1,13\cdot 2+1)=(31,27)$ staircase which we extend to $(35,31)$, which further gives $(35\cdot 3+2,31\cdot 3+2)=(107,95)$ and we finally extend this to a $(119,107)$ staircase. This gives us a vanishing slope of $\frac{9}{10}$.
\item $\zeta_{105}$. The homology vanishes unless $n\equiv 0,1,15,21,36,70,85,91\pmod{105}$. By Corollary \ref{cor: density}, we have vanishing in a $(69,64)$ staircase. We claim for all $i\ge 0$ that we have vanishing in a $(105\cdot 2^i-36,99\cdot 2^i-35)$ staircase, and prove this by induction as follows. The base case is true, and if we have vanishing for $(N,I)=(105\cdot 2^{i-1}-36,99\cdot 2^{i-1}-35)$, we have vanishing for $(2N+1,2I+1)=(105\cdot 2^i-71,99\cdot 2^i-69)$. There is only one possible $n$ with nonzero homology in the range $105\cdot 2^i-70\le n\le 105\cdot 2^i-36$, so by the proof of Corollary \ref{cor: density} this results in at most one more homological degree where homology could be nonzero. Hence, we have vanishing for $(2N+1+35, 2I+1+34)=(105\cdot 2^i-36, 99\cdot 2^i-35)$ as desired. This gives us a slope of vanishing that approaches $\frac{33}{35}$ as $i\rightarrow \infty$.
\end{enumerate}
\begin{figure}
    \centering
\begin{tikzpicture}[x=1cm,y=1cm,scale=0.55]
\drawnaxis{11}
\drawiaxis{10}
\drawstaircase{1}{11}{0}{black!10}{0}
\drawstaircase{4}{11}{0}{red!40}{0}

\foreach \n/\i/\txt in {
1/0/3,
4/2/2,
4/3/2,
5/3/6,
5/4/6,
6/3/2,
6/4/10,
6/5/8,
7/5/30,
7/6/30,
8/5/5,
8/6/70,
8/7/65,
10/7/5
 }{
    \cell{\n}{\i}{\txt}
  }
\foreach \n/\i/\txt in {
9/7/126,
9/8/126,
10/8/338,
10/9/333,
11/9/810,
11/10/810
 }{
    \smallcell{\n}{\i}{\txt}
  }
\end{tikzpicture}
\begin{tikzpicture}[x=1cm,y=1cm,scale=0.75]
\drawnaxis{5}
\drawiaxis{4}
\drawstaircase{1}{5}{0}{black!10}{0}
\drawstaircase{4}{5}{0}{red!40}{0}
\foreach \n/\i/\txt in {
1/0/10,
2/0/15,
2/1/15,
3/1/50,
3/2/50,
4/1/35,
4/2/304,
4/3/260,
5/2/130
 }{
    \cell{\n}{\i}{\txt}
  }
\foreach \n/\i/\txt in {
5/3/1740,
5/4/1610
 }{
    \smallcell{\n}{\i}{\txt}
  }
\end{tikzpicture}
\begin{tikzpicture}[x=1cm,y=1cm,scale=0.7]
\drawnaxis{7}
\drawiaxis{6}
\drawstaircase{1}{7}{0}{black!10}{0}
\drawstaircase{4}{7}{0}{red!40}{0}
\foreach \n/\i/\txt in {
1/0/5,
3/1/10,
3/2/10,
4/1/2,
4/2/22,
4/3/20,
5/3/70,
5/4/70,
6/3/3,
6/4/154,
6/5/151,
7/5/950,
7/6/950
 }{
    \cell{\n}{\i}{\txt}
  }
\end{tikzpicture}
\begin{tikzpicture}[x=1cm,y=1cm,scale=0.7]
\drawnaxis{7}
\drawiaxis{6}
\drawstaircase{1}{7}{0}{black!10}{0}
\drawstaircase{4}{7}{0}{red!40}{0}
\foreach \n/\i/\txt in {
1/0/7,
3/1/28,
3/2/28,
4/1/4,
4/2/60,
4/3/56,
5/3/420,
5/4/420,
6/3/14,
6/4/964,
6/5/950
 }{
    \cell{\n}{\i}{\txt}
  }
\foreach \n/\i/\txt in {
7/5/9842,
7/6/9842
 }{
    \smallcell{\n}{\i}{\txt}
  }
\end{tikzpicture}
\begin{tikzpicture}[x=1cm,y=1cm,scale=0.7]
\drawnaxis{7}
\drawiaxis{6}
\drawstaircase{1}{7}{0}{black!10}{0}
\drawstaircase{3}{5}{0}{red!40}{0}
\drawblock{6}{7}{0}{2}{blue!40}{0}
\foreach \n/\i/\txt in {
1/0/4,
3/1/2,
3/2/2,
5/3/24,
5/4/24,
6/3/5,
6/4/78,
6/5/73,
7/5/144,
7/6/144
 }{
    \cell{\n}{\i}{\txt}
  }
\end{tikzpicture}
\begin{tikzpicture}[x=1cm,y=1cm,scale=0.7]
\drawnaxis{7}
\drawiaxis{6}
\drawstaircase{1}{7}{0}{black!10}{0}
\drawstaircase{4}{7}{0}{red!40}{0}
\foreach \n/\i/\txt in {
1/0/5,
2/0/5,
2/1/5,
3/1/5,
3/2/5,
4/1/1,
4/2/11,
4/3/10,
5/3/35,
5/4/35,
6/4/265,
6/5/265,
7/4/10,
7/5/495,
7/6/485
 }{
    \cell{\n}{\i}{\txt}
  }
\end{tikzpicture}
\caption{Homology of $(\C R)_{-1}$ for (a) $G=S_3$, $R$ transpositions, (b) $G=S_5$, $R$ transpositions, (c) $G=D_5$, $R$ reflections, (d) $G=D_7$, $R$ reflections, (e) $G=A_3$, $R$ one conjugacy class of 3-cycles and (f) $G=C_5\rtimes C_4$, $R=\{(x,1)\colon x\in C_5\}$.}\label{fig: Mobius computation figures}
\end{figure}
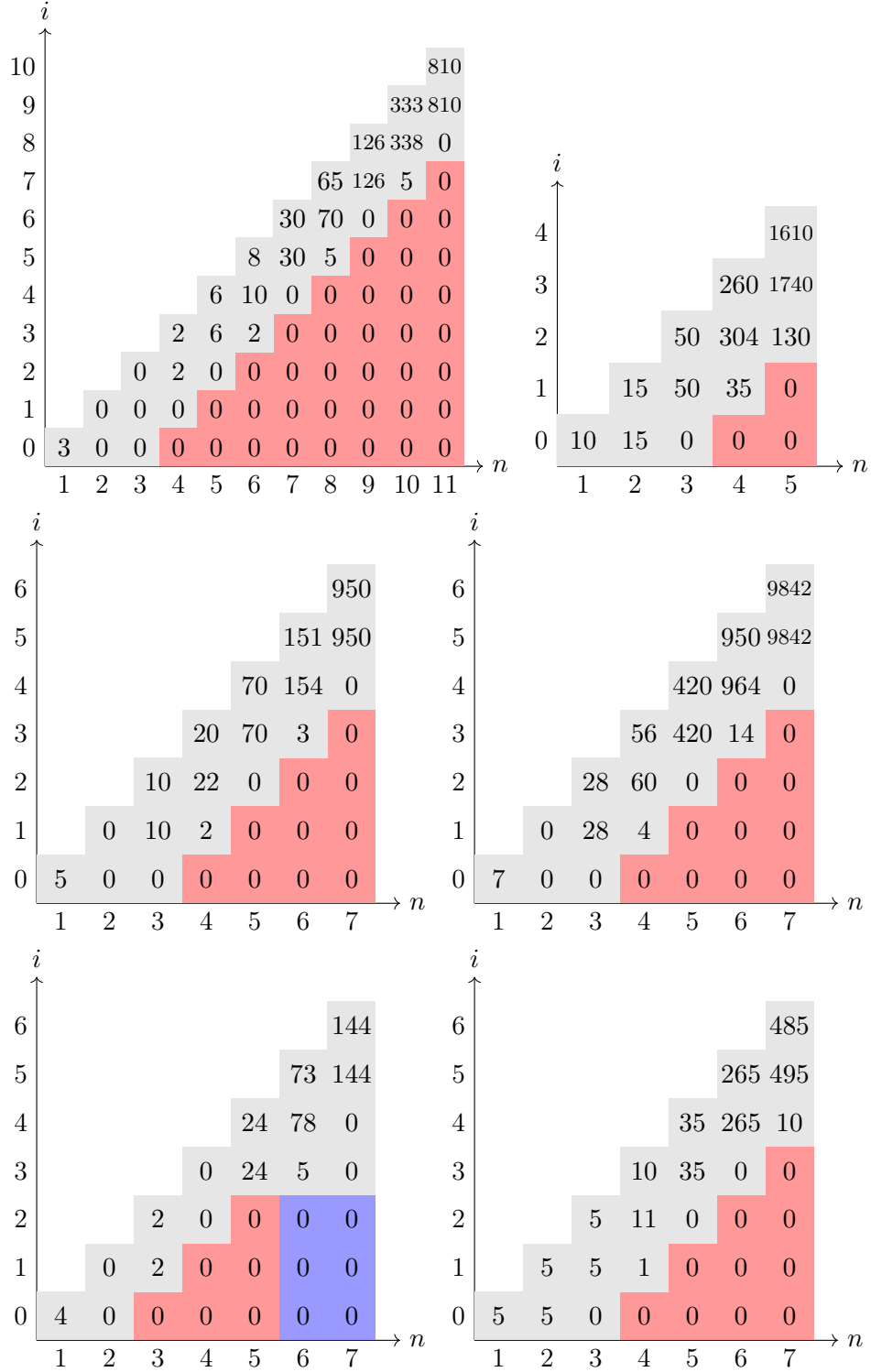
\subsection{Sign twist of $\C R$}\label{sec: individual neg twists} We now consider some examples of the braided vector spaces $(\C R)_{-1}$ where $R$ is a union of conjugacy classes in a group $G$, focusing on the cases in Figure \ref{fig: table mobius G,R}. Unlike the previous section which discusses twists of $\C_\wedge$, we do not have strong vanishing results coming from Proposition \ref{prop: additional conditions twist}. Indeed, as remarked earlier, since $\ord(-1)=2$, the divisibility condition $\ord(\zeta)\mid n(n-1)|G|$ in Proposition \ref{prop: additional conditions twist}(b) is trivially satisfied, so it does not give us any vanishing $n$. This is actually the truth, as we will see in the examples below. Because of this, we are only able to prove vanishing slopes by computing homology for small $n$. We present our computations in Figure \ref{fig: Mobius computation figures}. The example of $G=S_4$ and $R$ the conjugacy class of transpositions was already given in Figure \ref{fig: transpositions in S_4}. By applying Theorem \ref{thm: main extrapolating vanishing} to the appropriate vanishing staircases indicated, we get the desired slopes of vanishing in Figure \ref{fig: table mobius G,R}.
\subsection{Other braided vector spaces}\label{sec: intro individual bvs} Lastly, we consider other braided vector spaces which may not necessarily correspond to something arithmetic. These examples are usually connected to physics because braided vector spaces are solutions to the Yang-Baxter equation.

First, we look at Gaussian braided vector spaces $V$ of dimension $m$ as defined in \cite{Galindo_Rowell_2014}. Let $\zeta=\zeta_m$ if $m$ is odd and $\zeta=\zeta_{2m}$ if $m$ is even. Let $\{e_1,\ldots, e_m\}$ be a basis of $V$, and define the braiding by
$$R(e_i\otimes e_j)=\frac{1}{\sqrt m}\sum_{k=0}^{m-1}\zeta^{k^2+{k(j-i)}}e_{i+k}\otimes e_{j+k}.$$
As an example, we compute the homology for the three-dimensional Gaussian braided vector space in Figure \ref{fig: gaussian 3d}, and we from this we see a vanishing slope of $\frac 5 7$.
\begin{figure}[h]
    \centering
  \begin{tikzpicture}[x=1cm,y=1cm,scale=0.65]
\drawnaxis{8}
\drawiaxis{7}
\drawstaircase{1}{8}{0}{black!10}{0}
\drawstaircase{2}{6}{0}{red!40}{0}
\drawblock{7}{8}{0}{4}{blue!40}{0}
\foreach \n/\i/\txt in {
1/0/3,
7/5/81,
7/6/81,
8/5/81,
8/6/81
 }{
    \cell{\n}{\i}{\txt}
  }
\end{tikzpicture}
    \caption{Homology of $3$ dimensional Gaussian braided vector space.}
    \label{fig: gaussian 3d}
\end{figure}
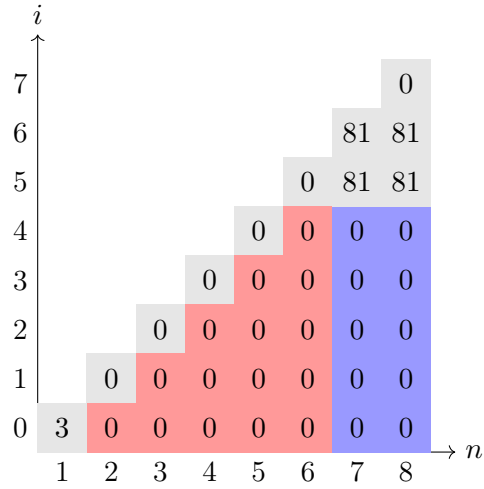

Next, we look at two dimensional braided vector spaces, which were classified completely in \cite{Hietarinta_1992}. The classification organizes the solutions according to the number of free parameters and the rank of the braiding. We choose to look at the solutions with no free parameters, since these are in some sense the most canonical. We also require them to be of full rank as in our case we require $R_{V,V}\colon V\otimes V\xrightarrow\sim V\otimes V$ to be an isomorphism. Then, there are three solutions, with braiding matrix written as follows with respect to the basis $e_1\otimes e_1,e_1\otimes e_2,e_2\otimes e_1,e_2\otimes e_2$:
$$R_1=\begin{pmatrix}
1 & & &1\\
& &-1& \\
&-1& & \\
& & &1
\end{pmatrix},\ R_2=\frac{1}{\sqrt 2}\begin{pmatrix}
1 & & &1\\
&1&1& \\
&-1&1& \\
-1& & &1
\end{pmatrix},\ R_3=\begin{pmatrix}
1 & & &\\
& 1&& \\
& & 1& \\
& & &1
\end{pmatrix}.$$

Note that the convention of the braiding matrix here differs from that in \cite{Hietarinta_1992} by swapping the second and third columns. We also normalize $R_2$ by dividing by $\sqrt 2$ so the eigenvalues have absolute value $1$. 

Let $V_i$ be the two-dimensional braided vector space with braiding $R_i$. The case of $V_3$ is not interesting as the braiding $R_3$ is trivial and $B_n$ acts trivially on $V_3^{\otimes n}$. In this case the homology can be deduced from the classical computation of the homology of configuration space and is simply $H_i(B_n,V_3^{\otimes n})=V_3^{\otimes n}$ for $0\le i\le \min(n-1,1)$ and $0$ otherwise. 

For $V_1$, we can consider all twists $(V_1)_\zeta$ by some $\zeta\in \C^\times$, and recall that this corresponds to multiplying $R_1$ by $\zeta$. It is not hard to see that if $1$ is not an eigenvalue of $\zeta R_1$, then $H_0(B_2,(V_1)_\zeta^{\otimes 2})=H_1(B_2,(V_1)_\zeta^{\otimes 2})=0$ and so the homology vanishes with slope at least $\frac 1 3$. We now consider the case when $1$ is an eigenvalue, which forces $\zeta=1$ or $\zeta=-1$. We compute the homology in Figure \ref{fig: V_1 computation}. We see that $V_1$ itself seems to exhibit homological stability, while the twist $(V_1)_{-1}$ vanishes with slope at least $\frac 6 {11}$.
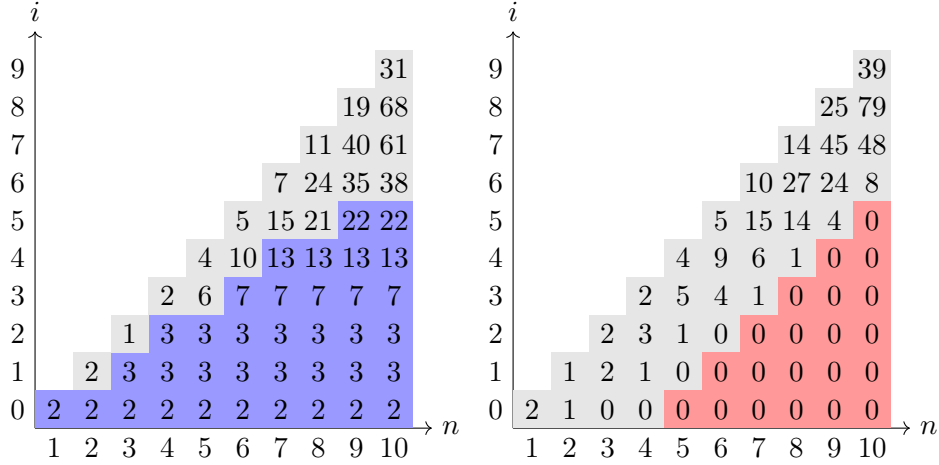
\begin{figure}[h]
    \centering
\begin{tikzpicture}[x=1cm,y=1cm,scale=0.5]
\drawnaxis{10}
\drawiaxis{9}
\drawstaircase{1}{10}{0}{black!10}{0}
\drawblock{1}{10}{0}{0}{blue!40}{2}
\drawblock{3}{10}{1}{1}{blue!40}{3}
\drawblock{4}{10}{2}{2}{blue!40}{3}
\drawblock{6}{10}{3}{3}{blue!40}{7}
\drawblock{7}{10}{4}{4}{blue!40}{13}
\drawblock{9}{10}{5}{5}{blue!40}{22}
\foreach \n/\i/\txt in {
2/1/2,
3/2/1,
4/3/2,
5/3/6,
5/4/4,
6/4/10,
6/5/5,
7/5/15,
7/6/7,
8/5/21,
8/6/24,
8/7/11,
9/6/35,
9/7/40,
9/8/19,
10/6/38,
10/7/61,
10/8/68,
10/9/31
  }{
    \cell{\n}{\i}{\txt}
  }
\end{tikzpicture}
\begin{tikzpicture}[x=1cm,y=1cm,scale=0.5]
\drawnaxis{10}
\drawiaxis{9}
\drawstaircase{1}{10}{0}{black!10}{0}
\drawstaircase{5}{10}{0}{red!40}{0}
\foreach \n/\i/\txt in {
1/0/2,
2/0/1,
2/1/1,
3/1/2,
3/2/2,
4/1/1,
4/2/3,
4/3/2,
5/2/1,
5/3/5,
5/4/4,
6/3/4,
6/4/9,
6/5/5,
7/3/1,
7/4/6,
7/5/15,
7/6/10,
8/4/1,
8/5/14,
8/6/27,
8/7/14,
9/5/4,
9/6/24,
9/7/45,
9/8/25,
10/6/8,
10/7/48,
10/8/79,
10/9/39
 }{
    \cell{\n}{\i}{\txt}
  }
\end{tikzpicture}
\caption{Homology of the two-dimensional braided vector space $V_1$ and $(V_1)_{-1}$.}\label{fig: V_1 computation}
\end{figure}

Finally, one can check that $V_2$ is actually isomorphic to the two-dimensional Gaussian braided vector space via a change of basis. We compute that the homology of $V_2$ itself vanishes for $2\le n\le 13$, and we are not sure whether the homology actually vanishes for all $n\ge 2$. Using the same reasoning as above, we see that all twists $(V_2)_\zeta$ vanish with slope at least $\frac 1 3$ unless $\zeta=\zeta_8$ or $\zeta=\bar\zeta_8$. Since these two are complex conjugates, they have the same homology, which we compute in Figure \ref{fig: V_2 computation}. We observe that there is homological vanishing of slope $\frac 1 4$ which seems to be sharp. Furthermore, all dimensions of homology are powers of $2$, and there appears to be an underlying pattern which also resembles the three-dimensional Gaussian braided vector space.
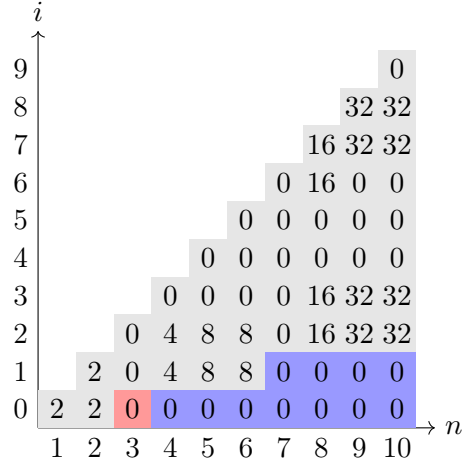
\begin{figure}[h]
    \centering
  \begin{tikzpicture}[x=1cm,y=1cm,scale=0.5]
\drawnaxis{10}
\drawiaxis{9}
\drawstaircase{1}{10}{0}{black!10}{0}
\drawstaircase{3}{3}{0}{red!40}{0}
\drawblock{4}{10}{0}{0}{blue!40}{0}
\drawblock{7}{10}{1}{1}{blue!40}{0}
\foreach \n/\i/\txt in {
1/0/2,
2/0/2,
2/1/2,
4/1/4,
4/2/4,
5/1/8,
5/2/8,
6/1/8,
6/2/8,
8/2/16,
8/3/16,
8/6/16,
8/7/16,
9/2/32,
9/3/32,
9/7/32,
9/8/32,
10/2/32,
10/3/32,
10/7/32,
10/8/32
 }{
    \cell{\n}{\i}{\txt}
  }
\end{tikzpicture}
    \caption{Homology of $(V_2)_{\zeta_8}$}
    \label{fig: V_2 computation}
\end{figure}
\subsection{Homological computations and algorithms}\label{sec: on homological computations}
The homology $H_i(B_n,V^{\otimes n})$ can be computed explicitly from the bar-complex as seen in Equation \eqref{eqn: homology to bar-complex}. Hence, in order to compute the dimensions of homology, it suffices to compute the rank of the differentials in the bar-complex. This is not difficult as one can easily write explicit matrices for the differentials from the definition of the shuffle product. Hence, most of the work comes from optimizing the computation of rank of these matrices so that one can improve the range of computation for $n$. 

Recall that if $V$ is a braided vector space of dimension $d$, the largest term in the bar-complex has dimension $D\approx \binom n {\lfloor n/2\rfloor} d^n = O((2d)^n)$, so the largest differentials have dimensions of this size. A simple Gaussian elimination algorithm would take $O(D^3)$ time and becomes infeasible quickly. 

The key improvement comes from the realization that the matrices corresponding to the differentials in the bar-complex are sparse. Let $z$ be the number of non-zero entries in each row, which is much smaller than $D$. A typical example in our case (say $V=\C R$ for $R$ the rack of transpositions in $S_3$ and $n=10$) could have $D\approx 10^{7}$ and $z\approx 100$. We use the block Wiedemann algorithm \cite{Wiedemann1986,KaltofenSaunders1991,Coppersmith1994} which is a probabilistic algorithm that computes the kernel vectors (and hence the rank) of a matrix over $\F_p$, which has complexity $O(D^2z)$, which for our applications saves a factor of $\approx 10^5$, although the larger constant overhead reduces this factor somewhat. To get a matrix over a finite field, we randomly chose a few large primes $p$ and reduce our matrix modulo $p$, checking that the rank in the same for each prime.

A smaller improvement to the algorithm is to note that the differential preserves certain invariants like the global monodromy for the case of racks from groups, allowing us to decompose the matrix for the differential into smaller blocks and running the algorithm on each of them. 

We believe that there could be a faster algorithm by exploiting the spectral sequences associated to filtrations similar to those in Section \ref{sec: proof of main theorem}. We attempted to implemented this but failed.

We prompted GPT-5.5 to implement the algorithm described above in C++. We did this in several steps, by first asking it to implement the simpler Gaussian elimination algorithm, then changing this to the block Wiedemann algorithm while keeping in mind complexity constraints, and implementing a multi-threaded version. We then asked it to optimize the code for speed as much as possible. We also did manual profiling of the code to see which parts took the most time and suggested changes to GPT-5.5. We verified correctness of the output in a variety of test cases at each step. The AI-generated code is available \href{https://drive.google.com/drive/folders/1AiI_iS3NKXp4XGrYsV0S7egW3xFw-9Pk?usp=sharing}{here}.
\section{Proof of arithmetic statements}\label{sec: arithmetic applications}
Finally, we link the homological results in the previous few sections to our arithmetic applications as stated in Section \ref{sec: intro patterson} and \ref{sec: intro character sums G-extensions}. 
\subsection{Möbius and character sums}
We first tackle the case of the Möbius function and character sums as it is slightly simpler. We will prove Theorem \ref{thm: character G extensions} and \ref{thm: multiple characters} and justify Figure \ref{fig: table mobius G,R}. We remark that this is rather similar to \cite[Section 6]{ES26} because we are in the same arithmetic situation.

\begin{proof}[Proof of Theorem \ref{thm: character G extensions}] Recall from the discussion in Section \ref{sec: arith sums from trace fns} that if we let $\mathcal L'_{G,R}=\pi'_*\bar \Q_l$ be pushforward of the constant sheaf from Hurwitz space $\Hur^{n_1,\ldots, n_k}_{G,R,\infty}$ to $\Conf^n$ as in Example \ref{eg: local system hurwitz space} and $\delta^*\mathcal L_\chi$ be the pullback of the Kummer sheaf via the discriminant map as in Example \ref{eg: local system character sheaves}, then for $f\in \Conf^n(\F_q)$ we have
$$\tr(\Frob_q, (\mathcal L'_{G,R}\otimes \delta^*\mathcal L_\chi)_{\bar f})=\#\{L\in \mathcal E^R_q(G;n_1,\ldots, n_k)\colon f_L=f\}\cdot \chi(\disc(f)).$$
Thus, by the Grothendieck-Lefschetz trace formula, we express the desired sum as
\begin{equation*}
\begin{split}
\sum_{L\in \mathcal E^R_q(G;n_1,\ldots, n_k)}\chi(\disc(f_L))&=\sum_{f\in \Conf^n(\F_q)}\tr(\Frob_q,(\mathcal L'_{G,R}\otimes \delta^*\mathcal L_\chi)_{\bar f}) \\
&=\sum_{i=0}^{2n}(-1)^i\tr(\Frob_q, H^i_c(\Conf^n_{\bar\F_q},\mathcal L'_{G,R}\otimes \delta^*\mathcal L_\chi)).
\end{split}
\end{equation*}

Now, recall from Example \ref{eg: local system hurwitz space} and \ref{eg: local system character sheaves} that the analytification of $\mathcal L'_{G,R}\otimes \delta^*\mathcal L_\chi$ over $\C$ is a direct summand of $V_\zeta ^{\otimes n}$ for the braided vector space $V=\C R$, where $\zeta$ is a root of unity satisfying $\ord(\zeta)=\ord(\chi)$. We use a comparison theorem from étale to singular cohomology \cite[Theorem 4.1.1]{ES26} to see that $H^i_c(\Conf^n_{\bar\F_q},\mathcal L_{G,R}'\otimes \delta^*\mathcal L_\chi)=0$ if we have $H_{2n-i}(B_n,(V_\zeta)^{*\otimes n})=H_{2n-i}(B_n,(V_{\bar\zeta})^{\otimes n})=0$ where we used the fact that the dual of $V$ is itself.

On the other hand, applying case (b) of Corollary \ref{cor: density for case (b), (c)}, we get that the homology $H_j(B_n,(V_{\bar\zeta})^{\otimes n})$ vanishes with slope $B=1-\frac{2^{\omega(o')}}{o'}$, i.e. there is some constant $C$ for which the homology is only nonzero when $j\ge Bn-2C$. Translating this to cohomological degree, $H^i_c(\Conf^n_{\bar\F_q},\mathcal L_{G,R}'\otimes \delta^*\mathcal L_\chi)$ is nonzero only if $i\le (2-B)n+2C$

The Weil bound \cite{Deligne_1980} tells us that the eigenvalue of $\Frob_q$ on $H^i_c(\Conf^n_{\bar\F_q},\mathcal L_{G,R}'\otimes \delta^*\mathcal L_\chi)$ is bounded above by $q^{i/2}$. Combining this with the bound for $i$ as well as the trivial Betti number bound $\sum_i \dim (H_i(B_n,V^{*\otimes n})\le 2^{n-1}(\dim V)^n$ from the bar-complex, we get the desired bound
$$\left|\sum_{L\in \mathcal E^R_q(G;n_1,\ldots, n_k)}\chi(\disc(f_L))\right|\le 2^{n-1}|R|^nq^{\left(\frac 1 2 +\frac{2^{\omega(o')-1}}{o'}\right)n+C}.$$
\end{proof}
Recall that the Möbius function is the special case when $\chi=\xi$ is the quadratic character, so the above proof tells us that a vanishing slope of $B$ for $(\C R)_{-1}$ leads to a $q$-exponent of the form $(1-\frac B 2)n+C$. Hence, the computations in Figure \ref{fig: transpositions in S_4} and \ref{fig: Mobius computation figures} justifies Figure \ref{fig: table mobius G,R}.

We now deal with the statement for multiple characters.
\begin{proof}[Proof of Theorem \ref{thm: multiple characters}] We now consider instead the pushforward $\mathcal L_{G,R}=\pi_*\bar\Q_l$ of the constant sheaf from $\Hur^{n_1,\ldots,n_k}_{G,R,\infty}$ to $\Conf^{n_1,\ldots, n_k}$ from Example \ref{eg: local system hurwitz space}, and the tensor product of pullbacks of Kummer sheaves $\mathcal L_{\vec\chi}$ in Equation \eqref{eqn: multiple kummer}. Then, for $x=(f_1,\ldots, f_k)\in \Conf^{n_1,\ldots, n_k}(\F_q)$, we have
\begin{equation*}
\begin{split}
&\tr(\Frob_q,(\mathcal L_{G,R}\otimes \mathcal L_{\vec\chi})_{\bar x})\\
=&\# \{L\in \mathcal E^R_q(G;n_1,\ldots, n_k)\colon f_{L,i}=f_i \ \forall 1\le i\le k\}\cdot \prod_{1\le i \le k} \chi_i(\disc(f_i))  \prod_{1\le i<j\le k} \chi_{ij}(\res(f_i,f_j))
\end{split}
\end{equation*}
so the Grothendieck-Lefschetz trace formula tells us that 
\begin{equation*}
\begin{split}
&\sum_{L\in \mathcal E^{R}_q(G;n_1,\ldots, n_k)} \prod_{1\le i\le k} \chi_{i} (\disc(f_{L,i}))\prod_{1\le i< j \le k}\chi_{ij}(\res(f_{L,i},f_{L,j}))\\
=& \sum_{i=0}^{2n}\tr(\Frob_q,H^i_c(\Conf^{n_1,\ldots, n_k}_{\bar\F_q},\mathcal L_{G,R}\otimes\mathcal L_{\vec\chi}))
\end{split}
\end{equation*}

Recall that the analytification of $\mathcal L_{G,R}$ as a $B_{n_1,\ldots, n_k}$-representation is a direct summand of $W_{n_1,\ldots, n_k}=(\C R_1)^{\otimes n_1}\otimes \cdots \otimes (\C R_k)^{\otimes n_k}$, and the analytification of $\mathcal L_{\vec\chi}$ is the one-dimensional representation $\C_{\vec\zeta}$ where a positive half-twist of adjacent strands of the same color $i$ acts by $\zeta_i$ and a positive full twist of two strands of different colors $i<j$ acts by $\zeta_{ij}^2$. The analogue of \cite[Theorem 4.1.1]{ES26} for multicolored configuration space $\Conf^{n_1,\ldots, n_k}$ holds as $\Conf^{n_1,\ldots, n_k}$ is a finite étale cover of $\Conf^n$ and \cite[Proposition 7.7]{EVW16} works for finite étale covers, so we have that $H^i_c(\Conf^{n_1,\ldots, n_k}_{\bar\F_q},\mathcal L_{G,R}\otimes\mathcal L_{\vec\chi})=0$ if $H_{2n-i}(B_{n_1,\ldots, n_k},W_{n_1,\ldots, n_k}\otimes \C_{\vec{\bar\zeta}})=0$, where we use the fact that the dual of $W_{n_1,\ldots, n_k}$ is itself.

We want to apply Shapiro's lemma, so we need to find the induced $B_n$-representation for $W_{n_1,\ldots, n_k}\otimes \C_{\vec{\bar\zeta}}$. The untwisted induced module $\Ind^{B_n}_{B_{n_1,\ldots, n_k}}W_{n_1,\ldots, n_k}$ is a direct summand of $V^{\otimes n}$ where $V=\C R$ corresponding to the tensor products which contain the factors $V_i=\C R_i$ exactly $n_i$ times. It is not hard to check that the twisted induced module $\Ind^{B_n}_{B_{n_1,\ldots, n_k}}W_{n_1,\ldots, n_k}\otimes \C_{\vec{\bar\zeta}}$ is exactly the direct summand of $V_{\vec{\bar\zeta}}^{\otimes n}$ defined in Section \ref{sec: multicolored case} in the braided monoidal abelian category $\mathcal V'$ twisted by $\vec{\bar\zeta}$ which contain the factors $V_i'$ exactly $n_i$ times. Here, we note that the action of the positive full twist by $\zeta_{ij}^2$ is spread over two half-twists $R_{V_i',V_j'}=\zeta_{ij}R_{V_i,V_j}$ and $R_{V_j',V_i'}=\zeta_{ij}R_{V_j,V_i}$. Thus, by Shapiro's lemma, we have that $H_{2n-i}(B_{n_1,\ldots, n_k},W_{n_1,\ldots, n_k}\otimes \C_{\bar\zeta})=0$ if $H_{2n-i}(B_n,V_{\vec{\bar\zeta}}^{\otimes n})=0$.

For any integer $N\ge 2$, Corollary \ref{cor: manin mumford} tells us that if $\vec{\zeta}$ (and hence $\vec{\bar\zeta}$) is not in a finite union of torsion cosets, then $V_{\vec \bar\zeta}$ is in $\mathcal F_V\setminus \mathcal F_{V,N}$ and vanishes in a $(N,N-2)$ staircase. Applying Theorem \ref{thm: main extrapolating vanishing}, the homology $H_j(B_n,(V_{\vec{\bar\zeta}})^{\otimes n})$ vanishes with slope $\frac{N-1}{N+1}$. More precisely, we keep track of the constant term by Equation \eqref{eqn: i in terms of n}, and we see that the largest $i=2n-j$ where cohomology $H^i_c(\Conf^{n_1,\ldots,n_k}_{\bar\F_q},\mathcal L_{G,R}\otimes\mathcal L_{\vec\chi})$ could be nonzero is 
\begin{equation}\label{eqn: i bound by n proof}
i=2n-(N-1)\left \lfloor \frac{n+1}{N+1}\right\rfloor + \max \left(\left\{\frac{n+1}{N+1}\right\}-2 ,0\right)\le \left(1+\frac{2}{N+1}\right)(n-1)+2 
\end{equation}
Hence, by applying the Weil bound and the trivial Betti number bound to the trace formula above, we have
$$\left|\sum_{L\in \mathcal E^{R}_q(G;n_1,\ldots, n_k)} \prod_{1\le i\le k} \chi_{i} (\disc(f_{L,i}))\prod_{1\le i< j \le k}\chi_{ij}(\res(f_{L,i},f_{L,j}))\right|\le 2^{n-1}|R|^nq^{\left(\frac 1 2+\frac 1 {N+1}\right)(n-1)+1}$$
which finishes the proof.
\end{proof}
\subsection{Bias in Gauss sums} We now prove Theorem \ref{thm: gauss sums prime powers} and \ref{thm: gauss sums general order} as well as justify Figure \ref{fig: table order}.
\begin{proof}[Proof of Theorem \ref{thm: gauss sums prime powers}] Recall from Equations \eqref{eqn: formula for 1_irr}, \eqref{eqn: formula for G_chi(f)} and Example \ref{eg: local system character sheaves} that for $f\in \Conf^n(\F_q)$,
$$\mathbf 1_{\text{irr}} (f) G_\chi(f)= C_{q,n,\chi} \cdot \frac 1 n \sum_{k=0}^{n-1} (-1)^k \tr(\Frob_q,(\mathcal L_{\wedge^k \text{std}_n}\otimes \delta^*\mathcal L_{\chi\cdot\xi})_{\bar f})$$
where $C_{q,n,\chi}$ is a constant depending on $q,n,\chi$ with absolute value $q^{n/2}$, so again by the Grothendieck-Leftschetz trace formula we have
$$\sum_{\substack{\deg(\pi)=n\\ \pi \text{ prime}}} \frac{G_\chi(\pi)}{q^{n/2}} = \frac{C_{q,n,\chi}}{q^{n/2}}\cdot\frac 1 n \sum_{k=0}^{n-1}\sum_{i=0}^{2n}(-1)^{i+k} \tr(\Frob_q,H^i_c(\Conf^n_{\bar\F_q},\mathcal L_{\wedge^k \text{std}_n}\otimes \delta^*\mathcal L_{\chi\cdot\xi}))$$
Now, by Example \ref{eg: local system rep S_n}, \ref{eg: local system character sheaves} and Equation \eqref{eqn: C_wedge decomp}, the analytification of the sheaves $\mathcal L_{\wedge^k \text{std}_n}\otimes \delta^*\mathcal L_{\chi\cdot\xi}$ over $\C$ is a direct summand of $V_{-\zeta}^{\otimes n}$ for $V=\C_\wedge$ where again $\zeta$ is a root of unity with $\ord(\zeta)=\ord(\chi)$. The comparison theorem \cite[Theorem 4.1.1]{ES26} then tells us that $H^i_c(\Conf^n_{\bar\F_q},\mathcal L_{\wedge^k \text{std}_n}\otimes \delta^*\mathcal L_{\chi\cdot\xi})=0$ if $H_{2n-i}(B_n,V_{-\zeta}^{\otimes n})=0$.

We assumed that the order $o=\ord(\chi)$ is a prime power or twice of an odd prime power. If $o=2$, there is nothing to prove as the statement follows even without homological vanishing. If $o=2^k$ for $k\ge 2$, then we also have $\ord(-\zeta)=\ord(\zeta)=2^k$. It is clear that $2^k$ does not divide $n(n-1)$ for $2\le n\le 2^k-1$. If $o=2p^k$ or $p^k$ for an odd prime $p$ with $k\ge 1$, then we have $\ord(-\zeta)=p^k$ or $2p^k$ respectively. Either way, we have that $p^k$ and $2p^k$ do not divide $n(n-1)$ for $2\le n \le p^k-1$. 

Hence, in all cases, applying Proposition \ref{prop: additional conditions twist}(c), we see that $H_i(B_n,V_{-\zeta}^{\otimes n})$ vanishes for all $i$ when $2\le n\le p^k-1$, which means it vanishes in an $(N,N-2)$ staircase when $N=p^k-1$. Then, as in Equation \eqref{eqn: i bound by n proof}, we get that the cohomology vanishes when $i\le (1+\frac{2}{p^k})(n-1)+1$. 

By Equation \eqref{eqn: C_wedge decomp}, we have that the sum of all the Betti numbers in the formula is at most half of the sum of the Betti numbers for the cohomology of $V_{-\zeta}^{\otimes n}$, so it is bounded above by $\frac 1 2 2^{n-1}\dim(V)=2^{2n-2}$. Combining the Weil bound with the Betti number bound gives
$$\left|\sum_{\substack{\deg(\pi)=n\\ \pi \text{ prime}}} \frac{G_\chi(\pi)}{q^{n/2}} \right|\le \frac{2^{2n-2}}{n}q^{(\frac 1 2 +\frac 1 M)(n-1)+1}$$
as desired.
\end{proof}
\begin{proof}[Proof of Theorem \ref{thm: gauss sums general order}]
We follow the previous proof almost exactly, except we now use Corollary \ref{cor: density for case (b), (c)}(c) to conclude that $H_i(B_n,V_{-\zeta}^{\otimes n})$ vanishes with slope $1-\frac{2^{\omega(\ord(-\zeta))}}{\ord(-\zeta)}=1-\frac{2^{\omega (o)}}{o}$, which gives the desired bound.
\end{proof}
Similarly, from the same argument above, one can prove the $q$-exponents in Figure \ref{fig: table order} using the computations in Section \ref{sec: individual C_wedge}.
\newpage
\printbibliography

@article{Kummer_1846, title={De residuis cubicis disquisitiones nonnullae analyticae.}, volume={1846}, DOI={10.1515/crll.1846.32.341}, number={32}, journal={Journal für die reine und angewandte Mathematik (Crelles Journal)}, author={Kummer, E.E.}, year={1846}, month={Jul}, pages={341–359}}

@article{Chinta_Friedberg_Hoffstein_2011, title={Double Dirichlet series and theta functions}, DOI={10.1007/978-1-4614-1219-9_6}, journal={Springer Proceedings in Mathematics}, author={Chinta, Gautam and Friedberg, Solomon and Hoffstein, Jeffrey}, year={2011}, month={Oct}, pages={149–170}}

@article{Church_Farb_2013, title={Representation theory and homological stability}, volume={245}, DOI={10.1016/j.aim.2013.06.016}, journal={Advances in Mathematics}, author={Church, Thomas and Farb, Benson}, year={2013}, month={Oct}, pages={250–314}}

@article{Gauss_1801, DOI={10.5479/sil.324926.39088000932822}, journal={Disquisitiones Arithmeticae}, author={Gauss, Carl Friedrich}, year={1801}}

@book{Davenport_2009, place={New York}, title={Multiplicative number theory}, publisher={Springer}, author={Davenport, Harold}, editor={Montgomery, Hugh L.Editor}, year={2009}}

@article{Heath-Brown_Patterson_1979, title={The distribution of Kummer sums at prime arguments.}, volume={1979}, DOI={10.1515/crll.1979.310.111}, number={310}, journal={Journal für die reine und angewandte Mathematik (Crelles Journal)}, author={Heath-Brown, D.R. and Patterson, S.J.}, year={1979}, month={Sep}, pages={111–130}}

@article{Patterson_1978, title={On the distribution of Kummer sums.}, volume={1978}, DOI={10.1515/crll.1978.303-304.126}, number={303–304}, journal={Journal für die reine und angewandte Mathematik (Crelles Journal)}, author={Patterson, S.J.}, year={1978}, month={Nov}, pages={126–143}}

@article{Heath-Brown_2000, title={Kummer’s conjecture for Cubic Gauss sums}, volume={120}, DOI={10.1007/s11856-000-1273-y}, number={1}, journal={Israel Journal of Mathematics}, author={Heath-Brown, D. R.}, year={2000}, month={Dec}, pages={97–124}}

@article{Dunn_Radziwill_2024, title={Bias in cubic Gauss sums: Patterson’s conjecture}, volume={200}, DOI={10.4007/annals.2024.200.3.3}, number={3}, journal={Annals of Mathematics}, author={Dunn, Alexander and Radziwiłł, Maksym}, year={2024}, month={Nov}}

@misc{david2026quarticgausssumsprimes,
      title={Quartic Gauss sums over primes and metaplectic theta functions}, 
      author={Chantal David and Alexander Dunn and Alia Hamieh and Hua Lin},
      year={2026},
      eprint={2306.11875},
      archivePrefix={arXiv},
      primaryClass={math.NT},
      url={https://arxiv.org/abs/2306.11875}, 
}

@article{Sawin_2024, title={General multiple Dirichlet series from perverse sheaves}, volume={262}, DOI={10.1016/j.jnt.2024.03.020}, journal={Journal of Number Theory}, author={Sawin, Will}, year={2024}, month={Sep}, pages={408–453}}

@misc{ES26,
      title={Averages of Arithmetic Functions over Conductors of Function Fields}, 
      author={Jordan Ellenberg and Mark Shusterman},
      year={2026},
      eprint={2601.01242},
      archivePrefix={arXiv},
      primaryClass={math.NT},
      url={https://arxiv.org/abs/2601.01242}, 
}

@article{EVW16, title={Homological stability for Hurwitz spaces and the Cohen-Lenstra conjecture over function fields}, volume={183}, DOI={10.4007/annals.2016.183.3.1}, number={3}, journal={Annals of Mathematics}, author={Ellenberg, Jordan and Venkatesh, Akshay and Westerland, Craig}, year={2016}, month={May}, pages={729–786}}

@misc{ETW17,
      title={Fox-Neuwirth-Fuks cells, quantum shuffle algebras, and Malle's conjecture for function fields}, 
      author={Jordan S. Ellenberg and TriThang Tran and Craig Westerland},
      year={2017},
      eprint={1701.04541},
      archivePrefix={arXiv},
      primaryClass={math.NT},
      url={https://arxiv.org/abs/1701.04541}, 
}

@misc{ellenberglandesman,
      title={Homological stability for generalized Hurwitz spaces and Selmer groups in quadratic twist families over function fields}, 
      author={Jordan S. Ellenberg and Aaron Landesman},
      year={2023},
      eprint={2310.16286},
      archivePrefix={arXiv},
      primaryClass={math.NT},
      url={https://arxiv.org/abs/2310.16286}, 
}

@misc{bergstrom2023hyperellipticcurvesscanningmap,
      title={Hyperelliptic curves, the scanning map, and moments of families of quadratic L-functions}, 
      author={Jonas Bergström and Adrian Diaconu and Dan Petersen and Craig Westerland},
      year={2023},
      eprint={2302.07664},
      archivePrefix={arXiv},
      primaryClass={math.NT},
      url={https://arxiv.org/abs/2302.07664}, 
}

@misc{landesman2025cohenlenstramomentsfunctionfields,
      title={The Cohen--Lenstra moments over function fields via the stable homology of non-splitting Hurwitz spaces}, 
      author={Aaron Landesman and Ishan Levy},
      year={2025},
      eprint={2410.22210},
      archivePrefix={arXiv},
      primaryClass={math.NT},
      url={https://arxiv.org/abs/2410.22210}, 
}

@misc{landesman2025homologicalstabilityhurwitzspaces,
      title={Homological stability for Hurwitz spaces and applications}, 
      author={Aaron Landesman and Ishan Levy},
      year={2025},
      eprint={2503.03861},
      archivePrefix={arXiv},
      primaryClass={math.AT},
      url={https://arxiv.org/abs/2503.03861}, 
}

@misc{landesman2025stablehomologyhurwitzmodules,
      title={The stable homology of Hurwitz modules and applications}, 
      author={Aaron Landesman and Ishan Levy},
      year={2025},
      eprint={2510.02068},
      archivePrefix={arXiv},
      primaryClass={math.NT},
      url={https://arxiv.org/abs/2510.02068}, 
}

@misc{ma2026weightfiltrationhurwitzspaces,
      title={Weight filtration of Hurwitz spaces and quantum shuffle algebras}, 
      author={Zhao Yu Ma},
      year={2026},
      eprint={2601.03871},
      archivePrefix={arXiv},
      primaryClass={math.NT},
      url={https://arxiv.org/abs/2601.03871}, 
}

@article{Sawin_2021_sqroot, title={Square-root cancellation for sums of factorization functions over short intervals in function fields}, volume={170}, DOI={10.1215/00127094-2020-0060}, number={5}, journal={Duke Mathematical Journal}, author={Sawin, Will}, year={2021}, month={Apr}}

@article{Sawin_2024_sqroot, title={Square-root cancellation for sums of factorization functions over squarefree progressions in $\mathbb{F}_q[t]$}, volume={233}, DOI={10.4310/acta.2024.v233.n2.a3}, number={2}, journal={Acta Mathematica}, author={Sawin, Will}, year={2024}, pages={285–418}}

@article{Zagier_1992, title={Realizability of a model in infinite statistics}, volume={147}, DOI={10.1007/bf02099535}, number={1}, journal={Communications in Mathematical Physics}, author={Zagier, Don}, year={1992}, month={Jun}, pages={199–210}}

@article{KS20, title={Shuffle algebras and perverse sheaves}, volume={16}, DOI={10.4310/pamq.2020.v16.n3.a9}, number={3}, journal={Pure and Applied Mathematics Quarterly}, author={Kapranov, Mikhail and Schechtman, Vadim}, year={2020}, pages={573–657}}

@article{Hanlon_Stanley_1998, title={A $q$-deformation of a trivial symmetric group action}, volume={350}, DOI={10.1090/s0002-9947-98-01880-7}, number={11}, journal={Transactions of the American Mathematical Society}, author={Hanlon, Phil and Stanley, Richard P.}, year={1998}, pages={4445–4459}}

@article{Deligne_1980, title={La conjecture de Weil: II}, volume={52}, DOI={10.1007/bf02684780}, journal={Publications Mathématiques de l’IHÉS}, author={Deligne, Pierre}, year={1980}, month={Dec}, pages={137–252}}

@article{Andruskiewitsch_Grana_2003, title={From racks to pointed Hopf algebras}, volume={178}, DOI={10.1016/s0001-8708(02)00071-3}, number={2}, journal={Advances in Mathematics}, author={Andruskiewitsch, Nicolás and Graña, Matı́as}, year={2003}, month={Sep}, pages={177–243}}

@article{Galindo_Rowell_2014, title={Braid representations from unitary braided vector spaces}, volume={55}, DOI={10.1063/1.4880196}, number={6}, journal={Journal of Mathematical Physics}, author={Galindo, César and Rowell, Eric C.}, year={2014}, month={Jun}}

@article{Hietarinta_1992, title={All solutions to the constant quantum Yang-Baxter equation in two dimensions}, volume={165}, DOI={10.1016/0375-9601(92)90044-m}, number={3}, journal={Physics Letters A}, author={Hietarinta, Jarmo}, year={1992}, month={May}, pages={245–251}}

@book{HS20, place={Providence, RI}, title={Hopf algebras and Root Systems}, publisher={American Mathematical Society}, author={Heckenberger, István and Schneider, Hans-Jürgen}, year={2020}}

@article{Garside_1969, title={The Braid Group and other groups}, volume={20}, DOI={10.1093/qmath/20.1.235}, number={1}, journal={The Quarterly Journal of Mathematics}, author={Garside, F. A.}, year={1969}, pages={235–254}}

@book{Brown_1982, title={Cohomology of groups}, publisher={Springer}, author={Brown, Kenneth S.}, year={1982}}

@book{Bombieri_Gubler_2009, place={Cambridge}, title={Heights in Diophantine geometry}, publisher={Cambridge University Press}, author={Bombieri, Enrico and Gubler, Walter}, year={2009}}

@article{Liu_Wood_Zureick-Brown_2024, title={A predicted distribution for Galois groups of maximal unramified extensions}, volume={237}, DOI={10.1007/s00222-024-01257-1}, number={1}, journal={Inventiones mathematicae}, author={Liu, Yuan and Wood, Melanie Matchett and Zureick-Brown, David}, year={2024}, month={Apr}, pages={49–116}}

@article{Wiedemann1986,
  author  = {Wiedemann, Douglas H.},
  title   = {Solving Sparse Linear Equations Over Finite Fields},
  journal = {IEEE Transactions on Information Theory},
  volume  = {32},
  number  = {1},
  pages   = {54--62},
  year    = {1986},
  doi     = {10.1109/TIT.1986.1057137}
}

@article{Coppersmith1994,
  author  = {Coppersmith, Don},
  title   = {Solving Homogeneous Linear Equations over {GF}(2) via Block Wiedemann Algorithm},
  journal = {Mathematics of Computation},
  volume  = {62},
  number  = {205},
  pages   = {333--350},
  year    = {1994},
  doi     = {10.2307/2153413}
}

@inproceedings{KaltofenSaunders1991,
  author    = {Kaltofen, Erich and Saunders, Barry D.},
  title     = {On Wiedemann's Method of Solving Sparse Linear Systems},
  booktitle = {Applied Algebra, Algebraic Algorithms and Error-Correcting Codes (AAECC-9)},
  series    = {Lecture Notes in Computer Science},
  volume    = {539},
  pages     = {29--38},
  publisher = {Springer},
  year      = {1991}
}

@article{Miller_Tosteson_2021, title={Representation stability for pure braid group Milnor fibers}, volume={374}, DOI={10.1090/tran/8466}, number={11}, journal={Transactions of the American Mathematical Society}, author={Miller, Jeremy and Tosteson, Philip}, year={2021}, month={Jul}, pages={8177–8199}}
\end{document}